\documentclass{amsart}
\usepackage{amsmath,amssymb,amsthm,mathrsfs}

\theoremstyle{plain}
\newtheorem{theorem}{Theorem}[section]
\newtheorem{lemma}[theorem]{Lemma}

\newtheorem{proposition}[theorem]{Proposition}

\theoremstyle{definition}
\newtheorem{remark}[theorem]{Remark}

\numberwithin{equation}{section}

\newcommand{\BC}{{\mathbb C}}\newcommand{\BD}{{\mathbb D}}

\newcommand{\BR}{{\mathbb R}}
\newcommand{\BT}{{\mathbb T}}

\newcommand{\BZ}{{\mathbb Z}}
\newcommand{\cA}{{\mathcal A}}\newcommand{\cB}{{\mathcal B}}
\newcommand{\cC}{{\mathcal C}}\newcommand{\cD}{{\mathcal D}}

\newcommand{\cL}{{\mathcal L}}

\newcommand{\cS}{{\mathcal S}}\newcommand{\cT}{{\mathcal T}}
\newcommand{\cU}{{\mathcal U}}
\newcommand{\cW}{{\mathcal W}}\newcommand{\cX}{{\mathcal X}}
\newcommand{\cY}{{\mathcal Y}}
\newcommand{\bA}{{\mathbf A}}\newcommand{\bB}{{\mathbf B}}
\newcommand{\bC}{{\mathbf C}}\newcommand{\bD}{{\mathbf D}}

\newcommand{\bH}{{\mathbf H}}

\newcommand{\bR}{{\mathbf R}}
\newcommand{\bS}{{\mathbf S}}
\newcommand{\bU}{{\mathbf U}}
\newcommand{\bW}{{\mathbf W}}

\newcommand{\fH}{{\mathfrak H}}

\newcommand{\fL}{{\mathfrak L}}

\newcommand{\fT}{{\mathfrak T}}

\newcommand{\wtilA}{\widetilde{A}}\newcommand{\wtilB}{\widetilde{B}}
\newcommand{\wtilC}{\widetilde{C}}\newcommand{\wtilD}{\widetilde{D}}

\newcommand{\ep}{\epsilon}

\newcommand{\si}{\sigma}\newcommand{\Si}{\Sigma}

\newcommand{\im}{\textup{Im\,}}

\newcommand{\kr}{\textup{Ker\,}}
\newcommand{\diag}{\textup{diag\,}}
\newcommand{\spec}{r_\textup{spec}}

\newcommand{\mat}[2]{\ensuremath{\left[\begin{array}{#1}#2\end{array}\right]}}

\newcommand{\sbm}[1]{\left[\begin{smallmatrix}#1\end{smallmatrix}\right]}
\newcommand{\ov}[1]{{\overline{#1}}}

\newcommand{\inn}[2]{\ensuremath{\left\langle #1,#2 \right\rangle}}
\newcommand{\tu}[1]{\textup{#1}}
\newcommand{\wtil}[1]{\widetilde{#1}}
\newcommand{\what}[1]{\widehat{#1}}
\newcommand{\ands}{\quad\mbox{and}\quad}

\newcommand{\ons}{\mbox{ on }}
\newcommand{\ora}[1]{\overrightarrow{#1}}

\newcommand{\bu}{{\mathbf u}}
\newcommand{\bv}{{\mathbf v}}
\newcommand{\bx}{{\mathbf x}}
\newcommand{\by}{{\mathbf y}}

\newcommand{\bSigma}{{\boldsymbol \Sigma}}

\newcommand{\uS}{\underline{S}}

\newcommand{\frakH}{{\mathfrak H}}
\newcommand{\frakL}{{\mathfrak L}}
\newcommand{\frakT}{{\mathfrak T}}

\newcommand{\newU}{\overset{\rightarrow}{\mathcal U}}
\newcommand{\newX}{\overset{\rightarrow}{\mathcal X}}
\newcommand{\newY}{\overset{\rightarrow}{\mathcal Y}}
\newcommand{\newu}{\overset{\rightarrow}{\mathbf u}}
\newcommand{\newx}{\overset{\rightarrow}{\mathbf x}}
\newcommand{\newy}{\overset{\rightarrow}{\mathbf y}}

\begin{document}

\title[Infinite Dimensional Bounded Real Lemma III]{Standard
versus strict Bounded
Real Lemma with infinite-dimensional state space III:\\ The dichotomous and bicausal cases}

\author[J.A. Ball]{J.A. Ball}
\address{J.A. Ball, Department of Mathematics, Virginia Tech,
Blacksburg, VA 24061-0123, USA}
\email{joball@math.vt.edu}

\author[G.J. Groenewald]{G.J. Groenewald}
\address{G.J. Groenewald, Department of Mathematics, Unit for BMI,
North-West University, Potchefstroom 2531, South Africa}
\email{Gilbert.Groenewald@nwu.ac.za}

\author[S. ter Horst]{S. ter Horst}
\address{S. ter Horst, Department of Mathematics, Unit for BMI,
North-West University, Potchefstroom 2531, South Africa}
\email{Sanne.TerHorst@nwu.ac.za}

\thanks{This work is based on the research supported in part by the National
Research Foundation of South Africa (Grant Numbers 93039, 90670, and 93406).}

\begin{abstract}
This is the third installment in a series of papers concerning the Bounded Real Lemma for infinite-dimensional discrete-time linear
input/state/output systems. In this setting, under appropriate conditions, the lemma characterizes when the transfer
function associated with the system has contractive values on the unit circle, expressed in terms of a Linear Matrix Inequality,
often referred to as the Kalman-Yakubovich-Popov (KYP) inequality. Whereas the first two installments focussed on causal systems
with the transfer functions extending to an analytic function on the disk, in the present paper the system is still causal but
the state operator is allowed to have nontrivial  dichotomy (the unit circle is not contained in its spectrum), implying that the transfer
function is analytic in a neighborhood of zero  and on a neighborhood of the unit circle rather than on the unit disk.
More generally, we consider bicausal systems, for which the transfer function need not be analytic in a neighborhood of zero.
For both types of systems, by a variation on  Willems' storage-function approach, we prove variations on the standard and
strict Bounded Real Lemma. We also specialize the results to nonstationary discrete-time systems with a dichotomy, thereby
recovering a Bounded Real Lemma due to Ben-Artzi--Gohberg-Kaashoek for such systems.
\end{abstract}



\maketitle

%%%%%%%%%%%%%%%%%%%%%%%%%%%%%%%%%%%%%%%%%%%%%%%%%%%%%%%%%%%%%%%%%%%%%%%%%%%%%%%%%%%%%%%%%%%%%%%% Add material
%%%%%%%%%%%%%%%%%%%%%%%%%%%%%%%%%%%%%%%%%%%%%%%%%%%%%%%%%%%%%%%%%%%%%%%%%%%%%%%%%%%%%%%%%%%%%%%%

\section{Introduction}\label{S:intro}

This is the third installment in a series of papers on the bounded real lemma for infinite-dimensional discrete-time linear systems
and the related Kalman-Yakubovich-Popov (KYP) inequality. We consider discrete-time  input-state-output  linear systems determined
by the following equations
\begin{equation}\label{dtsystem}
\Si:=\left\{
\begin{array}{ccc}
\bx(n+1)&=&A \bx(n)+B \bu(n),\\
\by(n)&=&C \bx(n)+D \bu(n),
\end{array}
\right. \qquad (n\in\BZ)
\end{equation}
where $A:\cX\to\cX$, $B:\cU\to\cX$, $C:\cX\to\cY$ and $D:\cU\to\cY$
are bounded linear Hilbert space operators i.e., $\cX$, $\cU$ and
$\cY$ are Hilbert spaces and the {\em system matrix} associated with
$\Si$ takes the form
\begin{equation}\label{sysmat}
M=\mat{cc}{A&B\\ C& D}:\mat{cc}{\cX\\ \cU}\to\mat{c}{\cX\\ \cY}.
\end{equation}
Associated with the system $\Si$ is the transfer function
\begin{equation}  \label{transfunct}
    F_{\Sigma}(z) = D + z C (I - zA)^{-1} B
\end{equation}
which necessarily defines an analytic function on some neighborhood
of the origin in the complex plane with values in the space $\cL(\cU,
\cY)$ of bounded linear operators from $\cU$ to $\cY$.  The Bounded
Real Lemma is concerned with the question of characterizing (in terms
of $A,B,C,D$) when $F_{\Sigma}$ has analytic continuation to the
whole unit disk ${\mathbb D}$ such that the supremum norm of $F$ over
the unit disk $\| F \|_{\infty,\BD} : =  \sup\{ \| F(z) \| \colon z
\in {\mathbb D}\}$ satisfies either (i) $\| F_{\Sigma}\|_{\infty,\BD}
\le 1$ ({\em standard version}), or  (ii) $\| F_{\Sigma}\|_{\infty,\BD}
< 1$ ({\em strict version}).

We first note the following terminology
which we shall use.  Given a selfadjoint operator $H$ on a
Hilbert space $\cX$, we say
that
\begin{enumerate}
    \item[(i)] $H$ is {\em strictly positive-definite} ($H \succ 0$) if there
    is a $\delta > 0$ so that $\langle H x, x
\rangle \ge \delta \| x \|^{2} $ for all $x \in \cX$.
\item[(ii)] $H$ is {\em positive-definite} if $\langle
Hx, x \rangle > 0$ for all $0\neq x \in \cX$.
\item[(iii)] $H$ is {\em positive-semidefinite} ($H \succeq 0$) if $\langle H x, x
\rangle \ge 0$ for all $x \in \cX$.
\end{enumerate}
Given two selfadjoint operators $H, K$ on $\cX$, we write $H \succ K$
or $K \prec H$ if $K - H \succ 0$ and similarly for $H \succeq K$ or
$K \preceq H$.
Note that if $\cX$ is finite-dimensional, then {\em strictly
positive-definite} and {\em positive-definite} are equivalent.
Then the standard and strict Bounded Real Lemmas for the finite-dimensional setting
(where $\cX$, $\cU$, $\cY$ are all finite-dimensional and one can view $A,B,C,D$ as finite
matrices) are as follows.

\begin{theorem} \label{T:finiteBRL} Suppose that we are given  $\cX$,
    $\cU$, $\cY$, $M = \sbm{ A & B \\ C & D }$, and $F_{\Sigma}$ as
    in \eqref{dtsystem}, \eqref{sysmat}, \eqref{transfunct}, with
    $\cX$, $\cU$, $\cY$ all finite-dimensional Hilbert spaces.  Then:
\begin{enumerate}
    \item {\rm \textbf{Standard Bounded Real Lemma} (see \cite{AV}): } Assume that $(A,B)$ is {\em controllable}
    (i.e.\ $\operatorname{span}_{k \ge 0} \{\im A^{k}B\} =
    \cX$) and $(C,A)$ is {\em observable} (i.e.\ $ \bigcap_{k \ge
    0} \ker C A^{k} = \{0\}$).  Then $\| F_{\Sigma} \|_{\infty,\BD}
    \le 1$ if and only if there exists a positive-definite
    matrix $H$ satisfying the Kalman-Yakubovich-Popov (KYP) inequality:
\begin{equation}\label{KYP}
\mat{cc}{A&B\\C&D}^*\mat{cc}{H&0\\0&I_{\cY}}\mat{cc}{A&B\\C&D}\preceq
\mat{cc}{H&0\\0&I_{\cU}}.
\end{equation}

\item  \rm{\textbf{Strict Bounded Real Lemma} (see \cite{PAJ}):} Assume that all eigenvalues of $A$
are in the unit disk ${\mathbb D}$.  Then $\| F_{\Sigma} \|_{\infty,\BD} < 1$ if and only
if there exists a positive-definite matrix $H$ so that
\begin{equation}\label{KYPstrict}
\mat{cc}{A&B\\C&D}^*\mat{cc}{H&0\\0&I_{\cY}}\mat{cc}{A&B\\C&D}\prec
\mat{cc}{H&0\\0&I_{\cU}}.
\end{equation}
 \end{enumerate}
\end{theorem}

Infinite-dimensional versions of the standard Bounded Real Lemma have
been studied by Arov-Kaashoek-Pik \cite{AKP05} and the authors
\cite{KYP1, KYP2}, while infinite-dimensional versions of the strict
Bounded Real Lemma have been analyzed by Yakubovich \cite{Yak74,
Yak75}, Opmeer-Staffans \cite{OpmeerStaffans2010} and the authors
\cite{KYP1, KYP2}.

In this paper we wish to study the following variation of the Bounded
Real Lemma, which we shall call the {\em dichotomous Bounded Real
Lemma}.  Given the system with system matrix $M = \sbm{ A & B \\ C &
D }$ and associated transfer function $F_{\Si}$ as in
\eqref{dtsystem}, \eqref{sysmat}, \eqref{transfunct}, we now assume
that the operator $A$ {\em admits dichotomy}, i.e., we assume that
{\em $A$ has no spectrum on the unit circle ${\mathbb T}$.}  Under
this assumption it follows that the transfer function $F_{\Si}$ in
\eqref{transfunct} can be viewed as an analytic $\cL(\cU, \cY)$-valued
function on a neighborhood of the unit circle ${\mathbb T}$.  The
\textbf{dichotomous Bounded Real Lemma} is concerned with the
question of characterizing in terms of $A,B,C,D$ when it is the case
that $\| F_{\Si} \|_{\infty, {\mathbb T}} : = \sup \{ \| F(z) \| \colon z \in
{\mathbb T}\}$ satisfies either $\| F_{\Sigma}\|_{\infty, \mathbb T} \le 1$
(standard version) or (ii) $\| F_{\Sigma} \|_{\infty, {\mathbb T}} < 1$ (strict
version).  For the finite-dimensional case we have the following
result.

\begin{theorem}  \label{T:finite-dichotBRL} Suppose that we are given  $\cX$,
    $\cU$, $\cY$ and $M = \sbm{ A & B \\ C & D }$ and $F_{\Sigma}$ as
    in \eqref{dtsystem}, \eqref{sysmat}, \eqref{transfunct}, with
    $\cX$, $\cU$, $\cY$ all finite-dimensional Hilbert spaces and with
    $A$ having no eigenvalues on the unit circle ${\mathbb T}$. Then:
    \begin{enumerate}
\item {\rm \textbf{Finite-dimensional standard dichotomous Bounded Real Lemma:}} Assume that
$\Sigma$ is minimal  ($(A,B)$ is controllable and $(C,A)$ is
observable).  Then the inequality $\| F_{\Si} \|_{\infty,\BT} \le 1$ holds if and only if there exists an
invertible selfadjoint matrix $H$ which satisfies the KYP-inequality
\eqref{KYP}.  Moreover, the dimension of the spectral subspace of $A$
over the unit disk is equal to the number of positive eigenvalues
(counting multiplicities) of $H$ and the dimension of the spectral subspace of $A$
over the exterior of the closed unit disk is equal to the number of
negative eigenvalues (counting multiplicities) of $H$.

\item {\rm \textbf{Finite-dimensional strict dichotomous Bounded Real Lemma:}} The
strict inequality $\|F_{\Si} \|_{\infty,\BT} < 1$ holds if and only if there exists an
invertible selfadjoint matrix $H$ which satisfies the strict KYP-inequality
\eqref{KYPstrict}.  Moreover, the inertia of $A$ (the dimensions of
the spectral subspace of $A$ for the disk and for the exterior of the closed
unit disk) is related to the inertia of $H$ (dimension of negative
and positive eigenspaces) as in the standard dichotomous Bounded Real
Lemma (item (1) above).
\end{enumerate}
\end{theorem}

We note that Theorem \ref{T:finite-dichotBRL} (2) appears as Corollary
1.2 in \cite{BAGK95} as a corollary of more general considerations
concerning input-output operators for nonstationary linear systems
with an indefinite metric; to make the connection between the result
there and the strict KYP-inequality \eqref{KYPstrict}, one should
observe that a standard Schur-complement computation converts the
strict inequality \eqref{KYPstrict} to the pair of strict inequalities
\begin{align*}
    & I - B^{*}H B - D^{*} D \succ 0, \\
    & H - A^{*}H A - C^{*}C - (A^{*}H B + C^{*}D)  Z^{-1}  (B^{*} H A + D^{*} C) \succ 0 \\
    & \quad \quad \text{where } Z= I - B^{*}H B - D^{*}D.
\end{align*}
We have not located an explicit statement of Theorem
\ref{T:finite-dichotBRL} (1) in the literature;  this will be a
corollary of the infinite-dimensional standard dichotomous Bounded
Real Lemma which we present in this paper (Theorem
\ref{T:dichotBRL} below).

Note that if $F = F_{\Sigma}$ is a transfer function of the form
\eqref{transfunct}, then necessarily $F$ is analytic at the origin.
One approach to remove this restriction is to designate some other
point $z_{0}$ where $F$ is to be analytic and adapt the formula
\eqref{transfunct} to a realization ``centered at $z_{0}$'' (see
\cite[page 141]{BGR} for details): e.g.\ for the case $z_{0} =
\infty$, one can use $F(z) = D + C (z I - A)^{-1} B$.  To get a
single chart to handle an arbitrary location of poles, one can use
the bicausal realizations used in \cite{BCR} (see \cite{BR} for the
nonrational operator-valued case); for the setting here,
where we are interested in a rational matrix functions analytic on a
neighborhood of the unit circle ${\mathbb T}$,  we suppose that
\begin{equation}   \label{M+-}
  M_{+} = \begin{bmatrix} \widetilde A_{+} \!&\! \widetilde B_{+} \\ \widetilde C_{+} \!&\! \widetilde D
\end{bmatrix} \!\colon\!\! \begin{bmatrix} \cX_{+} \\ \cU \end{bmatrix} \to
\begin{bmatrix} \cX_{+} \\ \cY \end{bmatrix}, \
M_{-} = \begin{bmatrix} \widetilde A_{-} \!&\!  \widetilde B_{-} \\
\widetilde C_{-} \!&\! 0 \end{bmatrix} \!\colon\!\! \begin{bmatrix} \cX_{-} \\
\cU \end{bmatrix} \to \begin{bmatrix} \cX_{-} \\ \cY \end{bmatrix}
\end{equation}
are two system matrices with spectrum of $A_+$, $\sigma(A_{+})$, and of $A_-$, $\sigma(\widetilde
A_{-})$, contained in the unit disk ${\mathbb D}$ and that $F(z)$ is
given by
\begin{equation}  \label{bicausal-transfunct}
    F(z) = \widetilde D + z \widetilde C_{+} (I - z \widetilde A_{+})^{-1} \widetilde B_{+} + \widetilde C_{-} (I
    - z^{-1} \widetilde A_{-})^{-1} \widetilde B_{-}.
\end{equation}
We shall give an interpretation of \eqref{bicausal-transfunct} as the
transfer function of a bicausal exponentially stable system in Section
\ref{S:bicausal} below.  In any case we can now pose the question
for the rational case where all spaces $\cX_{\pm}$, $\cU$, $\cY$ in
\eqref{M+-} are finite-dimensional:  {\em characterize in terms of
$M_{+}$ and $M_{-}$ when it is the case that $\| F \|_{\infty, {\mathbb T}}
\le 1$ (standard case) or $\| F \|_{\infty, {\mathbb T}} < 1$ (strict case).}
To describe the result we need to introduce the {\em bicausal
KYP-inequality} to be satisfied by a selfadjoint operator
\begin{equation}\label{Hblock}
H = \mat{cc}{H_{-} & H_{0} \\ H_{0}^{*} & H_{+} } \ons \cX =
\mat{c}{ \cX_{-} \\ \cX_{+} }
\end{equation}
given by
\begin{align}
 &  \sbm{ I & 0 & \widetilde A_-^* \widetilde C_{-}^{*} \\
   0 & \widetilde A_+^* & \widetilde C_+^*  \\
 0 & \widetilde  B_{+}^{*}   & \widetilde B_-^* \widetilde C_-^* + \widetilde D^{*}}
\sbm{H_{-} & H_{0} & 0 \\ H_{0}^{*} & H_{+} & 0 \\ 0 & 0 & I }
 \sbm{I & 0 &  0 \\
 0 & \widetilde A_+   & \widetilde B_+  \\
 \widetilde C_{-} \widetilde A_-  & \widetilde C_{+} & \widetilde C_- \widetilde B_- + \widetilde D}  \notag \\
 & \qquad \qquad\qquad  \preceq
\sbm{\widetilde A_-^*  & 0 & 0 \\
0 & I & 0 \\
\widetilde B_{-}^{*} & 0 & I}
\sbm{ H_{-} & H_{0} & 0 \\ H_{0}^{*} & H_{+} & 0 \\ 0 & 0 & I }
 \sbm{  \widetilde A_- & 0 & \widetilde B_- \\  0 & I &  0  \\ 0 & 0 & I }
    \label{KYP-bicausal}
  \end{align}
  as well as the {\em strict bicausal KYP-inequality}:  for some $\epsilon > 0$ we have
  \begin{align}
 &  \sbm{ I & 0 & \widetilde A_-^* \widetilde C_{-}^{*} \\
   0 & \widetilde A_+^* & \widetilde C_+^*  \\
 0 & \widetilde  B_{+}^{*}   & \widetilde B_-^* \widetilde C_-^* + \widehat D^{*}}
\sbm{H_{-} & H_{0} & 0 \\ H_{0}^{*} & H_{+} & 0 \\ 0 & 0 & I }
 \sbm{I & 0 &  0 \\
 0 & \widetilde A_+ & \widetilde B_+  \\
 \widetilde C_{-} \widetilde A_-  & \widetilde C_{-} & \widetilde C_- \widetilde B_- + \widetilde D}
 + \epsilon^2 \sbm{  \widetilde A_-^* \widetilde A_- & 0 & \widetilde A_-^* \widetilde B_- \\
 0 & I & 0 \\ \widetilde B_-^* \widetilde A_- & 0 & \widetilde B_-^* \widetilde B_- + I }  \notag \\
 & \qquad \qquad\qquad  \preceq
\sbm{\widetilde A_-^*  & 0 & 0 \\
0 & I & 0 \\
\widetilde B_{-}^{*} & 0 & I}
\sbm{ H_{-} & H_{0} & 0 \\ H_{0}^{*} & H_{+} & 0 \\ 0 & 0 & I }
 \sbm{ \widetilde A_- & 0 & \widetilde B_-  \\ 0 & I & 0 \\ 0 & 0 & I }.
    \label{KYP-bicausal-strict}
  \end{align}
    We say that the system matrix-pair $(M_{+}, M_{-})$ is
  {\em controllable} if both $(\widetilde A_{+}, \widetilde B_{+})$ and $(\widetilde A_{-},
  \widetilde A_- \widetilde B_{-})$ are controllable, and that $(M_{+}, M_{-})$ is {\em observable} if
  both $(\widetilde C_{+}, \widetilde A_{+})$ and $(\widetilde C_{-} \widetilde A_{-},
  \widetilde A_{-})$ are observable.
  We then have the following result.

\begin{theorem}  \label{T:finite-bicausalBRL}
Suppose that we are given $\cX_{+}$, $\cX_{-}$, $\cU$, $\cY$, $M_{+}$, $M_{-}$ as in
\eqref{M+-} with $\cX_{+}$, $\cX_{-}$, $\cU$, $\cY$ finite-dimensional
Hilbert spaces and both $A_{+}$ and $\widetilde A_{-}$ having spectrum inside
the unit disk ${\mathbb D}$. Further, suppose that $F$ is the rational
matrix function with no poles on the unit circle ${\mathbb T}$ given
by \eqref{bicausal-transfunct}.  Then:
\begin{enumerate}
    \item {\rm \textbf{Finite-dimensional standard bicausal Bounded
    Real Lemma:}}  Assume that $(M_{+}, M_{-})$ is controllable and
    observable.  Then we have $\| F \|_{\infty, {\mathbb T}} \le 1$ if and only there
    exists an invertible selfadjoint solution $H$ as in \eqref{Hblock} of  the bicausal KYP-inequality
    \eqref{KYP-bicausal}.  Moreover $H_{+} \succ 0$ and $H_{-} \prec
    0$.

    \item  {\rm \textbf{Finite-dimensional strict bicausal Bounded Real Lemma:}}
The strict inequality $\| F \|_{\infty, {\mathbb T}} < 1$ holds if and only if there exists an invertible selfadjoint
solution $H$ as in \eqref{Hblock} of the strict bicausal
KYP-inequality \eqref{KYP-bicausal-strict}.  Moreover, in this case
$H_{+} \succ 0$ and $H_{-} \prec 0$.
\end{enumerate}
\end{theorem}

We have not located an explicit statement of these results in the
literature; they also are corollaries of the infinite-dimensional
results which we develop in this paper (Theorem
\ref{T:bicausalBRL} below).

  The goal of this paper is to explore infinite-dimensional analogues
of Theorems \ref{T:finite-dichotBRL} and \ref{T:finite-bicausalBRL} (both standard and strict
versions).  For the case of trivial dichotomy (the stable case where
$\sigma(A) \subset {\mathbb D}$ in Theorem \ref{T:finite-dichotBRL}), we have recently obtained such results
via two distinct approaches:  (i) the State-Space-Similarity theorem approach (see \cite{KYP1}), and (ii)
the storage-function approach (see \cite{KYP2}) based on the work of Willems \cite{Wil72a,Wil72b}.
Both approaches in general involve additional complications
in the infinite-dimensional setting.  In the first approach (i), one must deal with possibly unbounded pseudo-similarity
rather than true similarity transformations, as explained in the penetrating paper of Arov-Kaashoek-Pik \cite{AKP06};
one of the contributions of \cite{KYP1} was to identify additional hypotheses (exact or $\ell^2$-exact controllability and
observability) which guarantee that the pseudo-similarities guaranteed by the Arov-Kaashoek-Pik theory can in fact be
taken to be bounded and boundedly invertible.  In the second approach (ii), no continuity properties of a storage function
are guaranteed a priori and in general one must allow a storage function to take the value $+\infty$; nevertheless, as shown
in \cite{KYP2}, it is possible to show that the Willems available storage function $S_a$ and a regularized version of the Willems
required supply $\uS_r$ (at least when suitably restricted) have a quadratic form coming from a possibly unbounded
positive-definite operator ($H_a$ and $H_r$ respectively) which leads to a solution (in an adjusted generalized sense
required for the formulation of what it should mean for an unbounded operator to be a solution) of the KYP-inequality.
Again, if the system satisfies an exact or $\ell^2$-exact controllability/observability
hypothesis, then we get finite-valued quadratic storage functions and associated bounded and boundedly invertible solutions
of the KYP-inequality.

It seems that the first approach (i) (involving the
State-Space-Similarity theorem with pseudo-similarities) does not adapt well
in the dichotomous setting, so we here focus on the second approach
(ii) (computation of extremal storage functions).  For the
dichotomous setting, there again is a notion of storage function but
now the storage functions $S$ can take values on the whole real line rather than just positive values,
 and quadratic storage functions should have the form $S(x) = \langle H x , x
\rangle$ (at least for $x$ restricted to some appropriate domain) with $H$ (possibly unbounded)
selfadjoint rather than just positive-definite.
Due to the less than satisfactory connection between closed forms and closed operators for forms not
defined on the whole space and not necessarily semi-bounded
(see e.g.\ \cite{RS, Kato}), it is difficult to make sense of quadratic storage functions in the
infinite-dimensional setting unless the storage function is finite-valued and the associated
self-adjoint operator is bounded.  Therefore, for the dichotomous setting here we deal only with the case
where $\ell^2$-exact controllability/observability
assumptions are imposed at the start, and we are able to consider only storage functions $S$
which are finite real-valued with the associated selfadjoint operators in an quadratic representation equal to bounded operators.
Consequently our results require either the strict inequality condition
$\| F \|_{\infty, {\mathbb T}} < 1$ on the transfer function $F$, or an
$\ell^{2}$-exact or exact controllability/observability assumption on the
operators in the system matrices.  Consequently, unlike what is done
in \cite{KYP1, KYP2} for the causal trivial-dichotomy setting, the present paper
has nothing in the direction of a Bounded Real Lemma for a
dichotomous or exponentially dichotomous system under only
(approximate) controllability and observability assumptions for the case
where $\| F \|_{\infty, {\mathbb T}} = 1$.

The paper is organized as follows. Apart from the current introduction, the paper consists of seven sections. In Sections
\ref{S:dichotsys} and \ref{S:bicausal} we introduce the dichotomous systems and bicausal systems, respectively, studied
in this paper and derive various basic results used in the sequel. Next, in Section \ref{S:storage} we introduce the notion of a
storage function for discrete-time dichotomous linear systems as well as the available storage $S_a$ and required supply $S_r$
storage functions in this context and show that they indeed are storage functions (pending the proof of a continuity condition
which is obtained later from Theorem \ref{T:SaSr-Con}). In Section \ref{S:SaSr} we show, under certain conditions, that $S_a$
and $S_r$ are quadratic storage functions by explicit computation of the corresponding invertible selfadjoint operators $H_a$
and $H_r$. The results of Sections \ref{S:storage} and \ref{S:SaSr} are extended to bicausal systems in
Section \ref{S:bicausal-storage}. The main results of the present paper, i.e., the infinite-dimensional versions of
Theorems \ref{T:finite-dichotBRL} and \ref{T:finite-bicausalBRL}, are proven in Section \ref{S:BRLs}. In the final section,
Section \ref{S:nonstat}, we apply our Dichotomous Bounded Real Lemma to discrete-time, nonstationary, dichotomous l
inear systems and recover a result of Ben-Artzi--Gohberg--Kaashoek \cite{BAGK95}.

\section{Dichotomous system theory}\label{S:dichotsys}

We assume that we are given a system $\Si$ as in \eqref{dtsystem} with system matrix $M = \sbm{ A & B \\ C &
D}$  and associated  transfer function $F_{\Si}$ as in \eqref{transfunct} with $A$ having dichotomy.
As a neighborhood of the unit circle ${\mathbb T}$ is in the resolvent set of $A$, by definition of $A$ having a dichotomy,
we see that $F_\Sigma(z)$ is analytic and uniformly bounded in $z$ on a neighborhood of ${\mathbb T}$.
One way to make this explicit is to decompose $F_\Sigma$ in the form
$F_\Sigma = F_{\Sigma, +} + F_{\Sigma, -}$ where $F_{\Sigma, +}(z)$  is analytic and uniformly bounded on a neighborhood
of the closed unit disk $\overline{\mathbb D}$ and where $F_{\Sigma, -}(z)$ is analytic and uniformly bounded on a neighborhood
of the closed exterior unit disk $\overline{{\mathbb D}_e}$ as follows.

The fact that $A$ admits
a dichotomy implies there is a direct (not necessarily orthogonal) decomposition
of the state space $\cX = \cX_{+} \dot + \cX_{-}$ so that with respect to
this decomposition $A$ has a block diagonal matrix decomposition of the form
\begin{equation}\label{Adec}
A = \begin{bmatrix} A_{-} & 0 \\ 0 & A_{+} \end{bmatrix}
:\mat{c}{\cX_-\\ \cX_+}\to\mat{c}{\cX_-\\ \cX_+}
\end{equation}
where $A_{+}: = A|_{\cX_{+}} \in \cL(\cX_{+})$ has spectrum inside the
unit disk ${\mathbb D}$ and $A_{-}: = A|_{\cX_{-}} \in \cL(\cX_{-})$
has spectrum in the exterior of the closed unit disk ${\mathbb
D}_{e}=\BC\backslash \overline{\BD}$.  It follows that $A_{+}$ is {\em exponentially stable}, $\spec(A_+)<1$, and
$A_{-}$ is invertible with inverse $A_{-}^{-1}$ exponentially stable. Occasionally we will view $A_+$ and $A_-$ as operators acting
on $\cX$ and, with some abuse of notation, write $A_-^{-1}$ for what is really a generalized inverse of
$A_-$:
\[
A_-^{-1} \cong \mat{cc}{ A_-^{-1} & 0 \\ 0 & 0}:\mat{c}{\cX_-\\ \cX_+}\to\mat{c}{\cX_-\\ \cX_+}
\]
i.e., the Moore-Penrose generalized
inverse of $A_-$ in case the decomposition $\cX_- \dot  + \cX_+$ is orthogonal---the
meaning will be clear from the context. Now decompose $B$ and $C$ accordingly:
\begin{equation}\label{BCdec}
B=\mat{c}{B_- \\ B_+}:\cU\to\mat{c}{\cX_- \\ \cX_+}\ \mbox{ and }\
C=\mat{cc}{C_- & C_+}:\mat{c}{\cX_-\\ \cX_+}\to\cY.
\end{equation}
We may then write
\begin{align*}
F_\Sigma(z) & = D + zC (I - z A)^{-1} B \\
& = D +  z\begin{bmatrix} C_- & C_+ \end{bmatrix} \begin{bmatrix}  I - zA_- & 0 \\ 0 &  I - zA_+ \end{bmatrix} ^{-1}
\begin{bmatrix} B_- \\ B_+ \end{bmatrix} \\
& =  D + z C_- ( I - z A_- )^{-1} B_-  + z C_+ (I - z A_+)^{-1} B_+  \\
& =  -C_- A_-^{-1} ( z^{-1}I -  A_-^{-1})^{-1} B_-  + D + z C_+ (I - z  A_+)^{-1} B_+ \\
& = F_{\Sigma,-}(z) + F_{\Sigma, +}(z)
\end{align*}
where
\begin{equation}  \label{FSigma-}
F_{\Sigma, -}(z) = - C_-  A_-^{-1} (I - z^{-1} A_-^{-1})^{-1} B_-
  = -\sum_{n=0}^\infty C_- (A_-^{-1})^{n+1} B_-  z^{-n}
\end{equation}
 is analytic on a neighborhood of $\overline{{\mathbb D}_e}$, with the series converging in operator norm on ${\mathbb D}_e$ due to
 the exponential stability of $A_-^{-1}$, and
where
\begin{equation}   \label{FSigma+}
F_{\Sigma,+}(z) =  D + z C_+ (I - z A_+)^{-1} B_+ = D + \sum_{n=1}^\infty C_+ A_+^{n-1} B_+ z^n
\end{equation}
is analytic on a neighborhood of $\overline{\mathbb D}$, with the series converging in operator norm on ${\mathbb D}$
due to the exponential stability of $A_+$.  Furthermore, from the convergent-series expansions for $F_{\Sigma, +}$ in
 \eqref{FSigma+} and for $F_{\Sigma, -}$ in \eqref{FSigma-} we read off that $F_\Sigma$ has the convergent Laurent
 expansion on the unit circle ${\mathbb T}$
 $$
   F_\Sigma(z) = \sum_{n=-\infty}^\infty F_n z^n
$$
with Laurent coefficients $F_n$ given by
\begin{equation}   \label{Laurent-coeff}
 F_n = \begin{cases}  D - C_- A_-^{-1} B_- &\text{if } n=0, \\
          C_+ A_+^{n-1} B_+ &\text{if } n > 0, \\
          -C_- A_-^{n-1}B_- &\text{if } n < 0.
\end{cases}
\end{equation}
As $\| F_\Sigma \|_{\infty, {\mathbb T}} : = \sup \{ \| F_\Si(z) \| \colon z \in {\mathbb T}\} < \infty$,  it follows that $F_\Si$
defines a bounded multiplication operator:
\[
M_{F_\Si}:L^2_\cU({\mathbb T})\to L^2_\cY({\mathbb T}),\quad
M_{F_\Si} \colon f(z) \mapsto F_\Sigma(z) f(z)
\]
with $\| M_{F_\Sigma} \| = \| F_\Sigma \|_{\infty, {\mathbb T}}$.  If we write this operator as a block matrix
$ M_{F_\Sigma} = [ M_{F_\Sigma} ]_{ij}$ ($-\infty < i,j < \infty$) with respect to the orthogonal decompositions
$$
  L^2_\cU({\mathbb T}) = \bigoplus_{n=-\infty}^\infty z^n \cU, \quad
  L^2_\cY({\mathbb T}) = \bigoplus_{n=-\infty}^\infty z^n \cY
$$
for the input and output spaces for $M_{F_\Sigma}$,  it is a standard calculation to verify that $[M_{F_\Sigma}]_{ij} = F_{i-j}$, i.e.,
the resulting bi-infinite matrix $[M_{F_\Sigma}]_{ij}$ is the {\em Laurent matrix} $\fL_{F_\Sigma}$ associated with $F_\Sigma$ given by
\begin{equation}  \label{Laurent-matrix}
   \fL_{F_\Sigma} = [ F_{i-j} ]_{i,j=-\infty}^\infty
\end{equation}
where $F_n$ is as in \eqref{Laurent-coeff}.  Another expression of this identity is the fact that $M_{F_\Sigma} \colon
L^2_\cU({\mathbb T}) \to L^2_\cY({\mathbb T})$ is just the frequency-domain expression of the time-domain operator
$\fL_{F_\Sigma} \colon \ell^2_\cU({\mathbb Z}) \to \ell^2_\cY({\mathbb Z})$, i.e.,
if we let $\widehat \bu(z) = \sum_{n=-\infty}^\infty \bu(n) z^n$ in $L^2_\cU({\mathbb T})$ be the bilateral $Z$-transform of
$\bu$ in $\ell^2_\cU({\mathbb Z})$ and similarly let $\widehat \by(z) = \sum_{n=-\infty}^\infty \by(n) z^n$ in $L^2_\cY({\mathbb T})$
be the bilateral $Z$-transform of $\by$ in $\ell^2_\cY({\mathbb Z})$, then we have the relationship
\begin{equation}   \label{MFvsLF}
\by = \fL_{F_\Sigma} \bu \quad \Longleftrightarrow \quad \widehat \by(z) = F_\Sigma(z) \cdot \widehat \bu(z) \text{ for almost all } z \in {\mathbb T}.
\end{equation}

We now return to analyzing the system-theoretic properties of the dichotomous system \eqref{dtsystem}.
Associated with the system operators $A, B, C,D$ are the diagonal
operators $\cA, \cB, \cC, \cD$ acting between the appropriate
$\ell^{2}$-spaces indexed by ${\mathbb Z}$:
\begin{equation}\label{cABCD}
\begin{aligned}
\cA = {\rm diag}_{k \in {\mathbb Z}} [A] \colon
\ell^{2}_{\cX}({\mathbb Z}) \to \ell^{2}_{\cX}({\mathbb Z}),\quad
&  \cB = {\rm diag}_{k \in {\mathbb Z}} [B] \colon
\ell^{2}_{\cU}({\mathbb Z}) \to \ell^{2}_{\cX}({\mathbb Z}), \\
  \cC = {\rm diag}_{k \in {\mathbb Z}} [C] \colon
\ell^{2}_{\cX}({\mathbb Z}) \to \ell^{2}_{\cY}({\mathbb Z}),\quad
  &  \cD = {\rm diag}_{k \in {\mathbb Z}} [D] \colon
\ell^{2}_{\cY}({\mathbb Z}) \to \ell^{2}_{\cU}({\mathbb Z}).
\end{aligned}
\end{equation}
We also introduce the bilateral shift operator
$$
\cS:\ell^2_\cX(\BZ)\to \ell^2_\cX(\BZ),\quad
\cS \colon \{\bx(k)\}_{k \in {\mathbb Z}} \mapsto \{ \bx(k-1)\}_{k
\in {\mathbb Z}}
$$
and its inverse
$$
\cS^{-1} = \cS^{*} \colon  \{\bx(k)\}_{k \in {\mathbb Z}} \mapsto \{
\bx(k+1)\}_{k \in {\mathbb Z}}.
$$
We can then rewrite the system equations
\eqref{dtsystem} in aggregate form
\begin{equation}   \label{dtsys-agg}
\Si:=\left\{
\begin{array}{ccc}
\cS^{-1} \bx &=& \cA \bx + \cB \bu,\\
\by &=& \cC \bx +  \cD \bu,
\end{array}
\right.
\end{equation}
We shall say that a system trajectory $(\bu,\bx,\by)=\{(\bu(n), \bx(n), \by(n) \}_{n \in {\mathbb Z}}$ is
{\em $\ell^2$-admissible} if all of $\bu$, $\bx$, and $\by$ are in $\ell^2$:
  $\bu = \{ \bu(n) \}_{n \in {\mathbb Z}} \in
\ell^{2}_{\cU}({\mathbb Z})$,  $\bx = \{ \bx(n) \}_{n \in {\mathbb
Z}} \in
\ell^{2}_{\cX}({\mathbb Z})$,  $\by = \{ \by(n) \}_{n \in {\mathbb
Z}} \in
\ell^{2}_{\cY}({\mathbb Z})$.
Note that the constant-diagonal structure of $\cA$, $\cB$, $\cC$, $\cD$ implies that each of these operators intertwines
the bilateral shift operator on the appropriate $\ell^2({\mathbb Z})$-space:
\begin{equation}  \label{commute}
\cA \cS = \cS \cA, \quad \cB \cS = \cS \cB, \quad \cC \cS = \cS \cC, \quad  \cD \cS  = \cS \cD
\end{equation}
where $\cS$ is the bilateral shift operator on $\ell^2_\cW$ with $\cW$ is any one of $\cU$, $\cX$, $\cY$ depending
on the context.

It is well known (see e.g.\ \cite[Theorem 2]{BG2}) that the
operator $A$ admitting a dichotomy  is equivalent to
$\cS^{-1} - \cA$ being invertible as an operator on
$\ell^{2}_{\cX}({\mathbb Z})$. Hence the dichotomy hypothesis enables
us to solve uniquely for $\bx \in \ell^{2}_{\cX}({\mathbb Z})$ and
$\by \in \ell^{2}_{\cY}({\mathbb Z})$ for any given $\bu \in
\ell^{2}_{\cU}({\mathbb Z})$:
\begin{equation}\label{is/o}
\begin{aligned}
   & \bx = (\cS^{-1} - \cA)^{-1} \cB \bu = (I - \cS \cA)^{-1} \cS \cB
    \bu =: T_{\Sigma,\, \rm is} \bu, \\
& \by = (\cD + \cC (\cS^{-1} - \cA)^{-1} \cB) \bu =
(\cD + \cC (I -  \cS \cA)^{-1} \cS \cB) \bu =: T_{\Sigma} \bu.
\end{aligned}
\end{equation}
where
\begin{align}
& T_{\Sigma, \, \rm is} = (\cS^{-1} -  \cA)^{-1} \cB:\ell^2_\cU(\BZ) \to \ell^2_\cX(\BZ)   \label{ISmap} \\
& T_{\Sigma}  = \cD + \cC (\cS^{-1} -  \cA)^{-1} \cB :\ell^2_\cU(\BZ) \to \ell^2_\cY(\BZ)   \label{IOmap}
\end{align}
are the respective {\em input-state} and {\em input-output} maps.   In general the  input-output map
$T_{\Sigma}$ in \eqref{IOmap} is not causal.
Given an $\ell^2_\cU({\mathbb Z})$-input signal $\bu$, rather than specification of an initialization condition on the state $\bx(0)$,
as in standard linear system theory for systems running on ${\mathbb Z}_+$,
in  order to specify a uniquely determined state trajectory $\bx$ for a given input trajectory $\bu$, the extra information
required to solve uniquely for the state trajectory $\bx$ in the dichotomous system \eqref{dtsystem} or
\eqref{dtsys-agg} is the specification that $\bx \in \ell^2_\cX({\mathbb Z})$, i.e., that the resulting trajectory $(\bu, \bx, \by)$
be $\ell^2$-admissible.

Next we express various operators explicitly in terms of $A_{\pm}$, $B_{\pm}$, $C_{\pm}$ and $D$. The following lemma
provides the basis for the formulas derived in the remainder of the section.  In fact this lemma amounts to the easy direction
of the result of Ben-Artzi--Gohberg--Kaashoek \cite[Theorem 2]{BG2} mentioned above.

\begin{lemma}\label{L:InvForm}
Let $\Si$ be the dichotomous system \eqref{dtsystem} with $A$ decomposing as in \eqref{Adec}. Then
$(\cS^{-1}- \cA)^{-1}=(I-\cS\cA)^{-1}\cS$ acting on $\ell^{2}_{\cX}(\BZ)$ is given explicitly as the following block matrix
with rows and columns indexed by ${\mathbb Z}$
\begin{equation}   \label{inverse}
[ (\cS^{-1} - \cA)^{-1}]_{ij} = \begin{cases}
A_{+}^{i-j-1} &\text{for } i > j, \\
- A_{-}^{i-j-1} &\text{for } i \le j
\end{cases}, \quad \mbox{with } A_+^0=P_{\cX_+}.
\end{equation}
\end{lemma}

\begin{proof}[\bf Proof]
Via the decomposition $\cX=\cX_- \dot + \cX_+$, we can identify $\ell^2_\cX(\BZ)$ with
$\ell^2_{\cX_-}(\BZ) \dot + \ell^2_{\cX_+}(\BZ)$. Write $\cS_+$ for the bilateral shift operator and $\cA_+$ for the
block diagonal operator with $A_+$ diagonal entries, both acting on $\ell^2_{\cX_+}(\BZ)$, and write $\cS_-$ for the
bilateral shift operator and $\cA_-$ for the block diagonal operator with $A_-$ diagonal entries, both acting on
$\ell^2_{\cX_-}(\BZ)$. Then with respect to the above decomposition of $\ell^2_\cX(\BZ)$ we have
\begin{align*}
(I-\cS\cA)^{-1}&=\mat{cc}{I-\cS_-\cA_- & 0\\ 0 & I-\cS_+ \cA_+}^{-1}\\
&=\mat{cc}{ (I-\cS_-\cA_-)^{-1} & 0\\ 0 & (I-\cS_+\cA_+)^{-1} }.
\end{align*}
Since $A_+$ has its spectrum in $\BD$, so do $\cA_+$ and $\cS_+\cA_+$, and thus
\[
(I-\cS_+\cA_+)^{-1}=\sum_{k=0}^\infty (\cS_+\cA_+)^k =\sum_{k=0}^\infty \cS_+^k\cA_+^k
\]
where we make use of observation \eqref{commute} to arrive at the final infinite-series expression.
Similarly, $A_-^{-1}$ has spectrum in $\BD$ implies that $\cA_-^{-1}$ and $\cA_-^{-1}\cS_-^{-1}$ have spectrum
in $\BD$, and hence
\begin{align*}
(I-\cS_-\cA_-)^{-1} & =  (\cS_-\cA_-)^{-1}((\cS_-\cA_-)^{-1}-I)^{-1}\\
&= -\cA_-^{-1}\cS_-^{-1}(I- (\cS_-\cA_-)^{-1})^{-1}
=-\cA_-^{-1}\cS_-^{-1}\sum_{k=0}^\infty (\cS_-\cA_-)^{-k}\\
&=-\cA_-^{-1}\cS_-^{-1}\sum_{k=0}^\infty \cA_-^{-k}\cS_-^{-k}
=-\sum_{k=1}^\infty \cA_-^{-k}\cS_-^{-k}.
\end{align*}
Inserting the formulas for $(I-\cS_+\cA_+)^{-1}$ and $(I-\cS_-\cA_-)^{-1}$ in the formula for $(I-\cS\cA)^{-1}$,
multiplying with $\cS$ from the left and writing out in block matrix form we obtain the desired formula
for $(I-\cS\cA)^{-1}\cS=(\cS^{-1}-A)^{-1}$.
\end{proof}

We now compute the input-output map $T_\Si$ and input-state map $T_{\Si,\, \rm is}$ explicitly.

\begin{proposition}\label{P:TF}
Let $\Si$ be the dichotomous system \eqref{dtsystem} with $A$ decomposing as in \eqref{Adec} and $B$ and $C$ as in
\eqref{BCdec}. The input-output map $T_\Si:\ell^2_\cU(\BZ)\to\ell^2_\cY(\BZ)$ and input-state map
$T_{\Si,\, \rm is}: \ell^2_\cU(\BZ)\to\ell^2_\cX(\BZ)$ of $\Si$ are then given by the following block matrix, with row
and columns indexed over $\BZ$:
\begin{equation*}%\label{iomapcoord}
 [T_\Si]_{ij} =
 \begin{cases} C_+  A_{+}^{i-j-1}  B_+ &\text{if } i>j, \\
  D - C_-A_{-}^{-1}B_- &\text{if } i=j, \\
  -C_- A_{-}^{i-j-1}B_- &\text{if } i<j,
  \end{cases}   \quad
  [T_{\Si,\, \rm is}]_{ij} =
 \begin{cases} A_{+}^{i-j-1}  B_+ &\text{if } i>j, \\
  - A_{-}^{i-j-1}B_- &\text{if } i\leq j.
  \end{cases}
\end{equation*}
In particular, $T_\Si$ is equal to the Laurent operator $\fL_{F_\Si}$ of the transfer function $F_\Si$ given in
\eqref{Laurent-matrix}, and for $\bu \in \ell^2_\cU({\mathbb Z})$ and $\by \in \ell^2_\cY({\mathbb Z})$ with
bilateral $Z$-transform notation as in \eqref{MFvsLF},
\begin{equation}  \label{transfunc-prop}
\by = T_\Sigma \bu \quad \Longleftrightarrow \quad \widehat \by(z) = F_\Sigma(z) \cdot \widehat \bu(z)
\text{ for almost all } z \in {\mathbb T}.
\end{equation}
\end{proposition}

\begin{proof}[\bf Proof]
Recall that $T_\Si$ can be written as $T_\Si=\cD+\cC(I-\cS\cA)^{-1}\cS\cB =\cD+\cC(\cS^{-1}-\cA)^{-1}\cB$. The block matrix
formula for $T_\Si$ now follows directly from the block matrix formula for $(\cS^{-1}-\cA)^{-1}$ obtained in Lemma \ref{L:InvForm}.
Comparison of the formula for $[T_\Sigma]_{ij}$ with the formula \eqref{Laurent-coeff} for the Laurent coefficients
$\{F_n\}_{n \in {\mathbb Z}}$ of $F_\Sigma$ shows that $T_\Sigma = \fL_{F_\Sigma}$ as operators from $L^2_\cU({\mathbb T})$
to $L^2_\cY({\mathbb T})$.   Finally the identity \eqref{transfunc-prop} follows upon combining the identity
$\fL_{F_\Sigma} = T_\Sigma$ with the general identity \eqref{MFvsLF}.
\end{proof}

It is convenient to also view $T_\Si=\fL_{F_\Sigma}$ as a block $2 \times 2$ matrix
with respect to the decomposition $\ell^{2}_{\cU}({\mathbb Z}) =\ell^{2}_{\cU}({\mathbb Z}_{-}) \oplus \ell^{2}_{\cU}({\mathbb Z}_{+})$
for the input-signal space and
$\ell^{2}_{\cY}({\mathbb Z}) =
\ell^{2}_{\cY}({\mathbb Z}_{-}) \oplus \ell^{2}_{\cY}({\mathbb Z}_{+})$
for the output-signal space.  We can then write
\begin{equation}   \label{dichotLaurent}
{\frakL}_{F_\Si} = \begin{bmatrix}
\wtil\frakT_{F_\Si} & \wtil\frakH_{F_{\Si}}
\\
\frakH_{F_{\Si}} & \frakT_{F_\Si} \end{bmatrix}
:\mat{c}{\ell^{2}_{\cU}({\mathbb Z}_{-}) \\ \ell^{2}_{\cU}({\mathbb Z}_{+})}
\to\mat{c}{\ell^{2}_{\cY}({\mathbb Z}_{-}) \\ \ell^{2}_{\cY}({\mathbb Z}_{+})}
\end{equation}
where
\begin{align}
 &   [ \widetilde \frakT_{F_\Si}]_{ij \colon i< 0, j< 0} = \begin{cases}
C_{+}
    A_{+}^{i-j-1}B_{+} &\text{for } 0 > i > j, \\
    D- C_{-} A_{-}^{-1} B_{-} &\text{for } i = j < 0, \\
    -C_{-}  A_{-}^{i-j-1}B_{-} &\text{for } i < j < 0,
    \end{cases}  \notag \\
&  [\frakT_{F_\Si}]_{ij \colon i\ge 0, j\ge 0} =
 \begin{cases} C_{+}
    A_{+}^{i-j-1}B_{+} &\text{for } i > j \ge 0, \\
    D- C_{-} A_{-}^{-1} B_{-} &\text{for } i = j \ge 0, \\
   -C_{-}  A_{-}^{i-j-1}B_{-} &\text{for }  0 \le i < j
    \end{cases}
    \label{Toeplitz}
 \end{align}
 are noncausal Toeplitz operators, and
 \begin{align}
    & [\widetilde \frakH_{F_{\Si}}]_{ij \colon i < 0, j \ge 0} =
    -C_{-} A_{-}^{i-j-1} B_{-} \text{ for } i<0,\, j \ge 0,
    \notag \\
 & [\frakH_{F_{\Si}}]_{ij \colon i \ge 0,\, j < 0} =
    C_{+}A_{+}^{i-j-1}B_{+} \text{ for } i\ge 0,\, j < 0
    \label{Hankel}
    \end{align}
are Hankel operators.

Next we consider the observability and controllability operators of $\Si$. For any integer $n$, let $\Pi_n:\ell^2_\cX(\BZ)\to\cX$
be the projection onto the $n^\tu{th}$ component of $\ell^2_\cX(\BZ)$. We then define the controllability operator $\bW_c$
and observability operator $\bW_o$ associated with the system $\Si$ as
\begin{align*}
\bW_c:\ell^2_\cU(\BZ)\to\cX, \quad & \bW_c \bu=\Pi_0 \bx=\Pi_0 T_{\Si,\,\rm is} \bu=\Pi_0(\cS^{-1}-\cA)^{-1}\cB \bu,\\
\bW_o:\cX\to\ell^2_\cY(\BZ),\quad &\bW_o x=\cC (I-\cS\cA)^{-1}\Pi_0^* x.
\end{align*}

\begin{lemma}\label{L:ObsConDec}
Let $\Si$ be the dichotomous system \eqref{dtsystem} with $A$ decomposing as in \eqref{Adec} and $B$ and $C$ as in
\eqref{BCdec}. Let the observability operator $\bW_o$ and controllability operator $\bW_c$ decompose as
\begin{align*}
\bW_{c}  &= \begin{bmatrix} \bW^{+}_{c} & \bW^{-}_{c} \end{bmatrix}
\colon  \ell^{2}_{\cU}({\mathbb Z}) = \begin{bmatrix}
\ell^{2}_{\cU}({\mathbb Z}_{-}) \\ \ell^{2}_{\cU}({\mathbb Z}_{+})
\end{bmatrix} \to \cX, \\
\bW_o &=\mat{c}{\bW_o^-\\ \bW_o^+}:\cX\to\mat{c}{\ell^2_\cY(\BZ_-)\\ \ell^2_\cY(\BZ_+)}.
\end{align*}
Then $\bW_{c}^{+}$ and $\bW_{c}^{-}$ are given by
\begin{align}
\bW_{c}^{+} =  {\rm row}_{j < 0} [ A_{+}^{-j-1}B_{+}],\quad
 &\quad \bW_{c}^{-} = {\rm row}_{j \ge 0} [- A_{-}^{-j-1} B_{-} ],
 \label{dichotcon}
\end{align}
and $\bW_{c}^{+}$ maps into $\cX_+$ and $\bW_{c}^{-}$ into $\cX_-$. Furthermore, $\bW_{o}^{+}$ and $\bW_{o}^{-}$ are given by
\begin{align}
\bW_{o}^{+} = {\rm col}_{i \colon i \ge 0} [C_{+} A_{+}^{i}],
\quad&\quad
\bW_{o}^{-} = {\rm col}_{i \colon i < 0}[ -C_{-} A_{-}^{i}],
 \label{dichotobs}
\end{align}
and $(\bW_o^{+})^{*}$ maps into $\cX_+$ and $(\bW_o^{-})^{*}$ maps into $\cX_-$. Finally, the Hankel
operators $\widetilde \frakH_{F_{-}}$ and $\frakH_{F_{+}}$ in \eqref{Hankel} have the following factorizations:
\begin{equation} \label{Hankelfact}
    \widetilde \frakH_{F_{\Si}} = \bW_{o}^{-} \bW_{c}^{-}, \quad
    \frakH_{F_{\Si}} = \bW_{o}^{+} \bW_{c}^{+}.
\end{equation}
\end{lemma}

\begin{proof}[\bf Proof]
The formulas in \eqref{dichotcon} follow directly by restricting the matrix representation of $T_{\Si,\,\rm is}$ obtained in
Proposition \ref{P:TF} to the zero-indexed row. Since $A_\pm$ and $B_\pm$ map into $\cX_\pm$, it follows directly that
$\bW_{c}^{\pm}$ maps into $\cX_\pm$. The analogous statements for $\bW_{o}$ follow by similar arguments,
now using \eqref{inverse} to compute the zero-indexed column of $\cC (I-\cS\cA)^{-1}$ explicitly. The factorization
formulas for $\widetilde \frakH_{F_{\Si}}$ and $\frakH_{F_{\Si}}$ follow directly from an inspection of the entries
 in the block matrix decompositions.
\end{proof}

\begin{remark}\label{R:stateWc}
Let $(\bu,\bx,\by)$ be an $\ell^2$-admissible trajectory for the system $\Si$. Then $\bx(0)=\Pi_0\bx=\bW_c \bu \in \im \bW_o$.
In fact, by shift invariance of the system $\Si$ we have $\bx(n)\in \im \bW_c$ for each $n\in\BZ$, because
\begin{align*}
\bx(n) &=\Pi_n \bx =\Pi_0 \cS^{-n}\bx =\Pi_0 \cS^{-n}(I-\cS\cA)^{-1}\cS \cB \bu\\
& =\Pi_0(I-\cS\cA)^{-1}\cS \cB \cS^{-n}\bu=\bW_c \cS^{-n}\bu,
\end{align*}
which holds since $\cS^{-n}$ commutes with $\cS$, $\cA$ and $\cB$.
\end{remark}

Another topic playing a prominent role in the theory of causal linear systems (see \cite{DP}) is that of controllability and observability.
For a causal system $\Sigma$ of the form \eqref{dtsystem} we say that $\Sigma$ (or the input pair $(A,B)$) is {\em controllable} if
$\overline{\operatorname{span}}_{k \ge 0} \im A^{k}B=\cX$, which in case $\bW_c$ is bounded is equivalent to $\bW_{c}$ having
dense range,
while $\Sigma$ (or the output pair $(C,A)$) is said to be {\em observable} if $\bigcap_{k \ge 0} \ker C A^{k} = \{0\}$,
which in turn is equivalent to $\kr \bW_o=\{0\}$ in case $\bW_o$ is bounded.

Note that if $\Sigma$ in \eqref{dtsystem} is a system having a dichotomy with associated
decomposition \eqref{Adec} and \eqref{BCdec}, then the assumed exponential
stability of the operators $A_{+}$ and $A_{-}^{-1}$ implies that the
associated controllability operators $\bW_{c}^{-} \colon
\ell^{2}_{\cU}({\mathbb Z}_{+}) \to \cX_{-}$ and $\bW_{c}^{+} \colon
\ell^{2}_{\cU}({\mathbb Z}_{-}) \to \cX_{+}$ as well as the
observability operators $\bW_{o}^{+} \colon \cX_{+} \to
\ell^{2}_{\cY}({\mathbb Z}_{+})$ and $\bW_{o}^{-} \colon \cX_{-} \to
\ell^{2}_{\cY}({\mathbb Z}_{-})$ are all bounded. We then say that $\Si$ (or the pair $(A,B)$) is {\em controllable} if $\bW_{c}$
has dense range and that $\Si$ (or the pair $(C,A)$) is {\em observable} if $\kr \bW_o=\{0\}$. With $(A,B,C)$ decomposed
as in \eqref{Adec} and \eqref{BCdec}, we see from Lemma \ref{L:ObsConDec} that controllability of $\Si$ is equivalent to
\begin{equation}   \label{dichot-con}
    (A_{+}, B_{+}) \text{ controllable} \ands (A_{-}^{-1}, A_{-}^{-1}B_{-})
    \text{ controllable},
\end{equation}
hence to $\overline{\operatorname{span}}_{k \ge 0} \im A_+^{k}B_+=\cX_+$ and
$\overline{\operatorname{span}}_{k \ge 1} \im A_-^{-k}B_-=\cX_-$, while $\Sigma$ being controllable is equivalent to
\begin{equation} \label{dichot-obs}
    (C_{+}, A_{+}) \text{ observable} \ands (C_{-}A_{-}^{-1}, A_{-}^{-1}) \text{
    observable},
\end{equation}
hence to $\bigcap_{k \ge 0} \ker C_+ A_+^{k} = \{0\}$ and $\bigcap_{k \ge 1} \ker C_- A_-^{-k} = \{0\}$.

We shall have need for stronger controllability/observability notions
for a dichotomous system defined as follows.   We shall
say that $\Sigma$ (or $(A,B)$) is {\em dichotomously $\ell^{2}$-exactly controllable} if
\begin{equation}  \label{dichot-exact-con}
  \im \bW_{c}^{+} = \cX_{+} \ands \im \bW_{c}^{-} = \cX_{-},\
  \text{or equivalently}\quad \im \bW_{c} = \cX.
 \end{equation}
Similarly, we say that $\Sigma$ (or $(C,A)$) is {\em dichotomously
$\ell^{2}$-exactly observable} if
\begin{equation}  \label{dichot-exact-obs}
  \im (\bW_{o}^{+})^{*} = \cX_{+} \text{ and } \im (\bW_{o}^{-})^{*} =
  \cX_{-}, \text{ or equivalently } \im \bW_{o}^{*} = \cX.
 \end{equation}
 In case $\Sigma$ is both dichotomously $\ell^{2}$-exactly
 controllable and dichotomously $\ell^{2}$-exactly observable, we
 shall say simply that $\Sigma$ is {\em dichotomously
 $\ell^{2}$-exactly minimal.}
 We note that these notions for the stable (non-dichotomous) case
 played a key role in the results of \cite{KYP1, KYP2}.

\begin{remark}
In that case that $\cX$ is finite dimensional, the notion of controllability (respectively, observability) for dichotomous systems
introduced here coincides with the more standard notion, namely, that
$\overline{\operatorname{span}}_{k \ge 0} \im A^{k}B=\cX$(respectively, $\bigcap_{k \ge 0} \ker C A^{k} = \{0\}$).
Indeed, to see that this is the case, note that it suffices to show that $(A_{-}^{-1}, A_{-}^{-1}B_{-})$ being a controllable
pair is equivalent to $(A_-,B_-)$ being a controllable pair. Since the two statements are symmetric, it suffices to prove
only one direction. Hence, assume the pair $(A_-,B_-)$ is controllable. Since $\cX$ is finite dimensional, this implies
there is a positive integer $n$ such that
\begin{align*}
\cX & =\im \mat{cccc}{B_-&A_-B_-&\cdots&A_-^{n-1}B_-} \\
&= \im A_-^{n} \mat{cccc}{A_-^{-n}B_-&A_-^{-n+1}B_-&\cdots&A_-^{-1}B_-} \\
&= A_-^{n}  \, \im  \mat{cccc}{A_-^{-1}B_-&\cdots&A_-^{-n+1}B_-&A_-^{-n}B_-}.
\end{align*}
Thus $\cX= \im  \mat{cccc}{A_-^{-1}B_-&\cdots&A_-^{-n+1}B_-&A_-^{-n}B_-}$, and we obtain that $(A_{-}^{-1}, A_{-}^{-1}B_{-})$
is a controllable pair. For the notions of observability the claim follows by a duality argument.

If $\cX$ is infinite-dimensional, it is not clear whether the two notions coincide.  Let us discuss here only the
situation for controllability as that for observability is similar.  Let $A \in \cL(\cX)$ and $B \in \cL(\cX, \cU)$
where now both $\cX$ and $\cU$ are allowed to be infinite-dimensional Hilbert spaces.  If $(A,B)$ is a controllable pair,
then, by definition, for a given $x \in \cX$
and $\epsilon > 0$, there is an $N = N(x, \epsilon) \in {\mathbb N}$ and vectors $u_0, u_1, \dots, u_N \in \cU$ so that
$ \| \sum_{k=0}^N A^k B u_k - x \| < \epsilon$.  Similarly, given $x \in \cX$, $N \in {\mathbb N}$ and $\epsilon > 0$, there is
a $\widetilde N = \widetilde N(x,N, \epsilon) \in {\mathbb N}$ so that there exist vectors $u'_0, u'_1, \dots, u'_{\widetilde N} \in \cU$
so that $\| \sum_{k=0}^{\widetilde N} A_k B u'_k - A^{N+1} x \| < \epsilon/ \| (A^{-1})^{N + 1} \|$.  Let us say that the pair $(A,B)$ is
 {\em uniformly controllable} if it is possible to take $\widetilde N(x,  N(x, \epsilon), \epsilon) = N(x, \epsilon)$, i.e., if:
{\em given $x \in \cX$ and $\epsilon > 0$ there is an $N = N_{x, \epsilon} \in {\mathbb N}$ so that
there is a choice of $u_0, u_1, \dots, u_{N} \in \cU$ so that}
$$
 \left\| \sum_{k=0}^{N} A^k B u_k - A^{N+1} x \right\| < \frac{ \epsilon}{ \| (A^{-1})^{N + 1}\|}.
$$
Note that the notions of {\em uniform controllability} and {\em controllability} are equivalent in the finite-dimensional
case---take $N_{x, \epsilon} = \dim \cX$
and then use the Cayley-Hamilton theorem to approximate $A^{N+1} x$ exactly by a vector of the form
$\sum_{k=0}^N  A^k B u_k$ ($N = \dim \cX$).
In the infinite-dimensional case arguably the condition appears to be somewhat contrived and is difficult to
check; nevertheless it is what is needed for the following result.

\smallskip

\noindent
\textbf{Proposition.}  {\sl  Assume that  $A$ is invertible and that the input pair $(A,B)$ is uniformly controllable.
Then  $(A^{-1}, A^{-1}B)$ is controllable.}

\smallskip

\noindent
\begin{proof}[\bf Proof]  Let $x \in \cX$ and  $\epsilon > 0$.  Let $N = N_{x, \epsilon}$ as in
 the uniformly-controllable condition:  thus there exist vectors $u_0, u_1, \dots, u_N \in \cU$ so that
$$
 \left\| \sum_{k=0}^N A^k B u_k - A^{N + 1} x \right\| < \frac{\epsilon}{\| (A^{-1})^{N+1} \| }.
$$
Rewrite this as
$$
\left\| A^{N+1} \left(  \sum_{k=0}^N (A^{-1})^k A^{-1} B u_{N-k} - x \right) \right\|  <  \frac{\epsilon}{ \| (A^{-1})^{N+1} \| }
$$
from which we get
\begin{align*}
& \left\| \sum_{k=0}^N (A^{-1})^k A^{-1} B u_{N-k}  - x \right\|   \\
& \quad \quad \quad \quad =
\left\| (A^{-1})^{N+1} \cdot A^{N+1} \left( \sum_{k=0}^N (A^{-1})^k A^{-1} B u_{N-k}  - x \right) \right\| \\
& \quad \quad \quad \quad <  \| (A^{-1})^{N+1} \| \cdot
\frac{\epsilon}{\| (A^{-1})^{N+1} \| } = \epsilon.
\end{align*}
As $x\in \cX$ and $\epsilon > 0$ are arbitrary, we conclude that $(A^{-1}, A^{-1} B)$ is controllable.
\end{proof}
\end{remark}

The following $\ell^2$-admissible-trajectory interpolation result will be useful in the sequel.

\begin{proposition}  \label{P:traj-int}
Suppose that $\Sigma$ is a dichotomous linear system as in \eqref{dtsystem}, \eqref{Adec}, \eqref{BCdec},
and that we are given a vector $u \in \cU$ and $x \in \cX$.  Assume that $\Sigma$ is dichotomously $\ell^2$-exactly controllable.
Then there exists an $\ell^2$-admissible system trajectory
$(\bu, \bx, \by)$ for $\Sigma$ such that
$$
  \bu(0) = u, \quad \bx(0) = x.
$$
\end{proposition}

\begin{proof}[\bf Proof]
As $\Sigma$ is $\ell^2$-exactly controllable, we know that $\bW_c^-$ and $\bW_c^+$ are surjective.  Write $x=x_+ + x_-$ with $x_\pm\in\cX_\pm$. $\bu_- \in \ell^2_\cU({\mathbb Z}_-)$ so that
$\bW_c^+ \bu_- = x_+$. Choose $\bu_- \in \ell^2_\cU({\mathbb Z}_-)$ so that
$\bW_c^+ \bu_- = x_+$. Next solve for $x'_-$ so that $x_- = A_-^{-1} x'_- - A_-^{-1} B_- u$, i.e., set
\begin{equation}   \label{solution}
   x'_- := A_- x_-  + A_- B_- u.
\end{equation}
Use the surjectivity of the controllability operator $\bW_c^-$ to find $\bu_+ \in \ell^2_\cU(\BZ_+)$  so that $\bW_c^- \bu_+ = x'_-$. We now define a new input signal $\bu$ by
$$
 \bu(n) = \begin{cases}  \bu_-(n) &\text{if } n < 0, \\
      u &\text{if } n=0, \\
      \bu_+(n-1) &\text{if } n \ge 1.  \end{cases}
$$
Since $\bu_+$ and $\bu_-$ are $\ell^2$-sequences, we obtain that $\bu\in\ell^2_\cU(\BZ)$. Now let $(\bu,\bx,\by)$ be the $\ell^2$-admissible system trajectory determined by the input sequence $\bu$. Clearly $\bu(0)=u$. So it remains to show that $\bx(0)=x$. To see this, note that
\begin{align*}
\bx(0) & = \bW_c \bu =\bW_c^+\bu_- +\bW_c^- (\mat{ccc}{u & 0 & \cdots} +\cS \bu_+)\\
&= x_+ -A_-^{-1}B_- u + A_-^{-1} \bW_c^- \bu_+
= x_+ -A_-^{-1}B_- u + A_-^{-1}x_-'\\
&= x_+ + x_-=x. \qedhere
\end{align*}
%
%Choose $\bu_- \in \ell^2_\cU({\mathbb Z}_-)$ so that
%$\bW_c^+ \bu_- = x_+$.  View $\bu_-$ as an element of $\ell^2_\cU({\mathbb Z})$ by setting $\bu_-(n) = 0$ for $n \ge 0$.
%Set $\bx_+ = T_{\Sigma_{+, is}} \bu_-$, $\by_+ = T_{\Sigma_+} \bu_-$.  Then $(\bu_-, \bx_+, \by_+)$ is an $\ell^2$-admissible
%system trajectory for $\Sigma_+$ with $\bx_+(0) = x_+$.
%
%Next solve for $x'_-$ so that $x_- = A_- x'_- + B_- u$, i.e., set
%\begin{equation}   \label{solution}
%   x'_- = A_-^{-1} x_-  - A_-^{-1} B_- u.
%\end{equation}
%Use the surjectivity of the controllability operator $\bW_c^-$ to find $\bu_+ \in \ell^2_\cU([1, \infty))$  so that
%$\bW_c^- \cS^* \bu_+ = x'_-$.  View $\bu_+$ as an element of $\ell^2_\cU({\mathbb Z})$ by setting $\bu_+(n) = 0$ for $n \le 0$.
%Set
%$$
%\bx_- = \cS T_{\Sigma_{-, is}} \cS^* \bu, \quad \by_- =\cS T_{\Sigma_-} \cS^* \bu.
%$$
%Then $(\bu_+, \bx_-, \by_-)$ is an $\ell^2$-admissible system trajectory for $\Sigma_-$ such that
%$ \bx_-(1) = x'_- := A_-^{-1} x_- - A_-^{-1} B_- u$,  i.e., such that
%$$
%  x_- = A \bx_-(1) + B_- u.
%$$
%Let us define a new input signal $\bu$ by
%$$
% \bu(n) = \begin{cases}  \bu_-(n) &\text{if } n < 0, \\
%      u &\text{if } n=0, \\
%      \bu_+(n) &\text{if } n \ge 1.  \end{cases}
%$$
%Then it is now just a matter of checking to see that $(\bu, \bx, \by)$ is an $\ell^2$-admissible trajectory for $\Sigma$ which
%meets the interpolation conditions $\bu(0) = u$ and $\bx(0) = \sbm{ x_- \\ x_+ }$.
\end{proof}

\section{Bicausal systems} \label{S:bicausal}

Even for the setting of rational matrix functions, it is not the
case that a rational matrix function $F$ which is analytic on a
neighborhood of the unit circle ${\mathbb T}$ necessarily has a realization of the form
\eqref{transfunct}, as such a realization for $F$ implies that
$F$ must be analytic at the origin.  What is required instead is a
slightly more general notion of a system, which we will refer to as a {\em bicausal system}, defined as follows.

A {\em bicausal system} $\Sigma$ consists of a pair of input-state-output linear systems $\Sigma_+$ and $\Sigma_-$
with $\Sigma_+$ running in forward time
and $\Sigma_-$ running in backward time
\begin{align}
& \Sigma_{-} \colon  \left\{ \begin{array}{rcl}  \bx_{-}(n) & = &
     \wtilA_{-} \bx_{-}(n+1) + \wtilB_{-} \bu(n),  \\
   \by_{-}(n) & = & \wtilC_{-} \bx_{-}(n)  \end{array}   \right. \quad (n\in\BZ) \label{dtsystem-b} \\
 & \Sigma_{+}\colon \left\{ \begin{array}{rcl}
\bx_{+}(n+1) & = & \wtilA_{+} \bx_{+}(n) + \wtilB_{+} \bu(n),  \\
 \by_{+}(n) & = & \wtilC_{+}  \bx_{+}(n) + \wtilD \bu(n)
\end{array} \right. \quad (n\in\BZ) \label{dtsystem-f}
\end{align}
with $\Si_-$ having state space $\cX_-$ and state operator $\wtilA_-$ on $\cX_-$ exponentially stable
(i.e., $\si(\wtilA_-)\subset\BD$) and $\Si_+$ having state space $\cX_+$ and $\wtilA_+$ on $\cX_+$ exponentially stable
($\si(\wtilA_+)\subset\BD$).   A system trajectory consists of a triple $\{\bu(n), \bx(n), \by(n) \}_{n \in {\mathbb Z}}$
such that
$$
  \bu(n) \in \cU, \quad \bx(n) = \sbm{ \bx_- \\ \bx_+ } \in \sbm{\cX_- \\ \cX_+}, \quad
  \by(n) = \by_-(n) + \by_+(n) \text{ with } \by_\pm(n) \in \cY
$$
such that $(\bu, \bx_-, \by_-)$ is a system trajectory of $\Sigma_-$ and $(\bu, \bx_+, \by_+)$ is a system trajectory
of $\Sigma_+$.  We say that the system trajectory $(\bu, \bx, \by) = (\bu, \sbm{ \bx_- \\ \bx_+}, \by_- + \by_+)$
is {\em $\ell^2$-admissible} if all system signals are in $\ell^2$:
$$
\bu \in \ell^2_\cU({\mathbb Z}), \quad \bx_+ \in \ell^2_{\cX_+}({\mathbb Z}),  \quad \bx_- \in \ell^2_{\cX_-}({\mathbb Z}),
\quad \by_\pm \in \ell^2_\cY({\mathbb Z}).
$$
Due to the assumed exponential stability of $\widetilde A_+$,  given $\bu \in \ell^2_\cU({\mathbb Z})$, there is a uniquely determined
$\bx_+ \in \ell^2_{\cX_+}({\mathbb Z})$ and $\by_+ \in \ell^2_\cY({\mathbb Z})$ so that $(\bu, \bx_+, \by_+)$
is an $\ell^2$-admissible system trajectory for $\Sigma_+$ and similarly for $\widetilde A_-$ due to the assumed
exponential stability of $\widetilde A_-$.  The result is as follows.

\begin{proposition}  \label{P:TSigBC}
Suppose that $\Sigma = (\Sigma_+, \Sigma_-)$ is a bicausal system, with $\widetilde A_+$
exponentially stable as an operator on $\cX_+$ and $\widetilde A_-$ exponentially stable as an operator on $\cX_-$.
Then:
\begin{enumerate}

\item Given any $\bu \in \ell^2_\cU({\mathbb Z})$, there is a unique $\bx_+ \in \ell^2_{\cX_+}({\mathbb Z})$ satisfying the
first system equation in \eqref{dtsystem-f}, with the resulting input-state map $T_{\Sigma_+, is}$ mapping $\ell^2_\cU({\mathbb Z})$
to $\ell^2_{\cX_+}({\mathbb Z})$ given by the block matrix
\begin{equation}  \label{TSig+is}
 [T_{\Sigma_+, is}]_{ij} =  \begin{cases} \widetilde A_+^{i-j-1} \widetilde B_+  &\text{for } i > j , \\
  0 &\text{for } i \le j.  \end{cases}
\end{equation}
The unique output signal $\by_+ \in \ell^2_\cY({\mathbb Z})$ resulting from the system equations \eqref{dtsystem-f}
with given input $\bu \in \ell^2_\cU({\mathbb Z})$ and resulting uniquely determined state trajectory $\bx_+$ in
$\ell^2_{\cX_+}({\mathbb Z})$
is then given by $\by_+ = T_{\Sigma_+} \bu$ with $T_{\Sigma_+} \colon \ell^2_\cU({\mathbb Z}) \to \ell^2_\cY({\mathbb Z})$
having block matrix representation given by
\begin{equation}   \label{TSig+}
 [ T_{\Sigma_+} ]_{ij} = \begin{cases}  \widetilde C_+ \widetilde A_+^{i-j -1} \widetilde B_+ &\text{for } i > j, \\
             \widetilde D &\text{for } i = j,  \\
             0 &\text{for } i < j.
 \end{cases}
\end{equation}
Thus $T_{\Sigma_+, is}$ and $T_{\Sigma_+}$ are block lower-triangular (causal) Toeplitz operators.

\item Given any $\bu \in \ell^2_\cU({\mathbb Z})$, there is a unique $\bx_- \in \ell^2_{\cX_-}({\mathbb Z})$ satisfying the first
system equation in \eqref{dtsystem-b}, with resulting input-state map $T_{\Sigma_-, is} \colon \ell^2_\cU({\mathbb Z}) \to
\ell^2_{\cX_+}({\mathbb Z})$ having block matrix representation given by
\begin{equation}  \label{TSig-is}
[ T_{\Sigma_-, is}]_{ij} = \begin{cases}   0 &\text{for } i>j, \\
                       \widetilde A_-^{j-i} \widetilde B_- &\text{for } i \le j.  \end{cases}
\end{equation}
The unique output signal $\by_- \in \ell^2_\cY({\mathbb Z})$ resulting from the system equations \eqref{dtsystem-b}
with given input $\bu \in \ell^2_\cU({\mathbb Z})$ and resulting uniquely determined state trajectory
$\bx_-$ in $\ell^2_{\cX_+}({\mathbb Z})$
is then given by $\by = T_{\Sigma_-} \bu$ with $T_{\Sigma_-} \colon \ell^2_\cU({\mathbb Z}) \to \ell^2_\cY({\mathbb Z})$
having block matrix representation given by
\begin{equation}   \label{TSig-}
[T_{\Sigma_-}]_{ij} = \begin{cases} 0 &\text{for } i > j, \\
    \widetilde C_- \widetilde  A_-^{j-i} \widetilde B_- &\text{for } i \le j. \end{cases}
\end{equation}
Thus $T_{\Sigma_-, is}$ and $T_{\Sigma_-}$ are upper-triangular (anticausal) Toeplitz operators.

\item The input-state map
for the combined bicausal system $\Sigma = (\Sigma_+, \Sigma_-)$ is then given by
    \[
T_{\Sigma, is } = \mat{c}{ T_{\Sigma_-, is} \\ T_{\Sigma_+, is}}
\colon \ell^2_\cU({\mathbb Z}) \to \ell^2_\cX (\BZ) =\mat{c}{\ell^2_{\cX_-}({\mathbb Z}) \\
\ell^2_{\cX_+}({\mathbb Z})}
    \]
with block matrix entries (with notation using the
natural identifications $\cX_+ \cong \sbm{ 0 \\ \cX_+ }$ and $\cX_- \cong \sbm{ \cX_- \\ 0 }$)
\begin{equation}   \label{TSigisBC}
[ T_{\Sigma, is} ]_{ij} = \begin{cases}  \widetilde A_+^{i-j-1} \widetilde B_+ &\text{for } i > j, \\
                           \widetilde A_-^{j-i} \widetilde B_- &\text{for } i \le j.
  \end{cases}
\end{equation}
Moreover, the input-output map $T_\Sigma \colon \ell^2_\cU({\mathbb Z}) \to \ell^2_\cY({\mathbb Z})$ of $\Si$ is given by
\[
T_\Sigma= T_{\Si_+}+T_{\Si_-}:\ell^2_\cU(\BZ)\to\ell^2_\cY(\BZ),
\]
having block matrix decomposition given by
\begin{equation}   \label{TSigBC}
 [ T_\Sigma ]_{ij} =  \begin{cases} \widetilde C_+ \widetilde A_+^{i-j-1} \widetilde B_+ &\text{for } i > j, \\
    \widetilde D + \widetilde C_- \widetilde B_- &\text{for } i = j, \\
    \widetilde C_- \widetilde A_-^{j-i} \widetilde B_-  &\text{for } i < j.
    \end{cases}
 \end{equation}

 \item For $\bu \in \ell^2_\cU({\mathbb Z})$ and $\by \in \ell^2_\cY({\mathbb Z})$,   $\widehat \bu$ and $\widehat \by$
 be the respective bilateral $Z$-transforms
 $$
 \bu(z) = \sum_{n=-\infty}^\infty \bu(n) z^n \in L^2_\cU({\mathbb T}), \quad
 \by(z) = \sum_{n=-\infty}^\infty \by(n) z^n \in L^2_\cU({\mathbb T}).
 $$
 Then
 $$
   \by = T_\Sigma \bu \quad \Longleftrightarrow \quad \widehat \by(z) = F_\Sigma(z) \cdot \widehat \bu(z)
   \text{ for  almost all } z \in {\mathbb T}
 $$
 where $F_\Sigma(z)$ is the {\em transfer function} of the bicausal system $\Sigma$ given by
\begin{equation}\label{FSigBiC-Laurent}
 \begin{aligned}
 F_\Sigma(z) & = \widetilde C_- (I - z^{-1} \widetilde A_-)^{-1} \widetilde B_- + \widetilde D + z \widetilde C_+
   (I - z \widetilde A_+)^{-1} \widetilde B_+\\
   & = \sum_{n=1}^\infty \widetilde C_- \widetilde A_-^n \widetilde B_- z^{-n} + (\widetilde D + \widetilde C_- \widetilde B_-)
   + \sum_{n=1}^\infty \widetilde C_+ \widetilde A_+^n \widetilde B_+ z^n.
\end{aligned}
\end{equation}
Furthermore, the Laurent operator $\fL_{F_\Sigma} \colon \ell^2_\cU({\mathbb Z}) \to \ell^2_\cY({\mathbb Z})$
associated with the function $F_\Sigma \in L^\infty_{\cL(\cU, \cY)}({\mathbb T})$ as in \eqref{Laurent-matrix}
is identical to the input-output operator $T_\Sigma$ for the bicausal system $\Sigma$
\begin{equation}   \label{Laurent=IO}
 \fL_{F_\Sigma} = T_\Sigma,
\end{equation}
and hence also, for $\bu \in \ell^2_\cU({\mathbb Z})$ and $\by \in \ell^2_\cY({\mathbb Z})$ and notation as in \eqref{MFvsLF},
\begin{equation}  \label{transfunc-propBC}
\by = T_\Sigma \bu \quad \Longleftrightarrow \quad \widehat \by(z) = F_\Sigma(z) \cdot \widehat \bu(z) \text{ for almost all }
z \in {\mathbb T}.
\end{equation}
  \end{enumerate}
\end{proposition}

\begin{proof}[\bf Proof]
We first consider item (1).  Let us rewrite the system equations \eqref{dtsystem-f} in aggregate form
\begin{equation}  \label{dtsystem-f-agg}
  \Sigma_+ \colon \left\{ \begin{array}{ccc}
  \cS^{-1} \bx_+ & = & \widetilde \cA_+ \bx_+ + \widetilde \cB_+ \bu \\
   \by_+ & = & \widetilde \cC_+ \bx_+ + \widetilde \cD \bu
   \end{array} \right.
\end{equation}
where
\begin{equation}\label{cABCD-tilde+}
\begin{aligned}
\widetilde \cA_+ &= {\rm diag}_{k \in {\mathbb Z}} [\widetilde A_+] \in \cL(\ell^2_{\cX_+}),\\
 \widetilde \cB_+ &= {\rm diag}_{k \in {\mathbb Z}} [\widetilde B_+] \in \cL(\ell^{2}_{\cU}({\mathbb Z}), \ell^{2}_{\cX_+}({\mathbb Z})),\\
 \widetilde \cC_+ &= {\rm diag}_{k \in {\mathbb Z}} [\widetilde C_+]  \in \cL(\ell^{2}_{\cX_+}({\mathbb Z}), \ell^{2}_{\cY}({\mathbb Z})),\\
 \widetilde \cD &= {\rm diag}_{k \in {\mathbb Z}} [D] \in \cL(\ell^{2}_{\cU}({\mathbb Z}), \ell^{2}_{\cY}({\mathbb Z})).
   \end{aligned}
\end{equation}
The exponential stability assumption on $\widetilde A_+$ implies that $\widetilde A_+$ has trivial exponential dichotomy (with
state-space $\cX_- = \{0\}$).
As previously observed (see \cite{BG2}), the exponential dichotomy of $\widetilde A_+$ implies that we can solve the first
system equation \eqref{dtsystem-f-agg} of $\Sigma_+$ uniquely for $\bx_+ \in \ell^2_{\cX_+}({\mathbb Z})$:
\begin{equation}  \label{bx+}
 \bx_+ = (\cS^{-1} - \widetilde \cA_+)^{-1}  \widetilde \cB_+ \bu =: T_{\Sigma_+, is} \bu
\end{equation}
and item (1) follows.  From the general formula \eqref{inverse} for $(\cS^{-1} - \cA)^{-1}$ in \eqref{inverse}, we see that for our case
here the formula for the input-state map $T_{\Sigma_+, is}$ for the system $\Sigma_+$ is given by \eqref{TSig+is}.
From the aggregate form of the system equations \eqref{dtsystem-f-agg} we see that the resulting input-output map
$T_{\Sigma_+} \colon \ell^2_\cU({\mathbb Z}) \to \ell^2_\cY({\mathbb Z})$ is then given by
$$
  T_{\Sigma_+} = \widetilde \cD + \widetilde \cC (\cS^{-1} - \widetilde \cA_+)^{-1} \widetilde \cB_+
   = \widetilde \cD + \widetilde \cC T_{\Sigma_+, is}.
$$
The block matrix decomposition \eqref{TSig+} for the input-output map $T_{\Sigma_+}$ now follows directly from
plugging in the matrix decomposition \eqref{TSig+is} for $T_{\Sigma_+, is}$ into this last formula.

The analysis for item (2) proceeds in a similar way.  Introduce operators
\begin{equation}\label{cABCD-tilde-}
\begin{aligned}
\widetilde \cA_- &= {\rm diag}_{k \in {\mathbb Z}} [\widetilde A_-] \in \cL(\ell^2_{\cX_-}),\\
 \widetilde \cB_- &= {\rm diag}_{k \in {\mathbb Z}} [\widetilde B_-] \in \cL(\ell^{2}_{\cU}({\mathbb Z}), \ell^{2}_{\cX_-}({\mathbb Z})),\\
\widetilde \cC_- &= {\rm diag}_{k \in {\mathbb Z}} [\widetilde C_-]  \in \cL(\ell^{2}_{\cX_-}({\mathbb Z}), \ell^{2}_{\cY}({\mathbb Z})),
   \end{aligned}
\end{equation}
Write the system \eqref{dtsystem-b} in the aggregate form
\begin{equation}  \label{dtsystem-b-agg}
  \Sigma_- \colon \left\{ \begin{array}{ccc}
   \bx_- & = & \widetilde \cA_- \cS^{-1} \bx_- + \widetilde \cB_- \bu \\
   \by_- & = & \widetilde \cC_- \bx_-.
   \end{array} \right.
\end{equation}
As $\widetilde  A_-$ is exponentially stable, so also is $\widetilde \cA_-$ and we may compute $(I  - \widetilde A_- \cS^{-1})^{-1}$
via the geometric series, using also that $\widetilde A_-$ and $\cS^{-1}$ commute as observed in \eqref{commute},
$$
 (I - \widetilde A_- \cS^{-1} )^{-1}  = \sum_{k=0}^\infty \widetilde A_-^k \cS^{-k}
$$
from which we deduce the block matrix representation
$$
 [ (I - \widetilde \cA_- \cS^{-1} )^{-1} ]_{ij} = \begin{cases}  0 &\text{for } i > j, \\ \widetilde A_-^{j-i} &\text{for } i \le j.
 \end{cases}
$$

We next note that we can  solve the first system equation in \eqref{cABCD-tilde-} for $\bx_-$ in terms of $\bu$:
$$
  \bx_- = (I - \cA_- \cS^{-1})^{-1} \widetilde B \bu =: T_{\Sigma_-, is} \bu.
$$
Combining this with the previous formula for the block-matrix entries for $(I - \widetilde \cA_- \cS^{-1} )^{-1}$ leads to the
formula \eqref{TSig-is} for the matrix entries of $T_{\Sigma_-, is}$.  From the second equation for the system \eqref{dtsystem-b-agg}
we see that then $\by_-$ is uniquely determined via the formula
$$
  \by_- = \widetilde \cC_-  \bx_- = \widetilde \cC_- T_{\Sigma_-, is} \bu.
$$
Plugging in the formula \eqref{TSig-is} for the block matrix entries of $T_{\Sigma_-, is}$ then leads to the formula
\eqref{TSig-} for the block matrix entries of the input-output map $T_{\Sigma_-}$ for the system $\Sigma_-$.

 Item (3) now follows by definition of the input-output map $T_\Sigma$ of the bicausal system $\Sigma$  as the sum
$ T_\Sigma = T_{\Sigma, -} + T_{\Sigma, +}$ of the input-output maps for the anticausal system $\Sigma_-$
and the causal system $\Sigma_+$ along with the formulas for $T_{\Sigma, \pm}$ obtained in items (1) and (2).

We now analyze item (4).  Define $F_\Sigma$ by either of the equivalent formulas in \eqref{FSigBiC-Laurent}.  Due to the exponential
stability of $\widetilde A_+$ and $\widetilde A_-$, we see that $F_\Sigma$ is a continuous $\cL(\cU, \cY)$-valued function
on the unit circle ${\mathbb T}$, and hence the multiplication operator $M_{F_\Sigma} \colon f(z) \mapsto
F_\Sigma(z) \cdot f(z)$ is a bounded operator from $L^2_\cU({\mathbb T})$ into $L^2_\cY({\mathbb T})$.
From the second formula for $F_\Sigma(z)$ in \eqref{FSigBiC-Laurent} combined with the formula \eqref{TSigBC} for the
block matrix entries of $T_\Sigma$, we see that the Laurent expansion for $F(z) =
\sum_{n=-\infty}^\infty F_n z^n$ on ${\mathbb T}$ is given by $F_n = [T_\Sigma]_{n,0}$ and that the  Laurent matrix
$[\fL_{F_\Sigma}]_{ij}  = F_{i-j}$ is the same as the matrix for the input-output operator $[T_\Sigma]_{ij}$.
We now see the identity \eqref{Laurent=IO} as an immediate consequence of the general identity \eqref{MFvsLF}.
 Finally, the transfer-function property
\eqref{transfunc-propBC} follows immediately from \eqref{Laurent=IO} combined with the general identity
\eqref{MFvsLF}.
\end{proof}

\begin{remark}\label{R:dichot-bicausal}
From the form of the input-output operator $T_\Si$ and transfer function $F_{\Si}$ of the dichotomous system $\Si$ in
\eqref{dtsystem}--\eqref{sysmat} that were obtained in Proposition \ref{P:TF} with respect to the decompositions of $A$ in
\eqref{Adec} and of $B$ and $C$ in \eqref{BCdec} it follows that a dichotomous system can be represented as a
bicausal system \eqref{dtsystem-f}--\eqref{dtsystem-b} with
\begin{equation}   \label{dichot-bicausal}
\begin{aligned}
(\wtilC_+,\wtilA_+,\wtilB_+,\wtilD)&=(C_+,A_+,B_+,D),\\
(\wtilC_-,\wtilA_-,\wtilB_-)&=(C_-, A_-^{-1},-A_-^{-1}B_-).
\end{aligned}
\end{equation}
The extra feature that a bicausal system coming from a dichotomous system has is that $\wtilA_-$ is invertible. In fact, if the operator
$\wtilA_-$ in a bicausal system \eqref{dtsystem-f}-\eqref{dtsystem-b} is invertible, it can be represented as a dichotomous
system \eqref{dtsystem} as well, by reversing the above transformation. Indeed, one easily verifies that if $\Si =
(\Sigma_+, \Sigma_-)$ is a bicausal system given by \eqref{dtsystem-f}-\eqref{dtsystem-b} with $\wtil{A}_-$ invertible,
then the system \eqref{dtsystem} with
\[
A=\mat{cc}{\wtilA_+ &0 \\ 0& \wtilA_-^{-1}},\quad B=\mat{c}{\wtilB_+\\-\wtilA_-^{-1}\wtilB_-},\quad C=\mat{cc}{\wtilC_+ & \wtilC_-},
\quad D=\wtilD
\]
is a dichotomous system whose input-output operator and transfer function are equal the input-output operator and transfer function
 from the original bicausal system.
\end{remark}

To a large extent, the theory of dichotomous system presented in Section \ref{S:dichotsys} carries over to bicausal systems,
with proofs that can be directly obtained from the translation between the two systems given above.
We describe here the main features.

The Laurent operator $\fL_{F_\Si}=T_\Si$ can again be decomposed as in \eqref{dichotLaurent} where now the Toeplitz operators
$\widetilde \fT_{F_{\Sigma}}$ and $\fT_{F_{\Sigma}}$ are given by
\begin{equation}\label{Toeplitz-nonreg}
\begin{aligned}
\mbox{  }
 [\fT_{F_{\Sigma}} ]_{ij \colon i<0, j<0}  &=
    \begin{cases} \widetilde C_{-} \widetilde A_{-}^{j-i} \widetilde
	B_{-} &\text{for } i< j < 0, \\
	\wtilD + \wtilC_-\wtilB_- &\text{for } i=j < 0, \\
	\wtilC_{+} \wtilA_{+}^{i-j-1} \wtilB_{+} &\text{for } j< i < 0, \end{cases}\\
[\fT_{F_{\Sigma}} ]_{ij \colon i\ge 0, j \ge 0} & = \begin{cases}
\widetilde C_{-} \widetilde A_{-}^{j-i} \widetilde B_{-} &\text{for }
0 \le i < j, \\
\wtilD + \wtilC_-\wtilB_- &\text{for } 0\le i=j, \\
\wtilC_{+}  \wtilA_{+}^{i-j-1} \wtilB_{+} &\text{for } 0\leq j<i, \end{cases}
\end{aligned}
\end{equation}
while the Hankel operators $\widetilde \fH_{F_{\Sigma}}$ and $\fH_{F_{\Sigma}}$ are given by
\begin{equation}
[ \widetilde \fH_{F_{\Sigma}} ]_{ij: i<0, j \ge 0}  =
 \widetilde C_{-} \widetilde A_{-}^{j-i} \widetilde B_{-},  \quad
[ \fH_{F_{\Sigma}} ]_{ij \colon i \ge 0, j< 0}  = \wtilC_{+}
\wtilA_{+}^{i-j-1} \wtilB_{+}. \label{Hankel-nonreg}
\end{equation}

For the subsystems $\Si_+$ and $\Si_-$ we define controllability operators $\bW_{c}^{+}$ and
$\bW_{c}^{-}$, respectively, as well as observability operators $\bW_{o}^{+}$ and
$\bW_{o}^{-}$, respectively, just as in the case of regular forward-time and backward-time systems:
\begin{equation}\label{con-obsdichot-nonreg}
\begin{aligned}
\bW_{c}^{+} &= \operatorname{row}_{j \in {\mathbb Z}_{-}} [\wtilA_{+}^{-j-1}
\wtilB_{+} ] \colon \ell^{2}_{\cU}({\mathbb Z}_{-}) \to \cX_{+},\\
\bW_{c}^{-} &= \operatorname{row}_{j \in {\mathbb Z}_{+}} [
 \widetilde A_{-}^{j} \widetilde B_{-} ] \colon
 \ell^{2}_{\cU}({\mathbb Z}_{+}) \to \cX_{-},\\
\bW_{o}^{+} &= \operatorname{col}_{i \in {\mathbb Z}_{+}} [ \wtilC_{+}
 \wtilA_{+}^{i}] \colon \cX_{+} \to \ell^{2}_{\cY}({\mathbb Z}_{+}),\\
\bW_{o}^{-} &= \operatorname{col}_{i \in {\mathbb Z}_{-}} [
 \widetilde C_{-} \widetilde A_{-}^{-i}] \colon \cX_{-} \to
 \ell^{2}_{\cY}({\mathbb Z}_{-}).
 \end{aligned}
\end{equation}
Setting $\cX=\cX_+ \dot{+} \cX_-$, we put these operators together to define the controllability operator $\bW_c$ and
observability operator $\bW_o$ of the bicausal system $\Si$ via
\begin{equation}\label{bicausalconobs}
\begin{aligned}
\bW_c&=\mat{cc}{\bW_c^+ & \bW_c^-}:\mat{c}{\ell^2_\cU(\BZ_-)\\\ell^2_\cU(\BZ_+)}\to\cX,\\
\bW_o&=\mat{c}{\bW_o^- \\ \bW_o^+}:\cX\to \mat{c}{\ell^2_\cY(\BZ_-)\\ \ell^2_\cY(\BZ_+)}.
\end{aligned}
\end{equation}
The fact that $\wtilA_+$ and $\wtilA_-$ are both stable implies that all the operators $\bW_c^{\pm}$, $\bW_c$,
$\bW_o^\pm$ and $\bW_o$ are bounded. Then it is now easily checked that we still recover factorizations of the
Hankel operators as before:
\begin{equation} \label{HankelfactBicausal}
    \widetilde \frakH_{F_{\Si}} = \bW_{o}^{-} \bW_{c}^{-}, \quad
    \frakH_{F_{\Si}} = \bW_{o}^{+} \bW_{c}^{+}.
\end{equation}

For the bicausal system $\Si = (\Si_+, \Si_-)$ given by \eqref{dtsystem-f}--\eqref{dtsystem-b} we say that $\Si$ is
{\em controllable} (respectively, {\em observable}) whenever the two systems $\Si_+$ and $\Si_-$ are both
controllable (respectively, observable), i.e., $\im \bW_c^+$ dense in $\cX_+$ and $\im \bW_c^-$ dense in $\cX_-$,
or equivalently, $\im \bW_c$ dense in $\cX$ (respectively, $\kr \bW_o^+=\{0\}$ and $\kr \bW_o^-=\{0\}$, or equivalently,
$\kr \bW_o=\{0\}$). Analogously, we say that $\Si$ is {\em $\ell^2$-exactly controllable} (respectively,
{\em $\ell^2$-exactly observable}) whenever $\im \bW_c=\cX$ (respectively, $\im \bW_o^*=\cX$).

An exception to the general rubric that the theory of dichotomous systems carries over directly to the theory of
bicausal systems is the following analogue of
the $\ell^2$-admissible-trajectory interpolation result (Proposition \ref{P:traj-int}), which has a somewhat different form
for the bicausal setting.

\begin{proposition} \label{P:traj-int-bi}   Let $\Sigma = (\Sigma_-, \Sigma_+)$ be a bicausal linear system as in \eqref{dtsystem-b}
and \eqref{dtsystem-f} with $\widetilde A_-$ exponentially stable on $\cX_-$ and $\widetilde A_+$ exponentially stable on
$\cX_+$ and suppose that $\Sigma$ is $\ell^2$-exactly controllable in the bicausal sense.
Then given any vectors $x_- \in \cX_-$, $x_+ \in \cX_+$, $u \in \cU$,   there is an $\ell^2$-admissible system trajectory
$(\bu, \bx_- \oplus \bx_+, \by)$ for $\Sigma$ satisfying the  interpolation conditions
\begin{equation}   \label{traj-int-bi}
  \bx_-(1) = x_-, \quad \bx_+(0) = x_+, \quad \bu(0) = u.
\end{equation}
\end{proposition}

\begin{proof}[\bf Proof]
By the $\ell^2$-exact controllability assumption, we can find $\bu_- \in \ell^2_\cU({\mathbb Z}_-)$ so that $\bW^+_c \bu = x_+$.
Similarly, we can find $\bu_+ \in \ell^2_\cU(\BZ_+)$ so that $\bW^-_c \bu_+ = x_-$.  Define a input signal $\bu \in \ell^2_\cU({\mathbb Z})$ by
$$
  \bu(n) = \begin{cases}  \bu_-(n) &\text{if } n < 0, \\
        u &\text{if } n=0, \\
        \bu_+(n-1) &\text{if } n > 0.   \end{cases}
$$
Now it is simple direct check that the $\ell^2$-admissible trajectory $(\bu, \bx, \by)$ determined by the input $\bu$ has the desired interpolation properties \eqref{traj-int-bi}.

The proof is close to that of the the corresponding result for the dichotomous setting, Proposition \ref{P:traj-int}. The key difference is that we must use $\bx_-(1)$ rather than $\bx_-(0)$ as a free parameter since in general we are not able to solve the equation $x_- = \widetilde A_- \bx_-(1) -\widetilde A_-\widetilde B_-  u$ for $\bx_-(1)$ (the analogue of
equation \eqref{solution}) since $\widetilde A_-$ need not be invertible in the bicausal setting.
\end{proof}

\section{Storage functions}
\label{S:storage}

Let $\Si$ be the dichotomous system given by \eqref{dtsystem}. A {\em storage function} for the system $\Si$ is a function
$S\colon \cX\to\BR$ so that
\begin{enumerate}
    \item $S$ is continuous at $0$:
\[
\{x_{n}\}_{n \in {\mathbb N}} \subset \cX,\
    \lim_{n \to \infty} x_{n} = 0\mbox{ in }\cX\ \Longrightarrow\ \lim_{n
    \to \infty} S(x_{n}) = S(0) \mbox{ in }{\mathbb R},
\]
\item $S$ satisfies the energy-balance relation:
\begin{equation}   \label{EB}
     S(\bx(n+1)) - S(\bx(n)) \le \| \bu(n)\|^{2}_{\cU} - \|
     \by(n)\|^{2}_{\cY} \qquad (n\in\BZ)
\end{equation}
along all $\ell^2$-admissible system trajectories $(\bu, \bx, \by)$ of $\Si$, and
\item $S$ satisfies the normalization condition $ S(0) = 0$.
\end{enumerate}
We further say that $S$ is a {\em strict storage function} for
$\Sigma$ if
$S$ is a storage function for $\Sigma$ with condition \eqref{EB}
replaced by the stronger condition: {\em there is a $\epsilon > 0$ so
that
\begin{equation}  \label{EBstrict}
    S(\bx(n+1)) - S(\bx(n))  + \epsilon^2 \| \bx(n) \|^2
    \le (1 - \epsilon^2) \| \bu(n) \|^{2}_{\cU} - \| \by(n) \|^{2}_{\cY} \ \  (n\in\BZ)
\end{equation}
along all $\ell^2$-admissible system trajectories $(\bu, \bx, \by)$ of $\Si$.}

Then we have the following result.

\begin{proposition}   \label{P:dichotStoragefunc}
Suppose that the dichotomous system \eqref{dtsystem} has a storage
function $S$.  Then the input-output map $T_{\Sigma}$ is
contractive, i.e., $\|T_{\Sigma}\| \le 1$. In case $\Sigma$ has a strict storage function $S$, the input-output map is a
strict contraction, i.e., $\|T_{\Sigma}\| < 1$.
 \end{proposition}

 \begin{proof}[\bf Proof] Let $S$ be a storage function for $\Sigma$. Take $\bu\in\ell^2_\cU(\BZ)$. Define $\bx$ by
 $\bx = T_{\Sigma, is} \bu$ and $\by$ by $\by = T_\Sigma \bu$, where $T_{\Sigma, is}$ and $T_\Sigma$ are as in
 \eqref{ISmap}--\eqref{IOmap}, so $(\bu, \bx,  \by)$ is an $\ell^2$-admissible system trajectory.
 If we sum \eqref{EB} from $n=-N$ to $n=N$ we get
 $$
S(\bx(N+1)) - S(\bx(-N))
\le \sum_{n=-N}^{N} \| \bu(n) \|^{2}_{\cU} -
\sum_{n=-N}^{N} \| \by(n) \|^{2}_{\cY}.
$$
Taking the limit as $N \to \infty$ and using the fact that both $\bx(-N) \to 0$ and $\bx(N) \to 0$ as $N \to \infty$, since
$\bx \in \ell^2_\cX({\mathbb Z})$, we obtain from the continuity of $S$ at 0  and the normalization condition $S(0) = 0$ that
both $S(\bx_{N+1}) \to 0$ and $S(\bx(-N)) \to 0$ as $N \to \infty$.  Hence taking the limit as $N \to \infty$ in the preceding estimate
gives
$$
0 \le  \| \bu \|^{2}_{\ell^{2}_{\cU}({\mathbb Z})} -
\| \by \|^{2}_{\ell^{2}_{\cY}({\mathbb Z})}
=\| \bu \|^{2}_{\ell^{2}_{\cU}({\mathbb Z})} -
\| T_\Si\bu \|^{2}_{\ell^{2}_{\cY}({\mathbb Z})}.
$$
Since $\bu$ was chosen arbitrarily in $\ell^2_\cU(\BZ)$, it follows that $\|T_{\Sigma}\| \le 1$.

If $\Sigma$ has a strict storage function we see that there is an $\epsilon > 0$ so that
\begin{align*}
S(\bx(n+1)) - S(\bx(n)) & \le S(\bx(n+1)) - S(\bx(n)) + \epsilon^2 \| \bx(n) \|^2  \\
& \le (1 - \epsilon^2) \| \bu(n) \|^2 - \| \by(n) \|^2
\end{align*}
so in particular we have
$$
 S(\bx(n+1) - S(\bx(n) ) \le (1 - \epsilon^2) \| \bu(n) \|^2 - \| \by(n)\|^2.
 $$
 Summing this last inequality from $n=-N$ to $n = N$ leaves us with
$$
S(\bx(N+1)) - S(\bx(-N)) \le (1- \epsilon^2) \sum_{n=-N}^{N} \| \bu(n) \|^{2}_{\cU} -
\sum_{n=-N}^{N} \| \by(n) \|^{2}_{\cY}.
$$
Taking the limit as $N \to \infty$ and using again the fact that both
$\bx(-N) \to 0$ and $\bx(N) \to 0$ as $N \to \infty$ along $\ell^2$-admissible system
trajectories then gives us
$$
0 \le
(1- \epsilon^2) \| \bu \|^{2}_{\ell^{2}_{\cU}({\mathbb Z})} -
\| \by \|^{2}_{\ell^{2}_{\cY}({\mathbb Z})}
=(1- \epsilon^2) \| \bu \|^{2}_{\ell^{2}_{\cU}({\mathbb Z})} -
\| T_\Si\bu \|^{2}_{\ell^{2}_{\cY}({\mathbb Z})}
$$
and we are able to conclude that $\| T_{\Sigma} \|^2 \le 1- \epsilon^2 < 1$.
\end{proof}

To get further results on storage functions for dichotomous systems, we shall assume from now on that the
transfer function $F_\Sigma$ is contractive on the unit circle ($\| F_\Sigma \|_{\infty, {\mathbb T}} \le 1$)
 as well as that $\Sigma$ is dichotomously $\ell^2$-exactly minimal
(see \eqref{dichot-exact-con}--\eqref{dichot-exact-obs}),
i.e.,
\begin{equation}\label{Con}
\|F_\Si\|_{\infty,\mathbb T} \le 1, \quad \im \bW_c = \cX, \quad
\im \bW_o^* = \cX.
\end{equation}

\begin{remark}  \label{R:local-storage}
A particular consequence of assumption \eqref{Con} is that $\Sigma$ is $\ell^2$-exactly controllable. As a consequence of
Proposition \ref{P:traj-int} we then see
that the second condition  \eqref{EB} in the definition of storage function can be replaced by the localized version:  {\em given
$u \in \cU$ and $x \in \cX$ we have the inequality}
\begin{equation}   \label{EB'}
S( Ax + Bu ) - S(x) \le \| u \|^2 - \| Cx + Du \|^2
\end{equation}
for the standard case, and
\begin{equation}   \label{EBstrict'}
S(Ax + Bu) - S(x) + \epsilon^2 \| x \|^2 \le (1 - \epsilon^2) \| u \|^2 - \| C x + D u \|^2
\end{equation}
for the strict case.  Once we have this formulation, we also see that we could equally well replace the phrase
{\em $\ell^2$-admissible system trajectories} simply with {\em system trajectories} in \eqref{EB} and \eqref{EBstrict}.
\end{remark}

With all the assumptions \eqref{Con} in force, we now define two candidate storage functions, referred to as the
available storage and required supply functions for the dichotomous linear system \eqref{dtsystem}, namely
\begin{align}
    S_{a}(x_{0}) &  =  \sup_{ \bu \in \ell^{2}_{\cU}({\mathbb Z}) \colon
    \bW_{c} \bu = x_{0}} \sum_{n=0}^{\infty} \left( \| \by(n)
    \|^{2} - \| \bu(n) \|^{2} \right) \qquad (x_o\in \im \bW_c)  \label{dichotSa} \\
    S_{r}(x_{0}) &  =  \inf_{ \bu \in \ell^{2}_{\cU}({\mathbb Z}) \colon
    \bW_{c} \bu = x_{0}}
 \sum_{n=-\infty}^{-1} \left( \| \bu(n)\|^{2} - \| \by(n) \|^{2} \right) \quad (x_o\in \im \bW_c)
 \label{dichotSr}
\end{align}
where $\by$ is the output signal determined by \eqref{IOmap}.

In order to show that $S_a$ and $S_r$ are storage functions, we shall need multiple applications of the following elementary
patching lemma.

\begin{lemma} \label{L:patch}  Suppose that $(\bu', \bx', \by')$ and $(\bu'', \bx'', \by'')$ are two $\ell^2$-admissible system trajectories
of the system $\Sigma$ such that $\bx'(0) = \bx''(0) =: x_0$.  Define a new triple of signals $(\bu, \bx, \by)$ by
\begin{align}
& \bu(n) = \begin{cases} \bu'(n) &\text{if } n< 0, \\
             \bu''(n) &\text{if } n \ge 0, \end{cases}      \quad
 \bx(n) = \begin{cases} \bx'(n) &\text{if } n \le 0, \\
               \bx''(n) &\text{if } n > 0, \end{cases}  \notag  \\
&
\by(n)  = \begin{cases}  \by'(n) & \text{if } n < 0, \\
   \by''(n) &\text{if } n \ge 0.  \end{cases}
   \label{patch}
\end{align}
Then $(\bu, \bx, \by)$ is again an $\ell^2$-admissible system trajectory.
\end{lemma}

\begin{proof}[{\bf Proof.}]  We must verify that $(\bu, \bx, \by)$ satisfy the system equations \eqref{dtsystem} for all $n \in {\mathbb Z}$.
For $n<0$ this is clear since $(\bu', \bx', \by')$ is a system trajectory.  Since $\bx'(0) = \bx''(0)$, we see that this holds for $n=0$.
That it holds for $n>0$ follows easily from the fact that $(\bu'', \bx'', \by'')$ is a system trajectory.  Finally note that $(\bu', \bx', \by')$
and $(\bu'', \bx'', \by'')$ both being $\ell^2$-admissible implies that $(\bu, \bx, \by)$ is $\ell^2$-admissible.
\end{proof}

Our next goal is to show that $S_a$ and $S_r$ are storage functions for $\Si$, and among all storage functions they are the
minimum and maximum ones.
We postpone the proof of Step 4 in the proof of items (1) and (2)
in the following proposition to Section \ref{S:SaSr} below.

\begin{proposition}   \label{P:dichotSa}
    Let $\Sigma$ be a dichotomous linear system as in \eqref{dtsystem} such that \eqref{Con} holds.   Then:
 \begin{enumerate}
     \item $S_{a}$ is a storage function for $\Sigma$.
     \item  $S_{r}$ is a storage function for $\Sigma$.
    \item  If $\widetilde S$ is any other storage function for
    $\Sigma$, then
  $$
  S_{a}(x_{0}) \le \widetilde S(x_{0}) \le S_{r}(x_{0}) \text{ for all }
  x_{0} \in \cX.
$$
\end{enumerate}
\end{proposition}

\begin{proof}[\bf Proof of (1) and (2)]   The proof proceeds in several steps.

\paragraph{\bf Step 1: $S_{a}$ and $S_{r}$ are finite-valued on $\cX$.}
Let $x_0\in \cX= \im\bW_c$. By the $\ell^2$-exact controllability assumption, there is a $\bu_0 \in \ell^2_\cU({\mathbb Z})$
so that $x_0=\bW_c \bu_0$. Let $(\bu_0, \bx_0, \by_0)$ be the  unique $\ell^2$-admissible system trajectory of $\Si$
defined by the input $\bu_0$. Then
\[
S_a(x_0)\geq \sum_{n=0}^\infty \|\by_0(n)\|^2-\|\bu_0(n)\|^2\geq -\|\bu_0\|^2>-\infty
\]
and similarly
\[
S_r(x_0)\leq \|\bu_0\|^2<\infty.
\]
It remains to show $S_a(x_0)<\infty$ and $S_r(x_0)>-\infty$.

Let $(\bu, \bx, \by)$ be any $\ell^2$-admissible system trajectory of $\Sigma$ with $\bx(0) =
x_{0}$.  Let $(\bu_{0}, \bx_{0}, \by_{0})$ be the particular $\ell^2$-admissible system trajectory with $\bx(0) = x_0$ as chosen above.
Then by Lemma \ref{L:patch} we may piece together these two trajectories to form a new $\ell^2$-admissible system trajectory
 $(\bu',  \bx', \by')$ of $\Si$ defined as follows:
 \begin{align*}
& \bu'(n) = \begin{cases} \bu_{0}(n) &\text{if } n<0 \\
 \bu(n) &\text{if } n \ge 0 \end{cases}, \quad
 \bx'(n) = \begin{cases} \bx_{0}(n) &\text{if } n \le 0 \\
 \bx(n) &\text{if } n>0 \end{cases},  \\
 & \by'(n) = \begin{cases} \by_{0}(n) &\text{if } n < 0, \\
  \by(n) &\text{if } n \ge 0 \end{cases}.
 \end{align*}
 Since $\|T_\Si\| \le 1$ and $(\bu', \bx', \by')$ is a system trajectory,
 we know that
\[
\sum_{n=-\infty}^{+\infty} \|  \by'(n) \|^{2}=\|\by'\|^2 = \|T_\Si \bu'\|^2 \le
 \|\bu'\|^2 =\sum_{n=- \infty}^{ +\infty} \| \bu'(n) \|^{2}.
\]
Let us rewrite
 this last inequality in the form
 $$
  \sum_{n=0}^{\infty}\| \by(n) \|^{2} - \sum_{n=0}^{\infty} \| \bu(n)
  \|^{2} \le \sum_{n=-\infty}^{-1} \| \bu_{0}(n) \|^{2} -
  \sum_{n=-\infty}^{-1} \| \by_{0}(n) \|^{2} < \infty.
 $$
 It follows that the supremum of the left hand side over all $\ell^2$-admissible trajectories $(\bu,\bx,\by)$ of $\Si$ with
 $\bx(0)=x_0$ is finite, i.e., $S_{a}(x_{0}) < \infty$.

 A similar argument shows that $S_{r}(x_{0}) > -\infty$ as follows.  Given
 an arbitrary $\ell^2$-admissible system trajectory $(\bu, \bx, \by)$ with $\bx(0) =
 x_{0}$, Lemma \ref{L:patch} enables us to form the composite $\ell^2$-admissible system trajectory
 $(\bu'', \bx'', \by'')$ of $\Si$ defined by
 \begin{align*}
&  \bu''(n) = \begin{cases} \bu(n) &\text{if } n<0 \\
 \bu_{0}(n) &\text{if } n \ge 0 \end{cases}, \quad
 \bx''(n) = \begin{cases} \bx(n) &\text{if } n \le 0 \\
 \bx_{0}(n) &\text{if } n>0 \end{cases},   \\
& \by''(n) = \begin{cases} \by(n) &\text{if } n < 0, \\
  \by_{0}(n) &\text{if } n \ge 0 \end{cases}.
 \end{align*}
 Then the fact that $\sum_{n=-\infty}^{+\infty} \| \by''(n) \|^{2}
 \le \sum_{n=-\infty}^{\infty} \| \bu''(n) \|^{2}$ gives us that
 $$
 \sum_{n=-\infty}^{-1} \| \bu(n) \|^{2} - \sum_{n=-\infty}^{-1}
 \|\by(n) \|^{2} \ge \sum_{n=0}^{\infty} \| \by_{0}(n) \|^{2} -
 \sum_{n=0}^{\infty} \| \bu_{0}(n) \|^{2} > -\infty
 $$
 and it follows from the definition \eqref{dichotSr} that
 $S_{r}(x_{0}) > -\infty$. By putting all these pieces together we
 see that both $S_{a}$ and $S_{r}$ are finite-valued on $\cX = \im \bW_c$.

\paragraph{\bf Step 2: $S_{a}(0) = S_{r}(0) = 0$.} This fact follows from the explicit
quadratic form for $S_a$ and $S_r$ obtained in Theorem \ref{T:SaSr-Con} below, but we include here
an alternative more conceptual proof to illustrate the ideas.  By noting that
    $(0,0,0)$ is an $\ell^2$-admissible system trajectory, we see from the
    definitions of $S_a$ in \eqref{dichotSa} and $S_r$ in \eqref{dichotSr} that $S_{a}(0) \ge 0$ and
    $S_{r}(0) \le 0$. Now let
    $(\bu, \bx, \by)$ be any $\ell^2$-admissible system trajectory such that $\bx(0) = 0$.
    Another application of Lemma \ref{L:patch} then implies that $(\bu', \bx', \by')$ given by
    $$
    (\bu'(n), \bx'(n), \by'(n)) =
    \begin{cases} (0,0,0) &\text{if } n < 0, \\
	(\bu(n), \bx(n), \by(n)) &\text{if } n \ge 0
	\end{cases}
    $$
    is also an $\ell^2$-admissible system trajectory.  From the assumption
    that $\|T_{\Sigma} \| \le 1$ we get that
    $$
    0 \le \sum_{n=-\infty}^{\infty}( \| \bu'(n) \|^{2}_{\cU} - \|
\by'(n)
    \|^{2}_{\cY}) = \sum_{n=0}^{\infty} (\| \bu(n)\|^{2}_{\cU} - \|
    \by(n) \|^{2}_{\cY}),
    $$
    so
    \begin{equation}  \label{Sa0}
\sum_{n=0}^{\infty} (\| \by(n) \|^{2}_{\cY} - \|\bu(n)\|^{2}_{\cU})
\le 0
    \end{equation}
    whenever $(\bu, \bx, \by)$ is an $\ell^2$-admissible system trajectory
    with $\bx(0) = 0$.
    From the definition in \eqref{dichotSa} we see
    that $S_{a}(0)$ is the supremum over all such expressions on the
    left hand side of \eqref{Sa0}, and we conclude that $S_{a}(0) \le 0$.
    Putting this together with the first piece above gives $S_{a}(0) =0$.

    Similarly, note that if $(\bu, \bx, \by)$ is an $\ell^2$-admissible trajectory
    with $\bx(0) = \bW_c \bu = 0$, then again by Lemma \ref{L:patch}
    $$
    (\bu''(n), \bx''(n), \by''(n)) = \begin{cases} (\bu(n), \bx(n),
    \by(n)) &\text{if } n<0, \\ (0,0,0) &\text{if } n \ge 0
    \end{cases}
    $$
 is also an $\ell^2$-admissible system trajectory. Since $\| T_{\Sigma} \|
    \le 1$ we get
    $$
    0 \le \sum_{n=-\infty}^{\infty} (\| \bu''(n) \|^{2}_{\cU} - \|
    \by''(n) \|^{2}_{\cY}) = \sum_{n=-\infty}^{-1}
( \| \bu(n) \|^{2}_{\cU}  - \| \by(n) \|^{2}_{\cY})
$$
From the definition \eqref{dichotSr} of $S_{r}(0)$ it follows
that $S_{r}(0) \ge 0$.  Putting all these pieces together, we
arrive at $S_{a}(0) = S_{r}(0) = 0$.

\paragraph{\bf Step 3: Both $S_{a}$ and $S_{r}$ satisfy the energy balance
    inequality \eqref{EB}.}
For $x_0\in\cX$, set
\[
\ora{\cU}_{x_0}=\{\wtil{\bu}\in \ell^2_\cU(\BZ) \colon \wtil{\bx}(0)=W_c \wtil{\bu}=x_0\},
\]
Also, for any Hilbert space $\cW$ let $P_+$ on $\ell^2_\cW(\BZ)$ be the orthogonal projection on $\ell^2_\cW(\BZ_+)$ and $P_-=I-P_+$. Then we can write $S_a(x_0)$ and $S_r(x_0)$ as
\[
S_a(x_0)=\sup_{\wtil\bu \in \ora{\cU}_{x_0}}\|P_+ \wtil{\by}\|^2-\|P_+ \wtil{\bu}\|^2
\ands
S_r(x_0)=\inf_{\wtil\bu \in \ora{\cU}_{x_0}}\|P_- \wtil{\bu}\|^2-\|P_- \wtil{\by}\|^2,
\]
where $\wtil{\by} = T_\Sigma \wtil{\bu}$ is the output of $\Si$ defined by the input $\wtil{\bu} \in
\ell^2_\cU({\mathbb Z})$. In general,
 if $\wtil{\bu}\in \ell^2_\cU(\BZ)$ is an input trajectory, then the corresponding uniquely determined $\ell^2$-admissible
 state and output trajectories are denoted by $\wtil{\bx}:= T_{\Sigma, is} \wtil{\bu}$ and $\wtil{\by} := T_\Sigma \wtil{\bu}$.

Now let $(\bu, \bx, \by)$ be an arbitrary fixed system trajectory for the dichotomous system $\Sigma$
and fix $n\in \BZ$. Set
\[
\ora{\cU}_*=\{\wtil{\bu}\in\ell^2_\cU(\BZ)\colon \wtil{\bx}(0)=\bx(n),\ \wtil{\bu}(0)=\bu(n)\}.
\]
Note that $\ora{\cU}_* $ is nonempty
by simply quoting Proposition \ref{P:traj-int}.
Observe that $\ora{\cU}_*\subset \ora{\cU}_{\bx(n)}$.
For an $\ell^2$-admissible system trajectory $(\wtil{\bu},\wtil{\bx},\wtil{\by})$ with $\wtil{\bu}\in \ora{\cU}_*$ we have
$\wtil{\by}(0)=\by(n)$ and $\wtil{\bx}(1)=\bx(n+1)$. Furthermore, $(\cS^*\wtil{\bu},\cS^*\wtil{\bx},\cS^*\wtil{\by})$
is also an $\ell^2$-admissible system trajectory of $\Si$ and
$$
(\cS^* \wtil{\bx})(0)=\wtil{\bx}(1)=\bx(n+1).
$$
 Hence
\[
\cS^* \ora{\cU}_* =\{\cS^* \wtil{\bu}\colon \wtil{\bu}\in \ora{\cU}_*\}\subset \ora{\cU}_{\bx(n+1)}.
\]
Next, since $\wtil{\by}(0)=\by(n)$ and $\wtil{\bu}(0)=\bu(n)$ we have
\[
\|P_+ \wtil{\by}\|^2-\|P_+ \wtil{\bu}\|^2
= \|P_+ \cS^* \wtil{\by}\|^2-\|P_+ \cS^*\wtil{\bu}\|^2 +\|\by(n)\|^2 -\|\bu(n)\|^2
\]
and
\[
\|P_- \cS^* \wtil{\bu}\|^2 - \|P_- \cS^* \wtil{\by}\|^2
= \|P_- \wtil{\bu}\|^2 - \|P_- \wtil{\by}\|^2 + \|\bu(n)\|^2 - \|\by(n)\|^2.
\]
We thus obtain that
\begin{align*}
S_a(\bx(n))
&=\sup_{\wtil{\bu}\in \ora{\cU}_{\bx(n)}} \|P_+\wtil{\by}\|^2 -\|P_+\wtil{\bu}\|
\geq \sup_{\wtil{\bu}\in \ora{\cU}_{*}} \|P_+\wtil{\by}\|^2 -\|P_+\wtil{\bu}\|^2\\
&=\|\by(n)\|^2 -\|\bu(n)\|^2
+\sup_{\wtil{\bu}\in \ora{\cU}_{*}} \|P_+\cS^*\wtil{\by}\|^2 -\|P_+\cS^*\wtil{\bu}\|^2\\
&=\|\by(n)\|^2 -\|\bu(n)\|^2
+\sup_{\wtil{\bu}\in \cS^*\ora{\cU}_{*}} \|P_+\wtil{\by}\|^2 -\|P_+\wtil{\bu}\|^2,
\end{align*}
and similarly for $S_r$ we have
\begin{align*}
S_r(\bx(n+1))
&=\inf_{\wtil\bu \in \ora{\cU}_{\bx(n+1)}}\|P_- \wtil{\bu}\|^2-\|P_- \wtil{\by}\|^2 \\
& =\inf_{\cS^*\wtil\bu \in \ora{\cU}_{\bx(n+1)}}\|P_- \cS^*\wtil{\bu}\|^2-\|P_- \cS^*\wtil{\by}\|^2\\
&\leq \inf_{\cS^*\wtil\bu \in \cS^*\ora{\cU}_{*}}\|P_- \cS^*\wtil{\bu}\|^2-\|P_- \cS^*\wtil{\by}\|^2 \\
& =\inf_{\wtil\bu \in \ora{\cU}_{*}}\|P_- \cS^*\wtil{\bu}\|^2-\|P_- \cS^*\wtil{\by}\|^2\\
&=\|\bu(n)\|^2 - \|\by(n)\|^2 +
\inf_{\wtil\bu \in \ora{\cU}_{*}} \|P_- \wtil{\bu}\|^2 - \|P_- \wtil{\by}\|^2.
\end{align*}
To complete the proof of this step it remains to show that
\begin{align}
& S_a(\bx(n+1))=\sup_{\wtil{\bu}\in \cS^*\ora{\cU}_{*}} \|P_+\wtil{\by}\|^2 -\|P_+\wtil{\bu}\|^2 =: s_a, \notag \\
& S_r(\bx(n))=\inf_{\wtil\bu \in \ora{\cU}_{*}} \|P_- \wtil{\bu}\|^2 - \|P_- \wtil{\by}\|^2 =: s_r.
\label{SaSrRed}
\end{align}

We start with $S_a$. Since $\cS^*\ora{\cU}_{*}\subset \ora{\cU}_{\bx(n+1)}$ we see that
$$
   s_a \le S_a(\bx(n+1)).
$$
To show that also $s_a \ge S_a(\bx(n+1))$,  let $(\wtil{\bu},\wtil{\bx},\wtil{\by})$ be an $\ell^2$-admissible system trajectory
with $\wtil{\bu}\in \ora{\cU}_{\bx(n+1)}$.   The problem is to show
\begin{equation}  \label{toshow1}
\sum_{n=0}^\infty ( \| \widetilde \by(n) \|^2 - \| \widetilde \bu(n) \|^2 ) \le s_a.
\end{equation}
Toward this goal,
let $(\what{\bu},\what{\bx},\what{\by})$ be any $\ell^2$-admissible trajectory with $\what{\bu}\in \cS^*\ora{\cU}_{*}$.
We then patch the two system trajectories together by setting
\begin{align*}
& \wtil{\bu}'(k) = \begin{cases} \what{\bu}(k) &\text{if } k<0 \\
\wtil{\bu}(k) &\text{if } k \ge 0 \end{cases}, \quad
\wtil{\bx}'(k) = \begin{cases} \what{\bx}(k) &\text{if } k \le 0 \\
\wtil{\bx}(k) &\text{if } k>0 \end{cases},  \\
& \wtil{\by}'(k) = \begin{cases} \what{\by}(k) &\text{if } k < 0, \\
\wtil{\by}(k) &\text{if } k \ge 0 \end{cases}.
\end{align*}
Clearly the input, state and output trajectories are all $\ell^2$-sequences. Note that
$(\what{\bu},\what{\bx},\what{\by})$ and $(\wtil{\bu},\wtil{\bx},\wtil{\by})$ are both $\ell^2$-admissible system trajectories.
Note that
$\what{\bu}(-1)=\bu(n)$, $\what{\bx}(-1)=\bx(n)$, $\what{\by}(-1)=\by(n)$ and $\what{\bx}(0)=\bx(n+1)$, we see that
$$
\widehat \bx(0) = A \widehat \bx(-1) + B \widehat \bu(-1)
= A \bx(n) + B \bu(n) = \bx(n+1) = \widetilde \bx(0).
$$
We can now apply once again Lemma \ref{L:patch} to conclude that $(\wtil{\bu}',\wtil{\bx}',\wtil{\by}')$ is also an
$\ell^2$-admissible trajectory for $\Sigma$.
Furthermore, we have $\wtil{\bu}'\in \cS^* \ora{\cU}_{*}$, $P_+ \wtil{\by}=P_+ \wtil{\by}'$ and $P_+ \wtil{\bu}=P_+ \wtil{\bu}'$. Thus
\[
\|P_+\wtil{\by}\|^2 -\|P_+\wtil{\bu}\|^2
=\|P_+\wtil{\by}'\|^2 -\|P_+\wtil{\bu}'\|^2
\leq\sup_{\wtil{\bu}\in \cS^*\ora{\cU}_{*}} \|P_+\wtil{\by}\|^2 -\|P_+\wtil{\bu}\|^2 =: s_a.
\]
Taking the supremum on the left-hand side over all $\ell^2$-admissible system trajectories with $\wtil{\bu}\in \ora{\cU}_{\bx(n+1)}$
then yields $S_a(x(n+1)) \le s_a$, and the first equality in \eqref{SaSrRed} holds as required.

To prove the second equality in \eqref{SaSrRed} we follow a similar strategy, which we will only sketch here.
The inclusion $\ora{\cU}_{*}\subset \ora{\cU}_{\bx(n)}$ shows $s_r$ is an upper bound. Any $(\wtil{\bu},\wtil{\bx},\wtil{\by})$
be $\ell^2$-admissible system
trajectory with $\wtil{\bu}\in \ora{\cU}_{\bx(n)}$ can be patched together with an $\ell^2$-admissible system trajectory with
input sequence from $\ora{\cU}_*$ to form a new $\ell^2$-admissible system trajectory $(\wtil{\bu}',\wtil{\bx}',\wtil{\by}')$
with $\wtil{\bu}'$ in $\ora{\cU}_*$, $P_- \wtil{\bu}= P_- \wtil{\bu}'$ and $P_- \wtil{\by}= P_- \wtil{\by}'$, so that
\[
\|P_-\wtil{\bu}\|^2-\|P_-\wtil{\by}\|^2
=\|P_-\wtil{\bu}'\|^2-\|P_-\wtil{\by}'\|^2
\geq\inf_{\wtil{\bu}\in \ora{\cU}_{*}} \|P_-\wtil{\bu}\|^2-\|P_-\wtil{\by}\|^2 =:s_r
\]
which then yields that $s_r$ is also a lower bound for $S_r(\bx(n))$ as required.

\paragraph{\bf Step 4: Both $S_{a}$ and $S_{r}$ are continuous at $0$.}  This is a consequence of the explicit quadratic form
obtained for $S_a$ and $S_r$ in Theorem \ref{T:SaSr-Con} below.

\paragraph{\bf Proof of (3)}
 Let $\widetilde S$ be any storage
    function for $\Sigma$. Let $x_0\in \im\bW_c$ and let $(\bu, \bx, \by)$ be any
    $\ell^2$-admissible dichotomous system trajectory for $\Sigma$ with $\bx(0)=x_0$.  Then
    $\widetilde S$ satisfies the energy balance relation
    \begin{equation}  \label{EB1}
    \widetilde S(\bx(n+1)) - \widetilde S(\bx(n)) \le \|
    \bu(n)\|^{2}_{\cU} -  \| \by(n) \|^{2}_{\cY}.
    \end{equation}
    Summing from $n=0$ to $n=N$ then gives
    \begin{align}\label{EB1'}
  \widetilde S(\bx(N+1)) - \widetilde S(x_0)&= \widetilde S(\bx(N+1)) - \widetilde S(\bx(0)) \notag\\
  &\le     \sum_{n=0}^{N} \left( \| \bu(n) \|^{2}_{\cU} - \| \by(n)
    \|^{2}_{\cY} \right).
    \end{align}
    As $\bx \in \ell^2_\cX({\mathbb Z})$ and $\widetilde S$ as part of being a storage function is continuous at $0$
    with $\widetilde S(0) = 0$, we see from $\bx(N+1) \to 0$ that $\widetilde S(\bx(N+1)) \to \widetilde S(0) = 0$
    as $N \to \infty$.  Hence letting $N \to \infty$ in \eqref{EB1'} gives
    $$
      - \widetilde S(x_0) \le \sum_{n=0}^{\infty} \left( \| \bu(n)
      \|^{ 2}_{\cU} - \| \by(n) \|^{2}_{\cY} \right).
   $$
   But by definition, the infimum of the right-hand side of this last
   expression over all system trajectories $(\bu,\bx,\by)$ of $\Si$ such that $\bx(0) = x_{0}$
   is exactly $-S_{a}(x_{0})$.  We conclude that $ - \widetilde
   S(x_{0}) \le - S_{a}(x_{0})$, and thus $S_{a}(x_{0}) \le \widetilde
   S(x_{0})$ for any $x_{0} \in \im\bW_c$.

   Similarly, if we sum up \eqref{EB1} from $n=-N$ to $n=-1$ we get
$$
  \widetilde S(x_0) - \widetilde S(\bx(-N)) =\widetilde S(\bx(0)) - \widetilde S(\bx(-N)) \le \sum_{n=-N}^{-1}(
 \| \bu(n)\|_{\cU}^{2} - \| \by(n)\|_{\cY}^{2}).
$$
Letting $N \to \infty$ in this expression then gives
$$
 \widetilde S(x_0) \le \sum_{n=-\infty}^{-1} ( \|\bu(n)\|^{2}_{\cU}
 - \| \by(n) \|_{\cY}^{2}).
$$
But by definition the infimum of the right-hand side of this last
inequality over all $\ell^2$-admissible system trajectories $(\bu, \bx, \by)$ with $\bx(0) =
x_{0}$ is exactly equal to $S_{r}(x_{0})$.  We conclude that
$\widetilde S(x_{0}) \le S_{r}(x_{0})$.  This completes the proof of
part (3) of Proposition \ref{P:dichotSa}.
\end{proof}

\paragraph{\bf Quadratic storage functions and spatial KYP-inequalities: the dichotomous setting}

Let us say that a function $S:\cX\to\BR$ is  {\em quadratic} if
there exists a bounded selfadjoint operator $H$ on $\cX$ such that $S(x)=S_H(x):=\inn{H x}{x}$
for all $x\in \cX$.  Trivially any function $S = S_H$ of this form satisfies conditions (1) and (3)
in the definition of storage function (see the discussion around \eqref{EB}).  To characterize which
bounded selfadjoint operators $H$ give rise to $S = S_H$ being a storage function, as we are assuming
that our blanket assumption \eqref{Con} is in force, we may quote the result of
Remark \ref{R:local-storage} to substitute the local version \eqref{EB'} (\eqref{EBstrict'}) of the energy-balance condition
in place of the original version \eqref{EB} (respectively \eqref{EBstrict} for the strict case).  Condition \eqref{EB'} applied to $S_H$
leads us to the condition
$$
 \inn{H (Ax+B u)}{Ax+B u} -\inn{H x}{x}   \leq \|u\|^2-\|Cx+D u\|^2,
$$
or equivalently
\begin{align*}
& \inn{H (Ax \!+\! B u)}{A x \!+\! B u}+\inn{ Cx \!+\! D
u}{Cx\!+\! D u}
\leq \inn{H x}{x}+\inn{u}{u}.
\end{align*}
holding for all $x \in \cX$ and $u \in \cU$.  In a more matricial form, we may write instead
\begin{align}
& \inn{\mat{cc}{H&0\\0&I}\mat{c}{x\\ u}}{\mat{c}{x\\ u}}    \notag \\
& \quad - \inn{\mat{cc}{H&0\\0&I}\mat{cc}{A&B\\ C&D}\mat{c}{x\\
u}}{\mat{cc}{A&B\\ C&D} \mat{c}{x\\ u}}\geq 0
 \label{KYPspatial}
\end{align}
for all $x \in \cX$ and $u \in \cU$.
Hence $H$ satisfies the spatial version \eqref{KYPspatial} of the KYP-inequality
\eqref{KYP}.  By elementary Hilbert-space theory, namely, that a
selfadjoint operator $X$ on a complex Hilbert space is uniquely determined by its associated
quadratic form $ x \mapsto \langle X x, x \rangle$, it follows that
$H$ solves the KYP-inequality \eqref{KYP}, but now
for an infinite-dimensional setup.

If we start with the strict version \eqref{EBstrict'} of the local energy-balance condition,
we arrive at the following criterion for the quadratic function $S_H$ to be a
strict storage function for the system $\Sigma$, namely
the spatial version of the strict  KYP-inequality:
\begin{align}
& \inn{\mat{cc}{H&0\\0&I}\mat{c}{x\\ u}}{\mat{c}{x\\u}}\notag  \\
& - \inn{\mat{cc}{H&0\\0&I}\mat{cc}{A&B\\
C&D}\mat{c}{x\\u}}{\mat{cc}{A&B\\ C&D}\mat{c}{x\\u}}\geq \epsilon \left\|\mat{c}{x\\ u}\right\|^2,
 \label{KYPspatialstrict}
\end{align}
and hence also the strict KYP-inequality in operator form \eqref{KYPstrict}, again now for the
infinite-dimensional setting.
Following the above computations in reversed order shows that the
spatial KYP-inequality \eqref{KYPspatial} and strict spatial KYP-inequality
\eqref{KYPspatialstrict} imply that $S_H$ is a storage
function and strict storage function, respectively.

\begin{proposition} \label{P:quadstorage}
Let $\Si$ be a dichotomous linear system as in \eqref{dtsystem}. Let
$H$ be a bounded, self adjoint
operator on $\cX$. Then $S_H$ is a
quadratic storage function for $\Si$ if and only if $H$ is a solution
of the KYP-inequality \eqref{KYP}. Moreover, $S_H$ is a strict quadratic storage
function if and only if $H$ is a  solution of the strict
 KYP-inequality \eqref{KYPstrict}.
\end{proposition}

\section{The available storage and required supply}  \label{S:SaSr}

We assume throughout this section that the dichotomous linear
system $\Sigma$ satisfies the standing assumption \eqref{Con}.
Under these conditions we shall show that the
available storage $S_a$ and required supply $S_r$ are quadratic
storage functions and we shall obtain explicit formulas for the associated
selfadjoint operators $H_{a}$ and $H_{r}$ satisfying the KYP-inequality
\eqref{KYP}.

The assumption that $\|F_\Si\|_{\infty, {\mathbb T}} \leq 1$ implies that the associated Laurent operator
$\fL_{F_\Si}$ in \eqref{Laurent-matrix} is a
contraction, so that the Toeplitz operators $\fT_{F_\Si}$ and
$\wtil{\fT}_{F_\Si}$ \eqref{Toeplitz} are also contractions. Thus
$I-\fL_{F_\Si}\fL_{F_\Si}^*$ and $I-\fL_{F_\Si}^*\fL_{F_\Si}$ are
both positive operators. Writing out these operators in terms of the
operator matrix decomposition \eqref{dichotLaurent} we obtain
\begin{equation}\label{decomLaurentdefect}
\begin{aligned}
I - \frakL_{F_\Si} \frakL_{F_\Si}^{*}
&=\begin{bmatrix} D_{\widetilde \frakT_{F_\Si}^{*}}^{2} - \widetilde
\frakH_{F_\Si}
	\widetilde \frakH_{F_\Si}^{*} &
-\widetilde \frakT_{F_\Si} \frakH_{F_\Si}^{*} - \widetilde
\frakH_{F_{\Si}}
\frakT_{{F}_{\Si}}^{*} \\[.3cm]
- \frakH_{F_{\Si}} \widetilde \frakT_{F_\Si}^{*} - \frakT_{F_\Si}
\widetilde
\frakH_{F_{\Si}}^{*}  &
D_{\frakT_{F_\Si}^{*}}^{2} - \frakH_{F_{\Si}} \frakH_{F_{\Si}}^{*}
\end{bmatrix}\\
I-\fL_{F_\Si}^* \fL_{F_\Si}
&=\mat{cc}{
D_{\wtil{\fT}_{\Si}}^2-\fH_{\Si}^*\fH_{F_\Si}&
-\fH_{F_\Si}^*\fT_{F_\Si}-\wtil{\fT}_{F_\Si}^*\wtil{\fH}_{F_\Si}\\
-\fT_{F_\Si}^*\fH_{F_\Si}-\wtil{\fH}_{F_\Si}^*\wtil{\fT}_{F_\Si}  &
D_{\fT_{\Si}}^2-\wtil{\fH}_{\Si}^*\wtil{\fH}_{F_\Si}
}.
\end{aligned}
\end{equation}
In particular, from $I-\fL_{F_\Si}\fL_{F_\Si}^*$ and $I-\fL_{F_\Si}^*\fL_{F_\Si}$ being positive operators we read off that
\begin{equation}   \label{Hankel-est'}
D_{\frakT_{F}^{*}}^{2} \succeq \frakH_{F_\Si} \frakH_{F_\Si}^{*},\ \
D_{\widetilde \frakT_{F_\Si}^{*}}^{2} \succeq \widetilde
\frakH_{F_\Si}\widetilde \frakH_{F_\Si}^{*},\ \
D_{\wtil{\fT}_{\Si}}^2\succeq
\fH_{\Si}^*\fH_{F_\Si},\ \
D_{\fT_{\Si}}^2\succeq \wtil{\fH}_{\Si}^*\wtil{\fH}_{F_\Si}.
\end{equation}

Applying Douglas' Lemma \cite{Douglas} along with the factorizations
in \eqref{Hankelfact} enables us to prove the following result.

%%%%%%%%%%%%%%%%%%%%%%%%%%%%%%%%%%%%%%%%%%%%%%%%%%%%%%%%%%%%%%%%%%%
%%%%%
\begin{lemma}\label{L:Xops}
Assume the dichotomous system $\Si$ in \eqref{dtsystem} satisfies
\eqref{Con}. Then there exist unique injective bounded linear operators $X_{o,+}$,
$X_{o,-}$, $X_{c,+}$ and $X_{c,-}$ such that
\begin{align}
& X_{o,+}: \cX_{+} \to \ell^2_\cU(\BZ_+),\quad
\bW_o^{+}  = D_{{\fT}_{F_\Si}^*}X_{o,+},\quad
\im X_{o,+}\subset \overline{\im} D_{{\fT}_{F_\Si}^*},
\label{fact1} \\
& X_{o,-}: \cX_{-} \to \ell^2_\cU(\BZ_-),\quad
\bW_o^{-} =D_{\wtil{\fT}_{F_\Si}^*}X_{o,-},\quad
\im X_{o,-}\subset \overline{\im} D_{\wtil{\fT}_{F_\Si}^*},
\label{fact2} \\
& X_{c,+}:\ \cX_{+} \to \ell^2_\cY(\BZ_+),\quad
(\bW_c^{+})^{*} =D_{\wtil{\fT}_{F_\Si}}X_{c,+},\quad
\im X_{c,+}\subset \overline{\im} D_{\wtil{\fT}_{F_\Si}},
\label{fact3} \\
& X_{c,-}: \cX_{-} \to \ell^2_\cY(\BZ_-),\quad
(\bW_c^{-})^*  =D_{{\fT}_{F_\Si}}X_{c,-},\quad
\im X_{c,-}\subset \overline{\im} D_{{\fT}_{F_\Si}}.
\label{fact4}
\end{align}
\end{lemma}

%%%%%%%%%%%%%%%%%%%%%%%%%%%%%%%%%%%%
\begin{proof}[\bf Proof]
 We give the details of the proof only for $X_{o,+}$ as the other
 cases are similar.  The argument is very much like the proof of
 Lemma 4.8 in \cite{KYP2} where the argument is more complicated due
 the unbounded-operator setting there.

Since $D_{\fT_{F_\Si}^*}^2\succeq \fH_{F_\Si}\fH_{F_\Si}$, the
Douglas factorization lemma \cite{Douglas} implies the existence of a
unique contraction operator
$Y_{o,+}:\ell^2_\cU(\BZ_-)\to\ell^2_\cU(\BZ_+)$ with
\[
\bW_o^+\bW_c^+=\fH_{F_{\Si}}=D_{\fT_{F_\Si}^*}Y_{o,+}
\ands \im Y_{o,+}\subset \ov{\im} D_{\fT_{F_\Si}^*}.
\]
As $\im \bW_{c} = \cX_{+}$, the Open Mapping Theorem guarantees that
$\bW_{c}^{+}$ has a bounded right inverse $\bW_{c}^{+ \dagger} : =
\bW_{c}^{+*} (\bW_{c}^{+} \bW_{c}^{+ *})^{-1}$. Moreover, $\bu = \bW_{c}^{+
\dagger} x$ is the least norm solution of the equation $\bW_{c}^{+}
\bu = x$:
\[
\bW_{c}^{+\dagger} (x) = {\text{arg min }} \{
\|\bu\|^{2}_{\ell^{2}_{\cU}({\mathbb Z}_{-})} \colon
\bu\in\cD(\bW_c^+),\  x =\bW_{c}^+ \bu \}\quad (x\in \im \bW_c^+).
\]
We now define $X_{o,+}$ by
$$
   X_{o,+} = Y_{o,+}\bW_{c}^{+ \dagger}.
$$
We then observe
$$
D_{{\fT}_{F_\Si}^*}X_{o,+} = D_{{\fT}_{F_\Si}^*} Y_{o,+} \bW_{c}^{+
\dagger} = \fH_{F_{\Sigma}} \bW_{c}^{ + \dagger} = \bW_{o}^{+}
\bW_{c}^{+} \bW_{c}^{+ \dagger} = \bW_{o}^{+}
$$
giving the factorization \eqref{fact1} as wanted.  Moreover, the factorization $X_{o,+} = Y_{o,+}
\bW_{c}^{+ \dagger}$ implies that $\im X_{o,+} \subset \im Y_{o,+}
\subset \overline{\im} D_{{\fT}_{F_\Si}^*}$; this property combined
with the factorization \eqref{fact1} makes the choice of $X_{o,+}$
unique.  Moreover, the containment $\im X_{o,+} \subset \overline{\im} D_{{\fT}_{F_\Si}^*}$
combined with the injectivity of $\bW_{o}^{+}$ forces the injectivity
of $X_{o,+}$.
\end{proof}

We are now ready to analyze both the available storage function $S_{a}$ and
the required supply  function $S_{r}$ for a system meeting hypotheses
\eqref{Con}.

\begin{theorem} \label{T:SaSr-Con}
Suppose that $\Sigma$ is a dichotomous discrete-time linear system as in \eqref{dtsystem} which satisfies hypotheses \eqref{Con}.
Then $S_{a} = S_{H_{a}}$ and $S_{r}= S_{H_{r}}$ are quadratic storage functions with associated selfadjoint operators
$H_a$ and $H_r$ bounded and boundedly invertible on $\cX$, and $S_{a}$ and $S_{r}$ are given by
\begin{align}
& S_{a}(x_{0})   =  \| X_{o,+} (x_{0})_{+} \|^{2}  -
\left\| P_a\frakT_{F}^{*} X_{o,+}(x_{0})_{+} -
P_a D_{\frakT_{F}}  \bW_{c}^{-\dagger} (x_{0})_{-}
\right\|^{2} \label{Ha} \\
& S_{r}(x_{0}) =
\left\| P_r
\widetilde \frakT_{F}^{*} X_{o,-}(x_{0})_{-} -
P_r D_{\widetilde \frakT_{F}} \bW_{c}^{+\dagger}(x_{0})_{+}
\right\|^{2} - \| X_{o,-} (x_{0})_{-} \|^{2} \label{Hr}
\end{align}
with $x_0=(x_0)_+ \oplus (x_0)_-$ the decomposition of $x_0$ with respect to the direct sum $\cX=\cX_+ \dot{+}\cX_-$,
the operators $X_{o,+}$ and $X_{o,-}$ as in Lemma \ref{L:Xops} and
\begin{align*}
 & \bW_{c}^{- \dagger} = \bW_{c}^{- *} (\bW_{c}^{-} \bW_{c}^{-*})^{-1},
 \quad \bW_{c}^{+ \dagger} =  \bW_{c}^{+ *} (\bW_{c}^{+}
 \bW_{c}^{+*})^{-1},  \\
& P_{a} =  P_{(D_{\fT_{F}} {\rm Ker} \bW_{c}^{-})^{\perp}}, \quad
P_{r} =  P_{(D_{\widetilde \fT_{F}} {\rm Ker} \bW_{c}^{+})^{\perp}}.
\end{align*}
In particular, $S_a$ and $S_r$ are continuous.

If we assume that the decomposition $\cX = \cX_{-} \dot + \cX_{+}$
inducing the decompositions \eqref{Adec} and \eqref{BCdec} is
actually orthogonal, which can always be arranged via an invertible
similarity-transformation change of coordinates in $\cX$ if
necessary, then with respect to the the orthogonal decomposition $\cX = \cX_{-} \oplus \cX_{+}$, $H_{a}$ and $H_{r}$ are given explicitly by
\begin{align}
    H_{a} & =
\begin{bmatrix}
X_{o,+}^{*} (I - \fT_{F} P_{a} \fT_{F}^{*}) X_{o,+}  &  X_{o,+}^{*} \fT_{F} P_{a} D_{\fT_{F}} \bW_{c}^{- \dagger}\\
\bW_{c}^{- \dagger *} D_{\fT_{F}} P_{a} \fT_{F}^{*} X_{o,+} &
- \bW_{c}^{- \dagger *} D_{\fT_{F}} P_{a} D_{\fT_{F}} \bW_{c}^{- \dagger}
\end{bmatrix}, \label{Ha-explicit}\\
H_{r} & =
\begin{bmatrix}
\bW_{c}^{+ \dagger *} D_{\widetilde \fT_{F}} P_{r} D_{\widetilde \fT_{F}} \bW_{c}^{+ \dagger} &
 - \bW_{c}^{+ \dagger *} D_{\widetilde \fT_{F}} P_{r} \widetilde \fT_{F}^{*} X_{o,-}\\
-X_{o,-}^{*} \widetilde \fT_{F} P_{r} D_{\widetilde \fT_{F}} \bW_{c}^{* \dagger} &
-X_{o,-}^{*} (I - \widetilde \fT_{F} P_{r} \widetilde \fT_{F}^{*}) X_{o,-}
\end{bmatrix}. \label{Hr-explicit}
\end{align}
Furthermore, the dimension of the spectral subspace of $A$ over the
unit disk agrees with the dimension of the spectral subspace of
$H_{a}$ and $H_{r}$ over the positive real line (= $\dim \cX_{+}$),
and the dimension of the spectral subspace of $A$ over the
exterior of the closed unit disk agrees with the dimension of the spectral subspace of
$H_{a}$ and $H_{r}$ over the negative real line (= $\dim \cX_{-}$).
\end{theorem}

\begin{proof}[\bf Proof]
To simplify notation, in this proof we write simply $F$ rather than
$F_{\Sigma}$.

We start with the formula for $S_a$. Fix $x_0\in \im \bW_c$. Let
$(\bu,\bx,\by)$ be any system trajectory of $\Si$ such that $x_0=\bW_c \bu$.

The first step in the calculation of $S_{a}$ is to reformulate the
formula from the definition \eqref{dichotSa} in operator-theoretic
form:
 \begin{equation}   \label{dichotSa'}
     S_{a}(x_{0}) = \sup_{ \bu \colon \bW_{c}\bu = x_{0}}
     \| (\frakL_{F} \bu)|_{{\mathbb
     Z}_{+}}\|_{\ell^{2}_{\cY}({\mathbb Z}_{+})}^{2} - \|
\bu|_{{\mathbb
     Z}_{+}} \|^{2}_{\ell^{2}_{\cU}({\mathbb Z}_{+})}.
 \end{equation}
 From the formulas \eqref{dichotLaurent}, \eqref{Toeplitz}, and
 \eqref{Hankel} for $\frakL_{F}$, in more detail we have
 $$
S_{a}(x_{0}) =  \sup_{ \bu_{+},\, \bu_{-} \colon
\bW_{c}^{+} \bu_{-} = (x_{0})_{+}, \,  \bW_{c}^{-}\bu_{+} =
(x_{0})_{-}}
\| \frakH_{F_{+}} \bu_{-} + \frakT_{F} \bu_{+}\|^{2} - \|
\bu_{+}\|^{2}.
 $$
 where $\bu_{-} \in \ell^{2}_{\cU}({\mathbb Z}_{-})$ and $\bu_{+} \in
 \ell^{2}_{\cU}({\mathbb Z}_{+})$ and where $x_{0} = (x_{0})_{-} +
 (x_{0})_{+}$ is the decomposition of $x_{0}$ into $\cX_{-}$ and
 $\cX_{+}$ components.  Recalling the factorization $\frakH_{F} =
 \bW_{o}^{+} \bW_{c}^{+}$ from \eqref{Hankel} as well as the
 constraint on $\bu_{-}$, we rewrite the objective function in the
 formula for $S_{a}(x_{0})$ as
$$
\| \frakH_{F} \bu_{-} + \frakT_{F} \bu_{+}\|^{2} - \|
\bu_{+}\|^{2} =
\| \bW_{o}^{+} (x_{0})_{+} + \frakT_{F} \bu_{+} \|^{2} - \|
\bu_{+}\|^{2}.
$$
Furthermore, by assumption $\bW_{c}^{+}$ is surjective, so there is
always a $\bu_{-} \in \ell^{2}_{\cU}({\mathbb Z}_{-})$ which achieves
the constraint $\bW_{c}^{+} \bu_{-} = (x_{0})_{+}$. In this way we
have eliminated the parameter $\bu_{-}$ and the formula for
$S_{a}(x_{0})$ becomes
\begin{equation}   \label{dichotSa''}
 S_{a}(x_{0}) = \sup_{\bu_{+} \in \ell^{2}_{\cU}({\mathbb Z}_{+})
 \colon \bW_{c}^{-} \bu_{+} = (x_{0})_{-}}
 \| \bW_{o}^{+} (x_{0})_{+} + \frakT_{F} \bu_{+} \|^{2} - \|
\bu_{+}\|^{2}.
\end{equation}
By Lemma \ref{L:Xops} there is a uniquely determined injective linear
operator $X_{o,+}$ from $\cX_{+}$  to  $\overline{\im} D_{\fT_{F}^{*}}$
so that $W_{o}^{+} = D_{\fT_{F}^{*}} X_{o,+}$.
Then the  objective function in \eqref{dichotSa''} becomes
\begin{align*}
&\| \bW_{o}^{+} (x_{0})_{+} + \frakT_{F} \bu_{+} \|^{2} - \|
\bu_{+}\|^{2}=\\
& \quad =\| D_{\frakT_{F}^{*}} X_{o,+} + \frakT_{F} \bu_{+} \|^{2}
- \|\frakT_{F}  \bu_{+}\|^{2} - \|D_{\frakT_{F}}  \bu_{+}\|^{2}\\
 & \quad = \| D_{\frakT_{F}^{*}} X_{o,+} (x_{0})_{+} \|^{2} +
 2{\rm Re} \langle D_{\frakT_{F}^{*}} X_{o,+} (x_{0})_{+}, \frakT_{F}
 \bu_{+} \rangle - \| D_{\frakT_{F}} \bu_{+} \|^{2} \\
& \quad = \| D_{\frakT_{F}^{*}} X_{o,+} (x_{0})_{+} \|^{2}
+ 2 {\rm Re} \langle X_{o,+} (x_{0})_{+}, \frakT_{F} D_{\frakT_{F}}
u_{+} \rangle - \| D_{\frakT_{F}} \bu_{+} \|^{2}  \\
& \quad = \| D_{\frakT_{F}^{*}} X_{o,+} (x_{0})_{+} \|^{2} +
2 {\rm Re} \langle \frakT_{F}^{*} X_{o,+} (x_{0})_{+}, D_{\frakT_{F}}
\bu_{+} \rangle - \| D_{\frakT_{F}} \bu_{+} \|^{2}  \\
& \quad = \| D_{\frakT_{F}^{*}} X_{o,+} (x_{0})_{+} \|^{2} + \|
\frakT_{F}^{*} X_{o,+} (x_{0})_{+} \|^{2}  -
\| \frakT_{F}^{*} X_{o,+}(x_{0})_{+} - D_{\frakT_{F}} \bu_{+} \|^{2} \\
& \quad = \| X_{o,+} (x_{0})_{+} \|^{2} -
\| \frakT_{F}^{*} X_{o,+}(x_{0})_{+} - D_{\frakT_{F}} \bu_{+} \|^{2}.
\end{align*}
In this way we arrive at the decoupled formula for $S_{a}(x_{0})$:
\begin{equation} \label{dichotSadecoupled}
    S_{a}(x_{0}) = \| X_{o,+}(x_{0})_{+} \|^{2} - \!\! \inf_{\bu_{+} \colon
    \bW_{c}^{-} \bu_{+} = (x_{0})_{-}} \!\! \| \frakT_{F}^{*}
    X_{o,+}(x_{0})_{+} - D_{\frakT_{F}} \bu_{+} \|^{2}.
\end{equation}
By assumption $\bW_{c}^{-}$ is surjective and hence $\bW_{c}^{-}$ is
right invertible with right inverse equal to $\bW_{c}^{-*}
(\bW_{c}^{-} \bW_{c}^{- *})^{-1}$. In particular, the minimal-norm
solution
$\bu_{+}^{0}$ of $\bW_{c}^{-} \bu_{+} = (x_{0})_{-}$ is given by
$$
\bu_{+}^{0} =\bW_{c}^{- *} (\bW_{c}^{-} \bW_{c}^{-*})^{-1} (x_{0})_{-}=\bW_{c}^{- \dagger} (x_{0})_{-}
$$
and then any other solution has the form
$$
  \bu_{+} = \bu_{+}^{0} + \bv_{+} \text{ where }  \bv_{+} \in {\rm
  Ker} \,
  \bW_{c}^{-}.
$$
By standard Hilbert space theory, it then follows that
\begin{align*}
  &   \inf_{\bu_{+} \colon
    \bW_{c}^{-} \bu_{+} = (x_{0})_{-}} \| \frakT_{F}^{*}
X_{o,+}(x_{0})_{+}
    - D_{\frakT_{F}} \bu_{+} \|^{2} \\
  & \quad   = \left\|  P_{(D_{\cT_{F}} {\rm Ker}
 \bW_{c}^{-})^{\perp}}\left(\frakT_{F}^{*} X_{o,+}(x_{0})_{+} -
D_{\frakT_{F}}
 \bu_{+}^0
\right)\right\|^{2}\\
 & \quad   = \left\|  P_{a}\left(\frakT_{F}^{*} X_{o,+}(x_{0})_{+} -
D_{\frakT_{F}}
 \bW_{c}^{-\dagger}(x_{0})_{-}
\right)\right\|^{2}
\end{align*}
and we arrive at the formulas \eqref{Ha}  for
$S_{a}$.   A few more notational manipulation leads to the explicit formula \eqref{Ha-explicit}
for $H_a$ when $\cX = \cX_- \dot{+} \cX_+$ is an orthogonal decomposition.

In a similar vein, the formula \eqref{dichotSr} for $S_{r}$ can be
written in operator form as
$$
S_{r}(x_{0}) = \inf_{\bu \colon \bW_{c} \bu = x_{0}} \| \bu_{-}
\|^{2} -\| P_{\ell^{2}_{\cY}({\mathbb Z}_{-})} \fL_{F} \bu \|^{2}.
$$
Then the objective function can be written as
\begin{align*}
&\| \bu_{-} \|^{2} -\| P_{\ell^{2}_{\cY}({\mathbb Z}_{-})} \fL_{F} \bu \|^{2}
=\| \bu_{-} \|^{2} - \| \widetilde \fH_{F} \bu_{+} + \widetilde \fT_{F} \bu_{-} \|^{2} \\
&\quad = \| \bu_{-} \|^{2} - \| \bW_{o}^{-} \bW_{c}^{-} \bu_{+} +
 \widetilde \fT_{F} \bu_{-} \|^{2}
 = \| \bu_{-} \|^{2} - \| \bW_{o}^{-} (x_{0})_{-} + \widetilde
 \fT_{F} \bu_{-} \|^{2} \\
 & \quad = \| \bu_{-} \|^{2} - \| D_{\widetilde\fT_{F}^{*}} X_{o,-}(x_{0})_{-}
 + \widetilde \fT_{F} \bu_{-} \|^{2}  \\
 & \quad = \| \bu_{-}\|^{2} - \| D_{\widetilde \fT_{F}^{*}} X_{o,-}
 (x_{0})_{-}\|^{2} - 2 {\rm Re} \langle D_{\widetilde \fT_{F}^{*}}
 X_{o,-} (x_{0})_{-}, \widetilde \fT_{F} \bu_{-} \rangle - \|
 \widetilde \fT_{F} \bu_{-} \|^{2} \\
 & \quad =  - \| D_{\widetilde \fT_{F}^{*}} X_{o,-} (x_{0})_{-} \|^{2}
 - 2 {\rm Re} \langle X_{o,-} (x_{0})_{-}, D_{\widetilde \fT_{F}^{*}}
 \widetilde \fT_{F} \bu_{-} \rangle + \| D_{\widetilde \fT_{F}}
 \bu_{-} \|^{2}  \\
 & \quad = - \| D_{\widetilde \fT_{F}^{*}} X_{o,-} (x_{0})_{-} \|^{2}
- 2 {\rm Re} \langle \widetilde \fT_{F}^{*} X_{o,-} (x_{0})_{-},
D_{\widetilde \fT_{F}} \bu_{-} \rangle + \| D_{\widetilde \fT_{F}}
\bu_{-} \|^{2} \\
& \quad = - \| D_{\widetilde \fT_{F}^{*}} X_{o,-} (x_{0})_{-} \|^{2} -
\| \widetilde \fT_{F}^{*} X_{o,-} (x_{0})^{-} \|^{2}
+ \| \widetilde \fT_{F}^{*} X_{o,-}(x_{0})_{-} - D_{\widetilde
\fT_{F}} \bu_{-} \|^{2} \\
& \quad = - \| X_{o,-} (x_{0})_{-} \|^{2} +  \| \widetilde
\fT_{F}^{*} X_{o,-} (x_{0})_{-} - D_{\widetilde \fT_{F}} \bu_{-} \|^{2}
 \end{align*}
where now $\bu_{+}$ is eliminated and the constraint on the free
parameter $\bu_{-}$ is $\bW_{c}^{+} \bu_{-} = (x_{0})_{+}$.
Thus
$$
  S_{r}(x_{0}) = - \| X_{o,-} (x_{0})_{-}\|^{2} + \inf_{\bu_{-}
  \colon \bW_{c}^{+} \bu_{-} = (x_{0})_{+}} \| \widetilde \fT_{F}^{*}
  X_{o,-} (x_{0})_{-} - D_{\widetilde \fT_{F}} \bu_{-} \|^{2}.
$$
We note that all possible solutions $\bu_{-}$ of the constraint
$\bW_{c}^{+} \bu_- = (x_{0})_{+}$ are given by
$$
\bu_{-} = \bW_{c}^{+*} (\bW_{c}^{+} \bW_{c}^{+*})^{-1} (x_{0})_{+} + \bv_{-}
= \bW_{c}^{+\dagger} (x_{0})_{+} + \bv_{-}
\text{ where } \bv_{-} \in {\rm Ker} \bW_{c}^{+}.
$$
Then standard Hilbert-space theory leads to the formulas \eqref{Hr}
 for $S_{r}$; a little more careful manipulation leads to the explicit form \eqref{Hr-explicit} for $H_{r}$.

We next wish to verify that $H_{a}$ and $H_{r}$ are invertible.  This follows as
an application of results referred to as inertial theorems; as these results
are well known for the finite-dimensional settings (see e.g.
\cite{LT}) but not so well known for the infinite-dimensional
settings,  we go through the results in some detail here.

As a consequence
of Proposition \ref{P:quadstorage}, we know that $H_{a}$ is a
solution of the KYP-inequality \eqref{KYP}. From the (1,1)-entry of
\eqref{KYP} we see in particular that
$$
   H_{a} - A^{*} H_{a} A - C^{*} C \succeq 0.
$$
Write $H_a$ as a block operator matrix with respect to the direct sum decomposition $\cX=\cX_+ \dot{+} \cX_-$ as
\begin{equation}\label{Hablock}
H_a=\begin{bmatrix} H_{a-} & H_{a0} \\ H_{a0}^{*} & H_{a+} \end{bmatrix} \ons \mat{c}{\cX_-\\\cX_+}.
\end{equation}
We can then rewrite the above inequality as
$$
 \begin{bmatrix} H_{a-} & H_{a0} \\ H_{a0}^{*} & H_{a+} \end{bmatrix}
     - \begin{bmatrix} A_{-}^{*} & 0 \\ 0 & A_{+}^{*} \end{bmatrix}
  \begin{bmatrix} H_{a-} & H_{a0} \\ H_{a0}^{*} & H_{a+} \end{bmatrix}
      \begin{bmatrix} A_{-} & 0 \\ 0 & A_{+} \end{bmatrix} \succeq
      \begin{bmatrix} C_{-}^{*} \\ C_{+}^{*} \end{bmatrix}
	  \begin{bmatrix} C_{-} & C_{+} \end{bmatrix}.
$$
From the diagonal entries of this block-operator inequality we get
$$
-H_{a-} + A_{-}^{*-1}H_{a-} A_{-}^{-1} \succeq  A_{-}^{*-1}
 C_{-}^{*} C_{-} A_{-}^{-1}, \ \
 H_{a+} - A_{+}^{*} H_{a+} H_{+} A_+ \succeq C_{+}^{*} C_{+}.
 $$
 An inductive argument then gives
 \begin{align*}
  -H_{a-} & \succeq \sum_{n=1}^{N} A^{* -n} C_{-}^{*} C_{-} A_{-}^{-n} -
 A_{-}^{* -N} H_{a -} A_{-}^{-N}, \\
 H_{a+} & \succeq \sum_{n=0}^{N} A_{+}^{*n} C_{+}^{*} C_{+} A_{+}^{n} +
 A_{+}^{* N+1} H_{a+} A_+^{N+1}.
 \end{align*}
 As both $A_{-}^{-1}$  and $A_{+}$  are exponentially stable, we may
 take the limit as $N \to \infty$ in both of the above expressions to
 get
 $$
  -H_{a-} \succeq  (\bW_{o}^-)^* \bW_{o}^{-}, \quad
  H_{a+} \succeq (\bW_{o}^+)^ * \bW_{o}^+.
  $$
 By the dichotomous $\ell^{2}$-exactly observable assumption,
 both operators $(\bW_{o}^{-})^*$ and $(\bW_{o}^{+})^*$  are surjective, and hence
$(\bW_{o}^-)^* \bW_{o}^{-} $ and $(\bW_{o}^+)^ * \bW_{o}^+$ are also surjective.
Thus we can
 invoke the Open Mapping Theorem to get that both $ (\bW_{o}^-)^* \bW_{o}^{-} $ and
 $(\bW_{o}^+)^ * \bW_{o}^+$ are bounded below. We conclude that both $H_{a+}$ and
 $-H_{a-}$ are strictly positive-definite, i.e., there is an $\epsilon > 0$ so
 that $H_{a+} \succeq \epsilon I$ and $H_{a-} \preceq - \epsilon I$.
 In particular, both $H_{a+}$ and $H_{a-}$ are invertible.

 It remains to put all this together to see that $H_{a}$ and $H_{r}$
 are invertible.  We do the details for $H_{a}$ as the proof for
 $H_{r}$ is exactly the same.  By Schur complement theory (see e.g.\
 \cite{DP}), applied to the block matrix decomposition of $H_a$ in \eqref{Hablock},
 given that the operator $H_{a+}$ is invertible (as we
 have already verified), then $H_{a}$ is also invertible if and only if the
 Schur complement
 $\cS(H_{a}; H_{a+}) : = H_{a-} - H_{a0} H_{a+}^{-1} H_{a0}^{*}$ is invertible.
 But we have already verified that both $H_{a-}$ and $-H_{a+}$ are
 strictly positive-definite.  Hence the Schur complement is the sum of a strictly positive-definite operator and an at worst positive-semidefinite operator, and hence is
 itself strictly positive-definite and therefore also
 invertible.  We next note the block diagonalization of $H_{a}$
 associated with the Schur-complement computation:
 $$
 \begin{bmatrix} H_{a-} & H_{a0} \\ H_{a0}^{*} & H_{a+} \end{bmatrix}
 = \begin{bmatrix} I & H_{a0} H_{a+}^{-1} \\  0 & I \end{bmatrix}
\begin{bmatrix} \cS(H_{a}; H_{a+}) & 0 \\ 0 & H_{a+} \end{bmatrix}
\begin{bmatrix} I & 0 \\ H_{a+}^{-1}H_{a0}^{*} & I \end{bmatrix}.
 $$
 Thus $H_{a}$ is congruent with $\sbm{ \cS(H_{a}; H_{a+}) & 0 \\ 0 &
 H_{a-} }$ where we have seen that
 $$
 \cS(H_{a}; H_{a+})  \succ 0 \text{ on } \cX_+, \quad
 H_{a-} \prec 0 \text{ on } \cX_{-}.
 $$
  In this way we arrive
 at the (infinite-dimensional) inertial relations between $H$ and $A$:
 the dimension of the spectral subspace of $A$ over the unit disk
 is the same as the dimension of the spectral subspace of $H$ over the
 positive real axis, namely $\dim \cX_{+}$, and the dimension of the
 spectral subspace of $A$ over the exterior of the unit disk is the
 same as the dimension of the spectral subspace of $H$ over the
 negative real axis, namely $\dim \cX_{-}$.
 \end{proof}

\begin{remark}  Rather than the full force of assumption \eqref{Con},  let us now only
assume that $\| F_{\Sigma} \|_{\infty, {\mathbb T}} \le 1$.
A careful analysis of the proof shows that $H_{a}$ and $H_{r}$ each being
bounded requires only the dichotomous $\ell^{2}$-exact controllability
assumption
(surjectivity of $\bW_{c}$). The invertibility of each of $H_{a}$ and
$H_{r}$ requires in addition the dichotomous $\ell^{2}$-exact
observability assumption (surjectivity of $\bW_{o}^{*}$).
Moreover, if the $\ell^{2}$-exact observability condition is weakened
to observability (i.e., $\bigcap_{n \ge 0}{\rm Ker} C_{+} A_{+}^{n} =
\{0\}$ and $\bigcap_{n \ge 0} {\rm Ker} C_{-} A_{-}^{-n-1} = \{0\}$),
then one still gets that $H_{a}$ and $H_{r}$ are injective but their
respective inverses may not be bounded.
 \end{remark}

\begin{remark}   If $\| F \|_{{\infty, \mathbb T}} < 1$ (where we are setting
    $F = F_{\Sigma}$), then $D_{\fT_{F}^{*}}$ and $D_{\widetilde
    \fT_{F}^{*}}$ are invertible, and we can solve uniquely for
    the operators $X_{0,+}$ and $X_{0,-}$  in Lemma \ref{L:Xops}:
 $$
  X_{0,+} = D_{\fT_{F}^{*}}^{-1} \bW_{o}^{+}, \quad
  X_{0,-} =  D_{\widetilde \fT_{F}^{*}}^{-1} \bW_{o}^{-}.
 $$
 We may then plug in these expressions for $X_{0,+}$ and $X_{0,-}$
 into the formulas \eqref{Ha}, \eqref{Hr}, \eqref{Ha-explicit},
 \eqref{Hr-explicit} to get even more explicit formulas for $S_{a}$,
 $S_{r}$, $H_{a}$ and $H_{r}$.
 \end{remark}

\section{Storage functions for bicausal systems}  \label{S:bicausal-storage}

We now consider how the analysis in Sections \ref{S:storage} and \ref{S:SaSr}, concerning storage functions
$S \colon \cX \to {\mathbb R}$, available storage $S_{a}$ and required supply $S_{r}$, quadratic storage function $S_{H}$, etc.,
can be adapted to the setting of a bicausal system $\Sigma = (\Sigma_+, \Sigma_-)$ with subsystems
\eqref{dtsystem-f} and \eqref{dtsystem-b}, where now $\ell^2$-admissible trajectories refer to signals of the form
$(\bu, \bx_{-} \oplus \bx_{+}, \by)$ such that $\by = \by_{-} + \by_{+}$ with $(\bu,
     \bx_{-}, \by_{-})$ an $\ell^2$-admissible system trajectory of  $\Sigma_{-}$ and
      $(\bu, \bx_{+}, \by_{+})$ an $\ell^2$-admissible system
     trajectory of $\Sigma_{+}$.  We define $S \colon \cX :
     =\cX_- \oplus \cX_+ \to {\mathbb R}$  to be a {\em storage function} for $\Sigma$ exactly
     as was done in Section \ref{S:storage} for the dichotomous case, i.e., we demand that
     \begin{enumerate}
     \item $S$ is continuous at $0$,
     \item $S$ satisfies the energy balance relation \eqref{EB} along all $\ell^2$-admissible system trajectories of the bicausal
     system $\Sigma = (\Sigma_-, \Sigma_+)$, and
     \item $S(0) = 0$.
     \end{enumerate}
     We again say that $S$ is a {\em strict storage function} for $\Sigma$ if the strict energy-balance relation
     \eqref{EBstrict} holds over all $\ell^2$-admissible system trajectories $(\bu, \bx, \by)$ for the bicausal system
     $\Sigma = (\Sigma_-, \Sigma_+)$.
    By following the proof of Proposition \ref{P:dichotStoragefunc} verbatim, but now interpreted for the more general setting
    of a bicausal system $\Sigma = (\Sigma_-, \Sigma_+)$, we arrive at the following result.

    \begin{proposition}  \label{P:Storagefunc-bic}  Suppose that $S$ is a storage function for the bicausal system
    $\Sigma = (\Sigma_-, \Sigma_+)$ in
    \eqref{dtsystem-b}--\eqref{dtsystem-f}, with $\widetilde A_\pm$ exponentially stable.      Then the input-output map
    $T_\Sigma$ is contractive ($\| T_\Sigma \| \le 1$).
    In case $S$ is a strict storage function for $\Sigma$, the input-output map is a strict contraction ($\| T_\Sigma \| < 1$).
    \end{proposition}

    To get further results for bicausal systems, we impose the condition \eqref{Con}, interpreted property for the bicausal setting
    as explained in Section \ref{S:bicausal}.  In particular, with the bicausal $\ell^2$-exact controllability assumption in place,
    we get the following analogue of Remark \ref{R:local-storage}.

    \begin{remark} \label{R:local-storage-bi}  We argue that {\em the second condition \eqref{EB} (respectively, \eqref{EBstrict} for the
    strict case) in the definition of a storage function for a bicausal system $\Sigma = (\Sigma_-, \Sigma_+)$ (assumed to be
    $\ell^2$-exactly controllable) can be replaced by the local condition}
    \begin{align}
&  S( x_- \oplus (\widetilde A_+ x_+ + B_+ u) ) - S((\widetilde A_- x_- + B_- u) \oplus x_+ ) \notag  \\
& \quad \quad \quad \quad \le \| u \|^2 - \| \widetilde C \widetilde A_- x_- + \widetilde C_+ x_+ + (\widetilde C_- \widetilde B_- + D) u \|^2.
\label{EB'bi}
\end{align}
{\em for the standard case, and by its strict version}
\begin{align}
&  S( x_- \oplus (\widetilde A_+ x_+ + B_+ u) ) - S((\widetilde A_- x_- + B_- u) \oplus x_+ )  \notag  \\
& \quad \quad \quad \quad
+ \epsilon^2  ( \| \widetilde A_- x_- + \widetilde B_- u \|^2 + \| x_+ \|^2 + \| u \|^2)
\notag  \\
& \quad \quad  \le \| u \|^2 - \| \widetilde C \widetilde A_- x_- + \widetilde C_+ x_+ + (\widetilde C_- \widetilde B_- + D) u \|^2.
\label{EBstrict'bi}
\end{align}
{\em for the strict case.}  Indeed,  by translation invariance of the system equations, it suffices to check the bicausal
energy-balance condition \eqref{EB} only at $n=0$ for any $\ell^2$-admissible trajectory $(\bu, \bx, \by)$.  In terms of
\begin{equation}  \label{bi-int}
x_{-} := \bx_{-}(n+1), \quad x_{+}: = \bx_{+}(n), \quad u: = \bu(n),
\end{equation}
we can solve for the other quantities appearing in \eqref{EB} for the case $n = 0$:
\begin{align*}
    \bx_{+}(1) & = \widetilde A_{+} x_{+} + \widetilde B_{+} u, \\
    \bx_{-}(0) & = \widetilde A_{-} x_{-} + \widetilde B_{-} u, \\
    \by(0) & =  \widetilde C_{-} \widetilde A_{-} x_{-} + C_{+} x_{+}
    + ( \widetilde C_- \widetilde B_- + \widetilde D) u.
\end{align*}
Then the energy-balance condition \eqref{EB} for the bicausal system $\Sigma$ reduces to \eqref{EB'bi}, so \eqref{EB'bi}
is a sufficient condition for $S$ to be a storage function (assuming conditions (1) and (3) in the definition of a storage
function also hold).
Conversely, given any $x_- \in \cX_-$, $x_+ \in \cX_+$, $u \in \cU$, the trajectory-interpolation result Proposition
\ref{P:traj-int-bi}  assures us that we can always embed the vectors $x_-$, $x_+$, $u$ into an $\ell^2$-admissible trajectory
so that  \eqref{bi-int} holds. We then see that condition \eqref{EB'bi} holding for all $x$, $x'$, $u$ is also  necessary for
$S$ to be  a storage function.  The strict version works out in a similar way, again by making  use of the interpolation result
Proposition \ref{P:traj-int-bi} .
\end{remark}

    We next define functions $S_a \colon \cX \to {\mathbb R}$ and $S_r \colon \cX \to {\mathbb R}$ via the formulas
    \eqref{dichotSa} and \eqref{dichotSr} but with $\bW_c$ taken to be the controllability operator as in
    \eqref{bicausalconobs} for a bicausal system.
    One can also check that  the following bicausal version of Proposition \ref{P:dichotSa} holds, but again with
    the verification of the continuity property for $S_a$ and $S_r$ postponed until more detailed information concerning
    $S_a$ and $S_r$ is developed below.

    \begin{proposition}   \label{P:dichotSa-bic}
    Let $\Sigma = (\Sigma_-, \Sigma_+)$ be a bicausal system as in \eqref{dtsystem-b}--\eqref{dtsystem-f}, with $\widetilde A_\pm$ exponentially stable.
    Assume that \eqref{Con} holds.  Then:
    \begin{enumerate}
    \item $S_a$ is a storage function for $\Sigma$.

    \item $S_r$ is a storage function for $\Sigma$.

    \item If $\widetilde S$ is any storage function for $\Sigma$, then
    $$
    S_a(x_0) \le \widetilde S(x_0) \le S_r(x_0) \text{ for all } x_0 \in\cX.
    $$
    \end{enumerate}
    \end{proposition}

\begin{proof}[\bf Proof]  The proof of Proposition \ref{P:dichotSa}  for the causal dichotomous setting extends verbatim to the
bicausal setting once we verify that the patching technique of Lemma \ref{L:patch} holds in exactly the same form for
the bicausal setting.  We therefore suppose that
\begin{equation} \label{bic-traj}
(\bu', \bx', \by'), \quad (\bu'', \bx'', \by'')
\end{equation}
are $\ell^2$-admissible trajectories for the bicausal system $\Sigma$ such that $\bx'(0) = \bx''(0)$.  In more detail,
this means that there are two $\ell^2$-admissible system
trajectories of the form $(\bu', \bx_-', \by_-')$ and $(\bu'', \bx_-'', \by_-'')$  for the anticausal system $\Sigma_-$ such that
$\bx_-'(0) = \bx_-''(0)$ and two $\ell^2$-admissible system trajectories of the form $(\bu', \bx_+', \by_+')$ and
$(\bu'', \bx_+'', \by_+'')$ for the causal system $\Sigma_+$ with $\bx_+'(0) = \bx_+''(0)$ such that we recover the state
and output components of the original trajectories for the bicausal system \eqref{bic-traj}  via
\begin{align*}
\bx'(n) = \bx_-'(n) \oplus \bx'_+(n),  \quad & \bx''(n) = \bx_-''(n) \oplus \bx_+''(n), \\
\by'(n) = \by'_-(n) + \by'_+(n), \quad & \by''(n) = \by''_-(n) + \by''_+(n).
\end{align*}
Let us define a composite input trajectory by
$$
    \bu(n) = \begin{cases}  \bu'(n) &\text{if } n< 0,  \\ \bu''(n) &\text{if } n \ge 0.  \end{cases}
$$
We apply the causal patching lemma to the system $\Sigma_+$ (having trivial dichotomy) to see that the composite trajectory
$(\bu, \bx_+, \by_+)$ with state and output given by
$$
  \bx_+(n) = \begin{cases} \bx'_+(n) &\text{if } n \le 0, \\ \bx_+''(n) &\text{if } n > 0, \end{cases}  \quad
  \by_+(n) =  \begin{cases} \by'_+(n) &\text{if } n < 0, \\ \by_+''(n) &\text{if } n \ge 0, \end{cases}
$$
is an $\ell^2$-admissible trajectory for the causal system $\Sigma_+$.  Similarly, we apply a reversed-orientation version of
the patching given by Lemma \ref{L:patch} to see that the trajectory $(\bu, \bx_-, \by_-)$ with state and output given by
$$
\bx_-(n) = \begin{cases}  \bx'_-(n) &\text{if } n<0, \\ \bx'_-(n) &\text{if } n \ge 0, \end{cases} \quad
\by_-(n) = \begin{cases} \by'_-(n)  &\text{if } n<0, \\ \by''_-(n) &\text{if } n \ge 0 \end{cases}
$$
is an $\ell^2$-admissible system trajectory for the anticausal system $\Sigma_-$.  It then follows from the definitions that
the composite trajectory $(\bu, \bx, \by)$ given by  \eqref{patch} is an $\ell^2$-admissible system trajectory for the bicausal
system $\Sigma$ as wanted.
 \end{proof}

\paragraph{\bf Quadratic storage functions and spatial KYP-inequalities: the bicausal setting.}
We define a quadratic function $S \colon \cX \to {\mathbb R}$ as was done at the end of Section \ref{S:storage} above:
$S(x) = \langle H x, x \rangle$ where $H$ is a bounded selfadjoint operator on $\cX$.  For the bicausal setting, we wish to make
explicit that $\cX$ has a decomposition as $\cX = \cX_- \oplus \cX_+$ which we now wish to write as a column decomposition
$\cX = \sbm{ \cX_- \\ \cX_+ }$.  After a change of coordinates which we choose not to go through explicitly, we may assume
that this decomposition is orthogonal.  Then any selfadjoint operator $H$ on $\cX$ has a $2 \times 2$ matrix representation
\begin{equation}\label{Hdec}
H = \mat{cc}{ H_{-} & H_{0} \\ H_{0}^{*} & H_{+} }\ons \cX=\mat{c}{\cX_-\\ \cX_+}.
\end{equation}
with associated quadratic function $S_H$ now given by
$$
S_{H}(x_- \oplus x_+) = \langle H (x_- \oplus x_+), x_- \oplus x_+ \rangle = \left\langle \begin{bmatrix}
H_{-} & H_{0} \\ H_{0}^{*} & H_{+} \end{bmatrix} \begin{bmatrix}
x_{-} \\ x_{+} \end{bmatrix},\, \begin{bmatrix} x_{-} \\ x_{+} \end{bmatrix} \right\rangle.
$$
If we apply the local criterion for a given function $S$ to be a storage function in the bicausal setting as given
by Remark  \ref{R:local-storage-bi},  we arrive at the following criterion for $S_H$ to be a storage function for the
bicausal system $\Sigma$:
\begin{align*}
& \left\langle \sbm{ H_{-} & H_{0} \\ H_{0}^{*} & H_{+} }
\sbm{ x_-  \\  \widetilde A x_+ + \widetilde B_+ u},\,
 \sbm{  x_-  \\ \widetilde A x_+ + \widetilde B_+ u } \right\rangle  \\
& \quad  \quad \quad \quad - \left\langle \sbm{ H_{-} & H_{0} \\ H_{0}^{*} & H_{+} }
\sbm{ \widetilde A_- x_- + \widetilde B_- u  \\ x_+ }, \,
\sbm{  \widetilde A_- x_- + \widetilde B_- u  \\ x_+  }
    \right\rangle  \\
&    \le \| u \|^{2} - \| \widetilde C_{-} \widetilde A_{-} x_{-} +
\widetilde C_+ x_+ + (\widetilde C_- \widetilde B_ -  + \widetilde D) u \|^{2}.
\end{align*}
which amounts to the spatial version of the bicausal KYP-inequality
\eqref{KYP-bicausal}.

Similarly, $S_{H}$ is a strict storage function exactly when there is
an $\epsilon > 0$  so that
\begin{align*}
& \left\langle \sbm{ H_{-} & H_{0} \\ H_{0}^{*} & H_{+} }
\sbm{ x_-  \\  \widetilde A x_+ + \widetilde B_+ u},\,
 \sbm{  x_-  \\ \widetilde A x_+ + \widetilde B_+ u } \right\rangle  \\
& \quad \quad \quad \quad - \left\langle \sbm{ H_{-} & H_{0} \\ H_{0}^{*} & H_{+} }
\sbm{ \widetilde A_- x_- + \widetilde B_- u  \\ x_+ }, \,
\sbm{  \widetilde A_- x_- + \widetilde B_- u  \\ x_+  }
    \right\rangle    \\
& \quad \quad \quad \quad  + \epsilon^2 ( \| \widetilde A_- x_- + \widetilde B_- u \|^2 + \| x_+\|^2 + \| u \|^2 ) \\
&    \le  \| u \|^{2} - \| \widetilde C_{-} \widetilde A_{-} x_{-} +
\widetilde C_+ x_+ + (\widetilde C_- \widetilde B_ -  + \widetilde D) u \|^{2}.
\end{align*}
 One can check that this is just the spatial version of the strict bicausal
KYP-inequality \eqref{KYP-bicausal-strict}.
 One can now check that the assertion of
  Proposition \ref{P:quadstorage} goes through as stated with the {\em
  dichotomous linear systems in \eqref{dtsystem}} replaced by a {\em
  bicausal system} \eqref{dtsystem-b}--\eqref{dtsystem-f}
 (with $\wtilA_{+}$ and $\widetilde A_{-}$ both
  exponentially stable), and with {\em KYP-inequality} (respectively {\em strict
  KYP-inequality} \eqref{KYPstrict}) replaced with {\em bicausal KYP-inequality}
  \eqref{KYP-bicausal} (respectively {\em strict bicausal
  KYP-inequality} \eqref{KYP-bicausal-strict}).  We have thus arrived at the following
  extension of Proposition \ref{P:quadstorage} to the bicausal setting.

  \begin{proposition}   \label{P:quadstor-bic}
  Suppose that $\Sigma = (\Sigma_- ,\Sigma_+)$ is a bicausal system \eqref{dtsystem-b}--\eqref{dtsystem-f}, with
  $\widetilde A_\pm$ exponentially stable.  Let $H$ be a selfadjoint operator as in \eqref{Hdec}, where we assume
  that coordinates are chosen so that the decomposition $\cX = \cX_- \oplus \cX_+$ is orthogonal.  Then $S_H$ is a quadratic
  storage function for $\Sigma$ if and only if $H$ is a solution of the bicausal KYP-inequality \eqref{KYP-bicausal}.
  Moreover, $S_H$ is a strict storage function for $\Sigma$ if and only if $H$ is a solution of the strict bicausal KYP-inequality
  \eqref{KYP-bicausal-strict}.
 \end{proposition}

  Furthermore, as noted in Section \ref{S:bicausal}, the Hankel
  factorizations \eqref{HankelfactBicausal} also hold in the bicausal setting.
  Hence Lemma \ref{L:Xops} goes through as stated, the only
  modification being the adjustment of the formulas for the operators
  $\bW_{o}^{\pm}$, $\bW_{c}^{\pm}$ to those in \eqref{con-obsdichot-nonreg}
  (rather than \eqref{dichotcon}, \eqref{dichotobs}).  It then follows that Theorem \ref{T:SaSr-Con}
  holds with exactly the same formulas \eqref{Ha}, \eqref{Hr},
  \eqref{Ha-explicit}, \eqref{Hr-explicit} for $S_{a}$, $S_{r}$,
  $H_{a}$ and $H_{r}$, again with the adjusted formulas for the
  operators $\bW_{o}^{\pm}$ and $\bW_{c}^{\pm}$.  As $S_a = S_{H_a}$ and $S_r = S_{H_r}$ with
  $H_a$ and $H_r$ bounded and boundedly invertible selfadjoint operators on $\cX_- \oplus \cX_+$, it follows that
  $S_a$ and $S_r$ are continuous, completing the missing piece in the proof of Proposition \ref{P:dichotSa-bic} above.
  We have arrived at the following extension of Theorem \ref{T:SaSr-Con} to the bicausal setting.

  \begin{theorem}  \label{T:SaSr-Con-bic}  Suppose that $\Sigma = (\Sigma_-, \Sigma_+)$
  is a bicausal system \eqref{dtsystem-b}--\eqref{dtsystem-f}, with $\widetilde A_\pm$ exponentially stable, satisfying the
  standing hypothesis \eqref{Con}.  Define the
  available storage $S_a$ and the required supply $S_r$ as in \eqref{dichotSa}--\eqref{dichotSr} (properly interpreted for the
  bicausal rather than dichotomous setting).   Then $S_a$ and $S_r$ are continuous.  In detail, $S_a$ and $S_r$
  are given by the formulas \eqref{Ha}--\eqref{Hr},  or equivalently,  $S_a = S_{H_a}$ and $S_r = S_{H_r}$ where
  $H_a$ and $H_r$ are given explicitly as in \eqref{Ha-explicit} and \eqref{Hr-explicit}.
   \end{theorem}

\begin{remark}   \label{R:bic/dichot}
A nice exercise is to check that the bicausal KYP-inequality \eqref{KYP-bicausal} collapses to the
standard KYP-inequality \eqref{KYP} in the case that $\widetilde A_-$ is invertible so that the the
bicausal system $\widetilde \Sigma$ can be converted to a dichotomous system as in Remark \ref{R:dichot-bicausal}.
Let us assume that $\widetilde \Sigma$ is a bicausal system as in \eqref{dtsystem-b} and \eqref{dtsystem-f}.  We assume
that $\widetilde A$ is invertible and we make the substitution \eqref{dichot-bicausal} to convert to a dichotomous linear system
as in \eqref{dtsystem}, \eqref{Adec}, \eqref{BCdec}.  The resulting bicausal KYP-inequality then becomes
\begin{align}
& \sbm{ I & 0 & A_-^{-1*} C_-^* \\ 0 & A_+^* & C_+^* \\ 0 & B_+^* & -B_-^* A_-^{*-1} C_-^* + D^* }
\sbm{ H_- & H_0 & 0 \\ H_0^* & H_+ & 0 \\ 0 & 0 & I }
\sbm{ I & 0 & 0 \\ 0 & A_+ & B_+ \\ C_- A_-^{-1} & C_+ & -C_- A_-^{-1} B_- + D}  \notag \\
& \quad \quad
 \preceq \sbm{ A_-^{-1*} & 0 & 0 \\ 0 & I & 0 \\ -B_-^* A_-^{*-1} & 0 & I }
\sbm{ H_- & H_0 \\ H_0^* & H_+ & 0 \\ 0 & 0 & I }
\sbm{ A_-^{-1} & 0 & -A_-^{-1} B_- \\ 0 & I & 0 \\ 0 & 0 & I }.
\label{KYP-1}
\end{align}
However the spatial version of the bicausal KYP-inequality \eqref{KYP-bicausal} corresponds to the quadratic form based
at the vector $\sbm{ \bx_-(1) \\ \bx_+(0) \\ \bu(0)}$ while the spatial version of the dichotomous (causal) KYP-inequality \eqref{KYP}
is the quadratic form based at the vector $\sbm{ \bx_-(0) \\ \bx_+(0) \\\ \bu(0) }$, where the conversion from the latter to the former
is given by
$$
   \sbm{\bx_-(0) \\ \bx_+(0) \\ \bu(0)  } = \sbm{ A_- & 0 & B_- \\ 0 & I & 0 \\ 0 & 0 & I } \sbm{ \bx_-(1) \\ \bx_+(0) \\ \bu(0) }.
$$
To recover the dichotomous KYP-inequality \eqref{KYP} from \eqref{KYP-1}, it therefore still remains to conjugate both sides
of \eqref{KYP-1} by $T = \sbm{ A_- & 0 & B_- \\ 0 & I & 0 \\ 0 & 0 & I}$ (i.e., multiply on the right by $T$ and on the left by $T^*$).
Note next that
\begin{align*}
& \sbm{ I & 0 & 0 \\ 0 & A_+ & B_+ \\ C_- A_-^{-1}  & C_+ & -C_- A_-^{-1} B_- + D }
\sbm{ A_- & 0 & B_- \\ 0 & I & 0 \\ 0 & D & 0 }
= \sbm{ A_- & 0 & B_- \\ 0 & A_+ & B_+ \\ C_- & C_+ & D},  \\
& \sbm{ A_-^{-1} & 0 & -A_-^{-1} B \\ 0 & I & 0 \\ 0 & 0 & I}
\sbm{ A_- & 0 & B_- \\ 0 & I & 0 \\ 0 & 0 & I }
= \sbm{ I & 0 & 0 \\ 0 & I & 0 \\ 0 & 0 & I}.
 \end{align*}
 Hence  conjugation of both sides of \eqref{KYP-1} by $T$ results in
 $$
  \sbm{ A_-^* & 0 & C_-^* \\ 0 & A_+^* & C_+^* \\ B_-^* & B_+^* & D^* }
 \sbm{ H_- & H_0 & 0 \\ H_0^* & H_- & 0 \\ 0 & 0 & I }
 \sbm{ A_- & 0 & B_- \\ 0 & A_- & B_+ \\ C_- & C_+ & D}
 \preceq \sbm{H_- & H_0 & 0 \\ H_0^* & H_- & 0 \\ 0 & 0 & I}
 $$
 which is just the dichotomous KYP-inequality \eqref{KYP}  written out when the matrices
 are expressed in the decomposed form \eqref{Adec}, \eqref{BCdec}, \eqref{Hdec}.

 The connection between the strict KYP-inequalities for the bicausal setting \eqref{KYP-bicausal-strict} and the dichotomous setting
 \eqref{KYPstrict} works out similarly.  In fact all the results presented here for dichotomous systems follow from the corresponding
 result for the bicausal setting by restricting to the associated bicausal system having $\widetilde A_- = A_-^{-1}$ invertible.
\end{remark}

 \section{Dichotomous and bicausal Bounded Real Lemmas}  \label{S:BRLs}

In this section we derive infinite-dimensional versions of the finite-dimensional Bounded Real Lemmas stated in the introduction.

Combining the results of Propositions \ref{P:dichotStoragefunc},
\ref{P:dichotSa}, \ref{P:quadstorage} and Theorem \ref{T:SaSr-Con}
leads us to the following infinite-dimensional version of the
standard dichotomous Bounded Real Lemma; this result
contains Theorem \ref{T:finite-dichotBRL} (1), as stated in the introduction, as a corollary.

\begin{theorem}  \label{T:dichotBRL}
    {\rm \textbf{Standard dichotomous Bounded Real Lemma:}}
Assume that the linear system $\Sigma$ in \eqref{dtsystem} has a
dichotomy and is dichotomously $\ell^{2}$-exactly controllable and
observable (both $\bW_{c}$ and $\bW_{o}^*$ are surjective).  Then the
following are equivalent:

\begin{enumerate}
    \item  $\| F_{\Sigma} \|_{\infty, {\mathbb T}} : = \sup_{z \in {\mathbb
    T}} \| F_{\Sigma}(z) \| \le 1.$

    \item There is a bounded and boundedly invertible selfadjoint
    operator $H$ on $\cX$ which satisfies the  KYP-inequality
    \eqref{KYP}.  Moreover, the dimension of the spectral subspace of
    $A$ over the unit disk (respectively, exterior of the closed unit
    disk)  agrees with the dimension of the spectral  subspace of $H$
    over the positive real line (respectively, over the negative real
    line).
 \end{enumerate}
  \end{theorem}

We shall next show how the infinite-dimensional version of the strict
dichotomous Bounded Real Lemma (Theorem \ref{T:finite-dichotBRL} (2))
can be reduced to the standard version (Theorem \ref{T:dichotBRL}) by
the same technique used for the stable (non-dichotomous) case (see
\cite{PAJ, KYP1, KYP2}). The result is as follows; the reader can check
that specializing the result to the case where all signal spaces
$\cU$, $\cX$, $\cY$ are finite-dimensional results in Theorem
\ref{T:finite-dichotBRL} (2) from the introduction as a corollary. Note that, as in the non-dichotomous case
(see \cite[Theorem 1.6]{KYP1}), there is no controllability or observability condition required here.

\begin{theorem} \label{T:strictdichot} {\rm \textbf{Strict dichotomous
    Bounded Real Lemma:}}
Assume that the linear system $\Sigma$ in \eqref{dtsystem} has a
dichotomy.  Then the following are equivalent:

\begin{enumerate}
    \item $\| F_{\Sigma} \|_{\infty, {\mathbb T}} : = \sup_{z \in {\mathbb
    T}} \| F_{\Sigma}(z) \| < 1.$

    \item There is a bounded and boundedly invertible selfadjoint
    operator $H$ on $\cX$ which satisfies the strict KYP-inequality
    \eqref{KYPstrict}.  Moreover the inertia of $A$  partitioned by the unit circle lines up with the
    inertia of $H$ (partitioned by the point $0$ on the real line) as  in the
    standard dichotomous Bounded Real Lemma (Theorem \ref{T:dichotBRL} above).
\end{enumerate}
\end{theorem}

\begin{proof}[\bf Proof]
    The proof of (2) $\Rightarrow$ (1) is a consequence of
    Propositions \ref{P:dichotStoragefunc}, \ref{P:dichotSa}, and
    \ref{P:quadstorage}, so it suffices to prove (1) $\Rightarrow$
    (2). To simplify the notation, we again write $F$ rather than
    $F_{\Sigma}$ throughout this proof.

    We therefore assume that $\| F \|_{\infty, {\mathbb T}} < 1$.  For
    $\epsilon > 0$, we let $\Sigma_{\epsilon}$ be the discrete-time
    linear system \eqref{dtsystem} with system matrix $M_{\epsilon}$
    given by
\begin{equation}  \label{Mepsilon}
 M_{\epsilon} =  \begin{bmatrix}  A & B_{\epsilon} \\ C_{\epsilon} &
 D_{\epsilon} \end{bmatrix} : =
  \mat{c|cc}{
    A &  B & \epsilon I_{\cX}\\
    \hline C &  D & 0 \\
    \epsilon I_{\cX} & 0 & 0 \\
    0  & \epsilon I_{\cU} & 0}
\end{equation}
with associated transfer function
\begin{align}
 F_{\epsilon}(z) & = \begin{bmatrix} D & 0 \\ 0 & 0 \\ \epsilon I_{\cU}
 & 0 \end{bmatrix} + z \begin{bmatrix}  C \\ \epsilon I_{\cX} \\ 0
 \end{bmatrix} (I - zA)^{-1} \begin{bmatrix} B & \epsilon I_{\cX}
 \end{bmatrix} \notag \\
 & = \begin{bmatrix} F(z) & \epsilon z C (I -
    zA)^{-1} \\ \epsilon z (I - zA)^{-1} B & \epsilon^{2} z (I -
    zA)^{-1} \\ \epsilon I_{\cU} & 0 \end{bmatrix}.
\label{Fepsilon}
\end{align}
As $M$ and $M_{\epsilon}$ have the same state-dynamics operator $A$,
the system $\Sigma_{\epsilon}$ inherits the dichotomy property from $\Sigma$.
As the resolvent expression $z (I - zA)^{-1}$ is uniformly bounded
in norm on ${\mathbb T}$, the fact that $\| F \|_{\infty, {\mathbb T}} <  1$
implies that $\| F_{\epsilon} \|_{\infty, {\mathbb T}}  < 1$ as long as $\epsilon > 0$ is
chosen sufficiently small.  Moreover, when we decompose
$B_{\epsilon}$ and $C_{\epsilon}$ according to \eqref{BCdec}, we get
\begin{align*}
B_{\epsilon} & = \begin{bmatrix} B_{-} & \epsilon I_{\cX_{-}} & 0 \\
   B_{+} & 0 & \epsilon I_{\cX_{+}} \end{bmatrix} \colon
   \begin{bmatrix} \cU \\ \cX_{-} \\ \cX_{+} \end{bmatrix} \to
   \begin{bmatrix} \cX_{-} \\ \cX_{+} \end{bmatrix}, \\
   C_{\epsilon} & = \begin{bmatrix} C_{-} & C_{+} \\ \epsilon
   I_{\cX_{-}} & 0 \\ 0 & \epsilon I_{\cX_{+}} \\ 0 & 0 \end{bmatrix}
   \colon \begin{bmatrix} \cX_{-} \\ \cX_{+} \end{bmatrix} \to
   \begin{bmatrix}  \cY \\ \cX_{-} \\ \cX_{+} \\ \cU \end{bmatrix},
\end{align*}
or specifically
\begin{align*}
B_{\epsilon -} = \begin{bmatrix} B_{- }& \epsilon I_{\cX_{-}} & 0 \end{bmatrix}, \quad &
B_{\epsilon +} = \begin{bmatrix} B_{+} & 0 & \epsilon I_{\cX_{+}} \end{bmatrix}, \\
C_{\epsilon -} = \begin{bmatrix} C \\ \epsilon I_{\cX_{-}} \\ 0 \\ 0 \end{bmatrix}, \quad &
C_{\epsilon +} = \begin{bmatrix} C_{+} \\ 0 \\ \epsilon I_{\cX_{+}} \\ 0 \end{bmatrix}.
\end{align*}
Hence we see that $(A_{+}, B_{\epsilon +})$ is exactly controllable
in one step and hence is $\ell^{2}$-exactly controllable.  Similarly,
$(A_{-}^{-1}, A_{-}^{-1} B_{\epsilon -})$ is $\ell^{2}$-exactly controllable
and both $(C_{\epsilon+}, A_{+})$ and $(C_{\epsilon -} A_{-}^{-1},
A_{-}^{-1})$ are $\ell^{2}$-exactly observable. As we also have $\| F_{\epsilon} \|_{\infty, {\mathbb T}} < 1$,
 in particular $\| F_{\epsilon} \|_{\infty, {\mathbb T}} \le 1$, so Theorem
\ref{T:dichotBRL} applies to the system $\Sigma_{\epsilon}$.
We conclude that there is bounded, boundedly invertible, selfadjoint
operator $H_{\epsilon}$ on $\cX$ so that the KYP-inequality holds
with respect to the system $\Sigma_{\epsilon}$:
$$
 \begin{bmatrix} A^{*} & C_{\epsilon}^{*} \\ B_{\epsilon}^{*} &
     D_{\epsilon}^{*} \end{bmatrix} \begin{bmatrix} H_{\epsilon} & 0
     \\ 0 & I_{\cY \oplus \cX \oplus \cU} \end{bmatrix}
 \begin{bmatrix} A & B_{\epsilon} \\ C_{\epsilon} & D_{\epsilon}
 \end{bmatrix} \preceq \begin{bmatrix} H_{\epsilon} & 0 \\ 0 & I_{\cU
 \oplus \cX} \end{bmatrix}.
 $$
 Spelling this out gives
$$
\begin{bmatrix} A^{*}H_{\epsilon}A + C^{*}C + \epsilon^{2}I_{\cX} \!&\!
    A^{*}H_{\epsilon} B +
    C^{*}D \!&\! \epsilon A^{*}H_{\epsilon} \\
 B^{*}H_{\epsilon}A + D^{*}C \!&\! B^{*} H_{\epsilon} B + D^{*} D + \epsilon^{2}I_{\cU} \!&\!
 \epsilon B^{*} H_{\epsilon} \\ \epsilon H_{\epsilon}A \!&\! \epsilon H_{\epsilon}B
 \!&\! \epsilon^{2} H_{\epsilon}
\end{bmatrix} \preceq \begin{bmatrix} H_{\epsilon} \!&\! 0 \!&\! 0 \\ 0 \!&\! I_{\cU} \!&\! 0 \\
0 \!&\! 0 \!&\! I_{\cX} \end{bmatrix}.
$$
By crossing off the third row and third column, we get the inequality
$$
\begin{bmatrix} A^{*} H_{\epsilon} A + C^{*} C + \epsilon^{2} I_{\cX}
    & A^{*}H_{\epsilon} B
    + C^{*} D \\ B^{*}H_{\epsilon} A + D^{*} C & B^{*} H_{\epsilon} B + D^{*} D +
    \epsilon^{2} I_{\cU} \end{bmatrix} \preceq \begin{bmatrix}
    H_{\epsilon} & 0
    \\ 0 & I_{\cU} \end{bmatrix}
$$
or
$$
\begin{bmatrix} A^{*} & C^{*} \\ B^{*} & D^{*} \end{bmatrix}
    \begin{bmatrix} H_{\epsilon} & 0 \\ 0 & I_{\cY} \end{bmatrix}
	\begin{bmatrix} A & B \\ C & D \end{bmatrix} + \epsilon^{2}
	    \begin{bmatrix} I_{\cX} & 0 \\ 0 & I_{\cU} \end{bmatrix}
		\preceq \begin{bmatrix} H_{\epsilon} & 0 \\ 0 & I_{\cU}
	    \end{bmatrix}.
$$
We conclude that $H_{\epsilon}$ serves as a solution to the strict KYP-inequality
\eqref{KYPstrict} for the original system $\Sigma$ as wanted.
\end{proof}

The results in Section \ref{S:bicausal-storage} for bicausal systems lead to the
following extensions of Theorems \ref{T:dichotBRL} and
    \ref{T:strictdichot} to the bicausal setting; note that Theorem
    \ref{T:finite-bicausalBRL} in the introduction follows as a corollary of this result.

    \begin{theorem} \label{T:bicausalBRL}
    Suppose that $\Sigma =(\Sigma_{+}, \Sigma_{-})$ is a bicausal
    linear system with subsystems $\Si_+$ and $\Si_-$ as in \eqref{dtsystem-b} and \eqref{dtsystem-f}, respectively, with both $A_{+}$    and $A_{-}$ exponentially stable and with associated transfer
    function $F_{\Sigma}$ as in \eqref{FSigBiC-Laurent}.

    \begin{enumerate}
\item 	Assume that $\Sigma$ is $\ell^{2}$-exactly minimal, i.e., the
	operators $\bW_{c}^{+}$, $\bW_{c}^{-}$, $(\bW_{o}^{+})^{*}$,
	$(\bW_{o}^{-})^{*}$ given by \eqref{con-obsdichot-nonreg} are
	all surjective.  Then $\| F_{\Sigma} \|_{\infty, {\mathbb T}} \le 1$
	if and only if there exists a bounded and boundedly
	invertible selfadjoint solution $H = \sbm{ H_{-} & H_{0} \\
	H_{0}^{*} & H_{+} }$ of the bicausal KYP-inequality
	\eqref{KYP-bicausal}.
	
\item Furthermore, $\| F_{\Sigma}\|_{\infty, {\mathbb T}} < 1$ holds if and only if there is
a bounded and boundedly invertible selfadjoint solution $H = \sbm{
H_{-} & H_{0} \\ H_{0}^{*} & H_{+} }$ of the strict bicausal
KYP-inequality \eqref{KYP-bicausal-strict}.
\end{enumerate}
\end{theorem}

\begin{proof}[\bf Proof]
To verify item (1), simply  combine the results of Propositions \ref{P:Storagefunc-bic}, \ref{P:dichotSa-bic}, \ref{P:quadstor-bic} and Theorem \ref{T:SaSr-Con-bic}.

    As for item (2), note that sufficiency follows already from the
    stream of Propositions \ref{P:Storagefunc-bic}, \ref{P:dichotSa-bic} and \ref{P:quadstor-bic}. As for necessity, let us verify that the same
    $\epsilon$-augmented-system technique as used in the proof of
    Theorem \ref{T:strictdichot} can be used to reduce the strict
    case of the result (item (2)) to the standard case  (item (1)).
    Let us rewrite the bicausal KYP-inequality \eqref{KYP-bicausal} as
\begin{align}
& \sbm{ H_- + \widetilde A_-^* \widetilde C_-^* \widetilde C_- \widetilde A_-
& H_0 \widetilde A_+ + \widetilde A_-^* \widetilde C_-^* \widetilde C_+
& H_0 \widetilde B_+ + \widetilde A_-^* \widetilde C_-^*  \widehat D
 \\
 \widetilde A_+^* H_0^* + \widetilde C_+^* \widetilde C_- \widetilde A_-
& \widetilde A_+^* H_+ \widetilde A_+ + \widetilde C_+^* \widetilde C_+
& \widetilde A_+^* H_+ \widetilde B_+ + \widetilde C_+^*  \widehat D
\\
\widetilde B_+^* H_0^* + \widehat D^*  \widetilde C_- \widetilde A_-
& \widetilde B_+^* H_+ \widetilde A_+ +  \widehat D^*  \widetilde C_+
& \widetilde B_+^* H_+ \widetilde B_+ +  \widehat D^* \widehat D} \notag \\
& \qquad\qquad\qquad \preceq
\sbm{ \widetilde A_-^* H_- \widetilde A_- & \widetilde A_-^* H_0 & \widetilde A_-^* H_- \widetilde B_-
\\
H_0^* \widetilde A_- & H_+ & H_0^*  \widetilde B_-
\\
\widetilde B_-^* H_- \widetilde A_- & \widetilde B_-^* H_0 &  \widetilde B_-^* H_- \widetilde B_- + I}
\label{bicausal-KYP'}
\end{align}
where we set
$$
    \widehat D = \widetilde C_- \widetilde B_- + \widetilde D.
$$
Let us now consider the $\epsilon$-augmented system matrices
\begin{align*}
& M_{+,\epsilon} = \mat{ccc} {\widetilde A_{+} & \widetilde B_{+ ,\epsilon} \\ \widetilde  C_{+ ,\epsilon} &
\widetilde D_{\epsilon} } =
\sbm{ \widetilde A_{+} & | & \widetilde B_{+} & \epsilon I_{\cX_{+}} \\ \hline \\
    \widetilde C_{+} & | & \widetilde D & 0 \\ \epsilon I_{\cX_{+}} & | & 0 & 0 \\ 0 & | &
    \epsilon I_{\cU} & 0 } \colon \sbm{ \cX_{+} \\ \hline \\ \cU \\
    \cX_{+} } \to \sbm{ \cX_{+} \\ \hline \\ \cY \\ \cX_{+} \\ \cU },
 \\ &
M_{-,\epsilon} = \begin{bmatrix} \widetilde A_{-} & \widetilde B_{- ,\epsilon} \\
    \widetilde C_{- ,\epsilon} & 0 \end{bmatrix} =
\sbm{ \widetilde A_{-} & | & \widetilde B_{-} & \epsilon I_{\cX_{-}}
\\ \hline \\
\widetilde C_{-} & | & 0 & 0 \\ \epsilon I_{\cX_{-}} & | & 0 & 0
\\ 0 & | & 0 & 0 } \colon \sbm{ \cX_{-} \\ \hline \\
\cU \\ \cX_{-} } \to
\sbm{ \cX_{-} \\ \hline  \\ \cY \\ \cX_{-} \\ \cU }.
\end{align*}
Then the system matrix-pair $(M_{\epsilon, +}, M_{\epsilon, -})$
defines a bicausal system $\Si_\epsilon=(\Si_{\ep,+},\Si_{\ep,-})$ where the subsystem $\Si_{\ep,+}$ is associated the system matrix $M_{\epsilon, +}$ and the subsystem $\Si_{\ep,-}$ is associated the system matrix $M_{\epsilon, -}$. Note that the state-dynamics operators $\widetilde A_{+}$ and
$\widetilde A_{-}$ of $\Si_\ep$ are exponentially stable and has transfer function $F_{\epsilon}$
given by
\begin{align*}
&  F_{\epsilon}(z)  =
 \sbm{ \widetilde D & 0 \\ 0 & 0 \\ \epsilon I_{\cU}
 & 0 }    + z \sbm{ \widetilde C_{+} \\ \epsilon I_{\cX_{+}}
 \\ 0 } (I - z \widetilde A_{+})^{-1} \sbm{ \widetilde B_{+} &
 \epsilon I_{\cX_{+}} }   \\
 & \quad +
 \sbm{ \widetilde C_{-} \\ \epsilon I_{\cX_{-}} \\ 0 } (I - z^{-1} \widetilde A_{-})^{-1} \sbm{
 \widetilde B_{-} & \epsilon I_{\cX_{-}}} \\
 & = \sbm{ F(z) & \epsilon z \widetilde C_{+} (I - z \widetilde A_{+})^{-1} +
 \epsilon \widetilde C_{-} (I - z \widetilde A_{-})^{-1} \\
\epsilon z (I - z \widetilde A_{+})^{-1} \widetilde B_{+} + \epsilon (I - z^{-1} \widetilde
 A_{-})^{-1}\widetilde B_{-}
 & \epsilon^{2} z (I - z \widetilde A_{+})^{-1} + \epsilon^{2}(I - z^{-1}\widetilde A_{-})^{-1}
 \\  \epsilon I_{\cU} & 0 }.
\end{align*}
Since by assumption the transfer function $F$ associated with the
original bicausal system $\Sigma = ( \Si_+, \Si_-)$
 has norm $\| F \|_{\infty, \mathbb T} < 1$ and both
$\widetilde A_{+}$ and $\widetilde A_{-}$ are exponentially stable, it is clear that we can
maintain $\| F_{\epsilon} \|_{\infty, {\mathbb T}} < 1$ with $\epsilon > 0$ as
long as we choose $\epsilon$ sufficiently small.  Due to the presence
of the identity matrices in the input and output operators for $M_{+,
\epsilon}$ and $M_{-, \epsilon}$, it is clear that the bicausal system
$\Sigma_{\epsilon}$ is $\ell^{2}$-exactly controllable and $\ell^{2}$-exactly observable
in the bicausal sense.  Then statement (1) of the present theorem (already verified)
applies and we are assured
that there is a bounded and boundedly invertible solution
$H = \sbm{ H_{+} & H_{0} \\ H_{0}^{*} & H_{-} }$ of the
bicausal KYP-inequality \eqref{bicausal-KYP'} associated with
$\Si_\ep=\left(\Si_{+,\epsilon}, \Si_{- ,\epsilon} \right)$.
Replacing $\widetilde B_{\pm}$, $\widetilde C_{\pm}$ and $\widetilde D$  in \eqref{bicausal-KYP'} by
$\widetilde B_{\pm,\ep}$, $\widetilde C_{\pm,\ep}$ and $\widetilde D_\ep$
leads to the $\epsilon$-augmented
version of the bicausal KYP-inequality:
\begin{align}
& \sbm{
H_- + \widetilde A_-^* \widetilde C_-^* \widetilde C_-  \widetilde A_- + \epsilon^2 \widetilde A_-^* \widetilde A_-
\!\!&\! H_0 \widetilde A_+ + \widetilde A_-^* \widetilde C_-^* \widetilde C_+
\!!&\!\! H_0 \widetilde B_+ + \widetilde A_-^* \widetilde C_-^*  \widehat D + \epsilon^2 \widetilde A_-^* \widetilde B_-
\!&  X_{14}
 \notag \\
\widetilde A_+^* H_0^* + \widetilde C_+^* \widetilde C_- \widetilde A_-
\!\!&\! \widetilde A_+^* H_+ \widetilde A_+ + \widetilde C_+^* \widetilde C_+  + \epsilon^2 I
\!\!&\!\! \widetilde A_+ H_+ \widetilde B_+ + \widetilde C_+^* \widehat D
\!& X_{24}
\notag \\
\widetilde B_+^* H_0^* +  \widehat D^* \widetilde C_- \widetilde A_-  + \epsilon^2 \widetilde B_-^* \widetilde A_-
\!\!&\! \widetilde B_+^* H_+ \widetilde A_+ + \widehat D^* \widetilde C_+
\!\!&\!\! \widetilde B_+^* H_+ \widetilde B_+ + \widehat D^* \widehat D + \epsilon^2 \widetilde B_-^* \widetilde B_-
                      + \epsilon^2 I_\cU
\!& X_{34}
\notag \\
 \epsilon H_0^* + \epsilon \widetilde C_-^* \widetilde C_-  \widetilde A_- + \epsilon^2 \widetilde A_-
\!\!&\!  \epsilon H_+ \widetilde A_+ + \epsilon \widetilde C_-^* \widetilde C_+
\!\!&\!\!  \epsilon \widetilde C_-^* \widehat D + \epsilon^2 \widetilde B_-
\!&   X_{44}}
\notag \\
&  \quad \quad \quad \quad \preceq
\sbm{ \widetilde A_-^* H_- \widetilde A_- & \widetilde A_-^* H_0 & \widetilde A_-^* H_- \widetilde B_- & \epsilon \widetilde A_-^* H_-
  \\
  H_0^* \widetilde A_- & H_+  & H_0^* \widetilde B_- & \epsilon H_0^*
  \\
  \widetilde B_-^* H_- \widetilde A_- & \widetilde B_-^* H_0 & \widetilde B_-^* H_- \widetilde B_- + I & \epsilon \widetilde B_-^* H_-
  \\
  \epsilon H_- \widetilde A_- & \epsilon H_0 & \epsilon H_- \widetilde B_- & \epsilon^2 H_- + I }
  \label{ep-bicausal-KYP'}
\end{align}
where
$$
\sbm{ X_{14}  \\ X_{24} \\ X_{34} \\ X_{44} }
 = \sbm{ \epsilon H_0 + \epsilon \widetilde A_-^* \widetilde C_-^* \widetilde C_-   + \epsilon^2 \widetilde A_-^*
 \\ \epsilon \widetilde A_+^* H_+ +  \epsilon \widetilde C_+^* \widetilde C_-
 \\ \epsilon \widehat D^* \widetilde C_- + \epsilon^2 \widetilde B_-^*
 \\   \epsilon^2 \widetilde C_-^* \widetilde C_- + \epsilon^2 I_{\cX_-} }.
 $$
 The ($4 \times 4$)-block matrices appearing in \eqref{ep-bicausal-KYP'} are to be understood as operators on
 $\sbm{ \cX_- \\ \cX_+ \\ \cU \\ \cX }$ where the last component $\cX$ further decomposes as $\cX = \sbm{ \cX_- \\ \cX_+ }$.
 Note that the operators in the fourth row a priori are operators with range in $\cX_-$ or $\cX_+$;  to get the proper interpretation
 of these operators as mapping in $\cX$,  one must compost each of these operators on the left by the canonical injection of
 $\cX_\pm$ into $\cX  = \sbm{ \cX_- \\ \cX_+ }$.   Similarly the operators in the fourth columns a priori are defined only on $\cX_-$
 or $\cX_+$; each of these should be composed on the right with the canonical projection of $\cX$ onto $\cX_\pm$.  On the
 other hand the identity operator $I$ appearing in the $(4,4)$-entry of the matrix on the right is the identity on the whole space
 $\cX = \sbm{ \cX_- \\ \cX_+ }$.

With this understanding of the interpretation for the fourth row and fourth column of the matrices in \eqref{ep-bicausal-KYP'}
in place, the next step is to
simply cross out the last row and last column in \eqref{ep-bicausal-KYP'} to arrive at the reduced block-($ 3 \times 3$) inequality
\begin{align*}
& \sbm{  H_- + \widetilde A_-^* \widetilde C_-^* \widetilde C_-  \widetilde A_-
     & H_0 \widetilde A_+ + \widetilde A_-^* \widetilde C_-^* \widetilde C_+
     & H_0 \widetilde B_+ + \widetilde A_-^* \widetilde C_-^*  \widehat D
\\
  \widetilde A_+^* H_0^* + \widetilde C_+^* \widetilde C_- \widetilde A_-
  & \widetilde A_+^* H_+ \widetilde A_+ + \widetilde C_+^* \widetilde C_+
  & \widetilde A_+ H_+ \widetilde B_+ + \widetilde C_+^* \widehat D
\\
  \widetilde B_+^* H_0^* +  \widehat D^* \widetilde C_- \widetilde A_-
  & \widetilde B_+^* H_+ \widetilde A_+ + \widehat D^* \widetilde C_+
  & \widetilde B_+^* H_+ \widetilde B_+ + \widehat D^* \widehat D  }  \\
& \quad \quad \quad \quad
  + \epsilon^2
  \sbm{ \widetilde A_-^* \widetilde A_- & 0 & \widetilde A_-^* \widetilde B_-
  \\
   0 &  I_{\cX_+} & 0
   \\
   \widetilde B_- \widetilde A_- & 0 & \widetilde B_-^* \widetilde B_- + I_\cU}
    \preceq
\sbm{ \widetilde A_-^* H_- \widetilde A_- & \widetilde A_-^* H_0 & \widetilde A_-^* H_- \widetilde B_-
  \\
  H_0^* \widetilde A_- & H_+  & H_0^* \widetilde B_-
  \\
  \widetilde B_-^* H_- \widetilde A_- & \widetilde B_-^* H_0 & \widetilde B_-^* H_- \widetilde B_- + I_\cU }.
\end{align*}
This last inequality amounts to the spelling out of the strict version of  the bicausal KYP-inequality \eqref{KYP-bicausal}, i.e., to
\eqref{KYP-bicausal-strict}.
\end{proof}

%%%%%%%%%%%%%%%%%%%%%%%%%%%%%%%%%%%%
\section{Bounded Real Lemma for nonstationary systems with dichotomy}
\label{S:nonstat}

In this section we show how the main result of Ben
Artzi-Gohberg-Kaashoek in \cite{BAGK95} (see
also \cite[Chapter 3]{HI} for closely related results) follows
from Theorem \ref{T:strictdichot} by the technique of embedding a
nonstationary
discrete-time linear system into an infinite-dimensional stationary
(time-invariant) linear system (see \cite[Chapter X]{OT100}) and
applying the
corresponding result for stationary linear systems.  We note that Ben
Artzi-Gohberg-Kaashoek took the reverse path: they obtain the result
for the stationary case as a corollary of the result for the
non-stationary case.

In this section we replace the stationary linear system
\eqref{dtsystem} with a nonstationary (or time-varying) linear system
of the form
\begin{equation}  \label{nonstat-dtsystem}
\Sigma_{\text{\rm non-stat}} \colon =  \left\{ \begin{array}{rcl}
\bx(n+1) & = & A_{n} \bx(n) +
B_{n} \bu(n), \\
\by(n) & = & C_{n} \bx(n) + D_{n} \bu(n), \end{array} \right. \quad
(n \in
{\mathbb Z})
\end{equation}
where $\{A_{n}\}_{n \in {\mathbb Z}}$ is a bilateral sequence of
 state space operators ($A_{n} \in \cL(\cX)$), $\{B_{n}\}_{n \in
{\mathbb Z}}$
is a bilateral sequence of input operators ($B_{n} \in \cL(\cU,
\cX)$), $\{ C_{n}\}_{n \in {\mathbb Z}}$ is a bilateral sequence of
output operators ($C_{n} \in \cL(\cX, \cY)$), and $\{D_{n}\}_{n \in
{\mathbb Z}}$ is a bilateral sequence of feedthrough operators
($D_{n} \in \cL(\cU, \cY)$).  We assume that all the operator
sequences $\{A_{n}\}_{n \in {\mathbb Z}}$, $\{B_{n}\}_{n \in {\mathbb
Z}}$, $\{C_{n}\}_{n \in {\mathbb Z}}$, $\{D_{n}\}_{n \in {\mathbb Z}}$
are uniformly bounded in operator norm. The system $\Sigma_{\rm
non-stat}$ is
said to have {\em dichotomy} if there is a bounded sequence of
projection operators $\{R_{n}\}_{n \in {\mathbb Z}}$ on $\cX$ such
that
\begin{enumerate}
    \item $\operatorname{Rank} R_{n}$ is constant and the equalities
$$
   A_{n} R_{n} = R_{n+1} A_{n}
$$
hold for all $n \in {\mathbb Z}$,
\item there are constants $a$ and $b$ with $a < 1$ so that
\begin{align}
    & \| A_{n+j-1} \cdots A_{n} x \| \le b a^{j} \| x \| \text{ for all }
    x \in \im R_{n}, \label{forwardstable} \\
    & \| A_{n+j-1} \cdots A_{n} y \| \ge \frac{1}{b a^{j}} \| y \|
    \text{ for all } y \in \kr R_{n}.  \label{backwardstable}
\end{align}
\end{enumerate}

Let us introduce spaces
$$
\newU : = \ell^{2}_{\cU}({\mathbb Z}), \quad
\newX: = \ell^{2}_{\cX}({\mathbb Z}), \quad
\newY : = \ell^{2}_{\cY}({\mathbb Z}).
$$
and define bounded operators $\bA \in \cL(\newX)$, $\bB \in
\cL(\newU, \newX)$, $\bC \in \cL(\newX, \newY)$, $\bD \in \cL(\newU,
\newY)$ by
$$
\bA = \operatorname{diag}_{n \in {\mathbb Z}} [A_{n}], \ \
\bB = \operatorname{diag}_{n \in {\mathbb Z}} [B_{n}], \ \
\bC = \operatorname{diag}_{n \in {\mathbb Z}} [C_{n}], \ \
\bD = \operatorname{diag}_{n \in {\mathbb Z}} [D_{n}].
$$
Define the shift operator $\bS$ on $\newX$  by
$$
\bS = [ \delta_{i,j+1} I_{\cX}]_{i,j \in {\mathbb Z}}
$$
with inverse  $\bS^{-1}$ given by
$$
\bS^{-1} = [\delta_{i+1,j} I_{\cX}]_{i,j \in {\mathbb Z}},
$$
Then the importance of the dichotomy condition is that $\bS^{-1} - \bA$
is invertible on $\ell^{2}_{\cX}({\mathbb Z})$, and conversely,
$\bS^{-1} - \bA$ invertible implies that the system
\eqref{nonstat-dtsystem} has dichotomy (see \cite[Theorem
2.2]{BAGK95}). In this case, given any $\ell^{2}$-sequence $\bx' \in
\newX$, the equation
\begin{equation}   \label{dichot-eq}
\bS^{-1} \bx = \bA \bx + \bx'
\end{equation}
admits a unique solution $\bx = (\bS^{-1} - \bA)^{-1} \bx' \in
\ell^{2}_{\cX}({\mathbb Z})$.  Write out $\bx$ as $\bx = \{ \bx(n)
\}_{n \in {\mathbb Z}}$.  Then the aggregate equation
\eqref{dichot-eq} amounts to the system of equations
\begin{equation}   \label{state-sys}
    \bx(n+1) = A_{n} \bx(n) + \bx'(n).
\end{equation}
In particular we may take $\bx'(n)$ to be of the form $\bx'(n) =
B_{n} \bu(n)$ where $\bu = \{ \bu(n) \}_{n \in {\mathbb Z}} \in
\ell^{2}_{\cU}({\mathbb Z})$.  Then we may uniquely solve for $\bx =
\{ \bx(n) \}_{n \in {\mathbb Z}} \in \newX$ so that
$$
   \bx(n+1) = A_{n} \bx(n) + B \bu(n).
$$
We may then use the output equation in \eqref{nonstat-dtsystem} to
arrive at an output sequence $\by = \{ \by(n) \}_{n \in {\mathbb Z}}
\in \newY$ by
$$
  \by(n) = C_{n} \bx(n)+ D_n \bu(n).
$$
Thus there is a well-defined map $T_{\Sigma}$ which maps the sequence
$\bu = \{\bu(n)\}_{n \in {\mathbb Z}}$ in $\newU$ to
the sequence $\by = \{\by(n) \}_{n \in {\mathbb Z}}$ in
$\newY$.  Roughly speaking, here the initial condition is replaced by boundary
conditions at $\pm \infty$: $R \bx(-\infty) = 0$ and $ (I - R)
\bx(+\infty) = 0$;  we shall use the more precise albeit more implicit
operator-theoretic formulation of the input-output map
(compare also to the discussion around
\eqref{dtsystem-f} and \eqref{dtsystem-b} for this formulation
in the stationary setting):
$T_{\Sigma} \colon \newU \to \newY$ is defined as:
{\em given a $\bu
= \{\bu(n)\}_{n \in {\mathbb Z}} \in \newU$, $T_{\Sigma} \bu$ is the
unique $\by =
\{ \by(n) \}_{n \in {\mathbb Z}} \in  \newY$ for which there is a $\bx = \{
\bx(n)\}_{n \in {\mathbb Z}} \in \newX$ so that the system equations
\eqref{nonstat-dtsystem} hold for all $n \in {\mathbb Z}$}, or in
aggregate operator-form,  by
$$
 T_{\Sigma} = \bD + \bC (\bS^{-1}  - \bA)^{-1} \bB = \bD + \bC (I -
 \bS \bA)^{-1} \bS \bB.
$$
The main theorem from \cite{BAGK95} can be stated as follows.

\begin{theorem} \label{T:tvKYP} (See \cite[Theorem 1.1]{BAGK95}.)
Given a
    nonstationary input-state-out\-put linear system
    \eqref{nonstat-dtsystem}, the following conditions are equivalent.
    \begin{enumerate}
	\item The system \eqref{nonstat-dtsystem} is
	dichotomous and the associated input-output operator
	$T_{\Sigma}$ has operator norm strictly less than 1 ($\|
	T_{\Sigma} \| < 1$).
	\item There exists a  sequence of constant-inertia invertible
selfadjoint
	operators $\{H_{n}\}_{n \in {\mathbb Z}} \in \cL(\cX)$ with
	both $\| H_{n} \|$ and $\| H_{n}^{-1} \|$ uniformly bounded
	in $n \in {\mathbb Z}$ such that
\begin{equation}   \label{tv-KYP}
\begin{bmatrix} A_{n}^{*} & C_{n}^{*} \\ B_{n}^{*} & D_{n}^{*}
\end{bmatrix} \begin{bmatrix} H_{n+1} & 0 \\ 0 &  I_{\cY}
\end{bmatrix}
\begin{bmatrix} A_{n} & B_{n} \\ C_{n} & D_{n} \end{bmatrix}
    \prec \begin{bmatrix} H_{n} & 0 \\ 0 & I_{\cU} \end{bmatrix}
\end{equation}
for all $n \in {\mathbb Z}$.
\end{enumerate}
\end{theorem}

\begin{proof}[\bf Proof]
One can check that the nonstationary dichotomy condition
\eqref{forwardstable}--\eqref{backwardstable} on the operator
sequence $\{A_{n}\}_{n \in {\mathbb Z}}$ translates to the stationary
dichotomy condition on $\bA$ with $\newX_{+} =  \im \bR$, $\newX_{-}
= \kr \bR$ where $\bR = \operatorname{diag}_{n \in {\mathbb Z}} [
R_{n}]$.  Then
$$
 \bU = \begin{bmatrix} \bA & \bB \\ \bC & \bD \end{bmatrix} \colon
\begin{bmatrix} \newX \\ \newU \end{bmatrix} \to \begin{bmatrix}
\newX \\ \newY
\end{bmatrix}
$$
is the system matrix for a big stationary dichotomous linear system
\begin{equation}   \label{statnewsys}
\bSigma: = \left\{ \begin{array}{rcl}
\newx(n+1) & = & \bA \newx(n) + \bB \newu(n) \\
\newy(n+1) & = & \bC \newx(n) + \bD \newu(n) \end{array} \right.
\end{equation}
where
\begin{align*}
& \newu = \{\newu (n) \}_{n \in {\mathbb Z}} \in
\ell^{2}_{\newU}({\mathbb Z}), \quad
\newx = \{\newx (n) \}_{n \in {\mathbb Z}} \in
\ell^{2}_{\newX}({\mathbb Z}), \notag \\
& \newy = \{\newy (n) \}_{n \in {\mathbb Z}} \in
\ell^{2}_{\newY}({\mathbb Z}).
\end{align*}
To apply Theorem \ref{T:strictdichot} to this enlarged stationary
dichotomous system $\bSigma$, we first need to check that the
input-output map $T_{\bSigma}$ is strictly contractive.
Toward this end, for each $k\in\BZ$ introduce a linear operator
$$
 \sigma_{k,\cU} \colon \newU=\ell^{2}_{\cU}({\mathbb Z}) \to
 \ell^{2}_{\newU}({\mathbb Z})
$$
defined by
$$
  \bu = \{ \bu(n)\}_{n \in {\mathbb Z}} \to \sigma_{k} \bu =
  \{\newu^{(k)}(n)\}_{n \in {\mathbb Z}}
 $$
where we set
$$
 \newu^{(k)}(n) = \{ \delta_{m,n+k} \bu(n) \}_{m \in {\mathbb Z}} \in
 \ell^{2}_{\cU}({\mathbb Z}) = \newU.
 $$
In the same way we define $\sigma_{k,\cX}$ and $\sigma_{k,\cY}$, changing only $\cU$ to $\cX$ and $\cY$, respectively, in the definition:
$$
 \sigma_{k,\cX} \colon \newX=\ell^{2}_{\cX}({\mathbb Z}) \mapsto
 \ell^{2}_{\newX}({\mathbb Z}), \quad
 \sigma_{k,\cY} \colon \newY=\ell^{2}_{\cY}({\mathbb Z}) \mapsto
 \ell^{2}_{\newY}({\mathbb Z}).
$$
Then one can check that $\sigma_{k,\cU}$, $\sigma_{k,\cX}$ and $\sigma_{k,\cY}$ are all isometries. Furthermore, the operators
\begin{align*}
\si_{\newU}=\mat{ccccc}{\cdots&\si_{-1,\cU}&\si_{0,\cU}&\si_{1,\cU}&\cdots} : \ell^2_{\newU}(\BZ)\to \ell^2_{\newU}(\BZ), &  \\
\si_{\newX}=\mat{ccccc}{\cdots&\si_{-1,\cX}&\si_{0,\cX}&\si_{1,\cX}&\cdots} : \ell^2_{\newX}(\BZ)\to \ell^2_{\newX}(\BZ), &  \\
\si_{\newY}=\mat{ccccc}{\cdots&\si_{-1,\cY}&\si_{0,\cY}&\si_{1,\cY}&\cdots} : \ell^2_{\newY}(\BZ)\to \ell^2_{\newY}(\BZ), &
\end{align*}
are all unitary. The relationship between the input-output
map $T_{\bSigma}$ for the stationary system $\bSigma$ and the
input-output map $T_{\Sigma}$ for the original nonstationary system
is encoded in the identity
$$
  T_{\bSigma} = \bigoplus_{k=-\infty}^{\infty} \sigma_{k,\cY} \,
  \cS_{\cY}^{*k} \, T_{\Sigma}
  \, \cS_{\cU}^{k} \, \sigma_{k,\cU}^{*}
  =\si_{\newY} \diag_{k\in\BZ}\, [T_\Sigma]\, \si_{\newU}^*
$$
where $\cS_{\cU}$ is the bilateral shift on $\ell^{2}_{\cU}({\mathbb
Z})$ and $\cS_{\cY}$ is the bilateral shift on
$\ell^{2}_{\cY}({\mathbb Z})$,
i.e., the input-output map $T_{\bSigma}$ is unitarily equivalent to the
infinite inflation $T_{\Sigma}
\otimes I_{\ell^{2}({\mathbb Z})}=\diag_{k\in\BZ}[T_\Si]$ of $T_{\Sigma}$.  In particular, it follows that
$\| T_{\bSigma} \| = \| T_{\Sigma}\|$, and hence the hypothesis
that $\| T_{\Sigma} \| < 1$ implies that also $\| T_{\bSigma} \| < 1$.

We may now apply Theorem \ref{T:strictdichot} to conclude that there
is a bounded
invertible selfadjoint operator $\bH$ on $\newX$ such that
\begin{equation}  \label{boldKYP}
    \begin{bmatrix} \bA^{*} & \bC^{*} \\ \bB^{*} & \bD^{*}
    \end{bmatrix} \begin{bmatrix} \bH & 0 \\ 0 & I_{\newY}
\end{bmatrix} \begin{bmatrix} \bA & \bB \\ \bC & \bD \end{bmatrix} -
\begin{bmatrix} \bH & 0 \\ 0 & I_{\newU} \end{bmatrix} \prec 0.
\end{equation}
Conjugate this identity with the isometry  $\sigma_0$:
\begin{align}
& \begin{bmatrix} \sigma_{0}^{*} & 0 \\ 0 & \sigma_{0}^{*} \end{bmatrix}
\begin{bmatrix} \bA^{*} & \bC^{*} \\ \bB^{*} & \bD^{*}
    \end{bmatrix} \begin{bmatrix} \bH & 0 \\ 0 & I_{\newY}
\end{bmatrix} \begin{bmatrix} \bA & \bB \\ \bC & \bD \end{bmatrix}
\begin{bmatrix} \sigma_{0} & 0 \\ 0 & \sigma_{0} \end{bmatrix}  \notag  \\
& \quad \quad \quad \quad
    - \begin{bmatrix} \sigma_{0}^{*} & 0 \\ 0 & \sigma_{0}^{*}
\end{bmatrix}
 \begin{bmatrix} \bH & 0 \\ 0 & I_{\newU} \end{bmatrix}
     \begin{bmatrix} \sigma_{0} & 0 \\ 0 & \sigma_{0} \end{bmatrix}
\prec 0.
\label{reducedKYP}
\end{align}
If we define $H_{n} \in \cL(\cX)$ by
$$
   H_{n} = \iota_{n}^{*} \, \sigma_{0}^{*} \, \bH \, \sigma_{0} \,
\iota_{n}
$$
where $\iota_{n}$ is the embedding of $\cX$ into the $n$-th
coordinate subspace of $\ell^{2}_{\cX}({\mathbb Z})$, then one can
check that the identity
\eqref{boldKYP} collapses to the identity \eqref{tv-KYP} holding for
all $n \in {\mathbb Z}$.  This completes the proof of Theorem
\ref{T:tvKYP}.
\end{proof}

\paragraph{\bf Acknowledgement}
This work is based on the research supported in part by the National
Research Foundation of South Africa. Any opinion, finding and conclusion or
recommendation expressed in this material is that of the authors and
the NRF does not accept any liability in this regard.

\end{document}